\documentclass[10pt]{article}

\usepackage{hyperref}

\usepackage{titlesec}

\usepackage{amsmath}

\usepackage{booktabs}
\usepackage{tabularx}

\usepackage{epsfig, amssymb, amsfonts, dsfont}
\usepackage{amsthm}
\usepackage{amstext}
\usepackage{amsopn}
\usepackage{mathrsfs}
\usepackage{subfigure}
\usepackage{graphicx}
\usepackage[left=2.3cm,top=2cm,right=2.3cm,bottom=2cm, nohead,foot=1cm]{geometry}
\usepackage[makeroom]{cancel}
\usepackage{color}
\usepackage{enumerate}
\usepackage{comment}
\usepackage{stackrel}
\usepackage{stmaryrd}
\usepackage{mathtools}
\usepackage{amsbsy}

\usepackage{tikz}
\usepackage{tikz-cd}

\SetSymbolFont{stmry}{bold}{U}{stmry}{m}{n}

\numberwithin{equation}{section}

\newtheorem{theorem}{Theorem}[section]

\newtheorem{lemma}[theorem]{Lemma}

\newtheorem{corollary}[theorem]{Corollary}
\newtheorem{proposition}[theorem]{Proposition}

\newtheorem{notation}[theorem]{Notation}
\theoremstyle{definition}

\newtheorem{remark}[theorem]{Remark}

\newcommand{\R}{{\mathbb R}}

\newcommand{\Hess}{{\mathbf{H}}}

\DeclareFontFamily{U}{mathx}{\hyphenchar\font45}
\DeclareFontShape{U}{mathx}{m}{n}{
      <5> <6> <7> <8> <9> <10>
      <10.95> <12> <14.4> <17.28> <20.74> <24.88>
      mathx10
      }{}
\DeclareSymbolFont{mathx}{U}{mathx}{m}{n}
\DeclareMathSymbol{\bigtimes}{1}{mathx}{"91}

\usepackage{relsize}

\newcommand{\Tr}{{\textup{Tr}}}

\title{A submartingale for the probability of avoiding the origin in one-point interaction ground-state diffusion: $d \in \{2,3\}$}

\date{  }

 \author{\textbf{Barkat Mian}\footnote{ {\tt bmian@utk.edu}} \vspace{.1cm}  \\  University of Tennessee Knoxville, Department of Mathematics   }

\begin{document}
\maketitle

\begin{abstract}
We study the near-origin behavior on $[0,T]$ of the singular diffusion whose transition density is given by a Doob transform of the integral kernel of the semigroup generated by the \(d\)-dimensional Schr\"odinger operator \(L^\gamma\) with a one-point potential at the origin, where the driving family is the ground state of \(L^\gamma\) and \(d\in\{2,3\}\). We construct a submartingale whose increasing component grows only at times when the diffusion visits the origin. Using this submartingale, we show that the diffusion hits the origin with positive probability and that, conditionally on hitting the origin by time \(T\), the first hitting time has a truncated generalized inverse Gaussian (GIG) distribution. We further study the dynamics under conditioning to avoid the origin: under the conditional law, the diffusion is not a standard Brownian motion, but instead admits a representation in terms of a regularized drift and a continuous martingale. While these properties are known in dimension two, the present submartingale-based approach provides an alternative verification and treats dimensions two and three in a unified manner.
\end{abstract}

\vspace{.2cm}

\section{Introduction}
Fix \(d\in\{2,3\}\). The \(d\)-dimensional heat equation
with a one-point interaction at the origin can be formally represented as
\begin{align}\label{HeatEquationPointInteraction}
\frac{\partial u}{\partial t}(t,x)
=
\frac12\Delta u(t,x)
+
\nu\delta_0(x)u(t,x),
\qquad
t>0,\quad x\in\mathbb R^d,
\end{align}
where \(\delta_0\) denotes the Dirac delta measure concentrated at the
origin, \(\Delta:=\sum_{j=1}^d\frac{\partial^2}{\partial x_j^2}\)
denotes the \(d\)-dimensional Laplace operator, and \(\nu>0\) denotes an infinitesimal coupling constant. The formal operator \(\frac12\Delta+\nu\delta_0(x)\) appearing in~\eqref{HeatEquationPointInteraction} can be rigorously realized as a self-adjoint extension of \(\frac12\Delta |_{C_0^\infty(\mathbb R^d\setminus\{0\})}\) in \(L^2(\mathbb R^d)\), and these self-adjoint extensions form a one-parameter family \(\{L^\alpha\}_{\alpha\in(-\infty,\infty]}\). In this work, we consider \(L^\gamma\) with \(\gamma>0\). The operator \(L^\gamma\) can also be obtained as the norm-resolvent limit of Schr\"odinger operators with suitably rescaled short-range potentials; see, for instance,~\cite[Chaps.~I.1 and I.5]{AGHH}. In particular, for \(d=3\), \(L^\gamma\) can be approximated as \(\varepsilon\downarrow0\) by
\begin{align}\label{RegularizedHamiltonian}
\frac{1}{2}\Delta + \left(\frac{\pi^2}{8} + \gamma \varepsilon\right) \frac{1}{\varepsilon^2} v\!\left(\frac{\cdot}{\varepsilon}\right), \qquad \varepsilon > 0,
\end{align}
where $v$ is a bounded, compactly supported potential satisfying $\int_{\mathbb{R}^3} v(x) \, dx = \frac{4\pi}{3}$; see~\cite[Eq.~(5)]{CKMV2}. For the characterization of the domain of the operator \(L^\gamma\), see~\cite[Thm.~1.1.3]{AGHH} or~\cite[Eq.~(2.1)]{GrummtKolb}. The approximation scheme for \(L^\gamma\) can be found in~\cite[Chap.~I.5]{AGHH} or in the discussion preceding~\cite[Eq.~(1.2)]{CM} for \(d=2\).

The self-adjoint operator \(L^\gamma\) can be roughly understood as \(\frac{1}{2}\Delta\) on \(\R^d\setminus\{0\}\), but the fundamental solution \(p^\gamma\) to the equation \(\partial_tu=L^\gamma u\) on \((0,\infty)\times(\R^d\setminus\{0\})\) is different from the free heat kernel and has been computed in~\cite{Albeverio}. For \(t>0\) and \(x,y\in\R^d\setminus\{0\}\), \(p_t^\gamma(x,y)\) admits the explicit decomposition
\begin{align}\label{FundamentalSolution}
p_t^\gamma(x,y)
=
g_t(x-y)+h_t^\gamma(x,y),
\end{align}
where \(g_t(z):=(2\pi t)^{-d/2}e^{-|z|^2/(2t)}\) denotes the free heat kernel on \(\R^d\), and the kernel \(h_t^\gamma(x,y)\) (see~\eqref{DefPerturbedKernel}) represents the singular perturbation term induced by the point potential at the origin; see~\cite[Eqs.~(3.4) and (3.15)]{Albeverio}. Since \(\int_{\mathbb R^d\setminus\{0\}}p_t^\gamma(x,y)\,dy>1\), the kernel \(p_t^\gamma(x,y)\) does not define a Markov transition density. For a fixed \(T>0\), if \(\{H_t^\gamma\}_{t\in[0,T]}\) is a suitable strictly positive driving family such that the Doob transform of \(p_t^\gamma(x,y)\), given by
\begin{align}\label{DoobTransformedKernel}
q_{s,t}^{T,\gamma}(x,y)
:=
p_{t-s}^\gamma(x,y)
\frac{H_{T-t}^\gamma(y)}{H_{T-s}^\gamma(x)},
\qquad
0\le s<t\le T.
\end{align}
defines a transition probability density, then the corresponding Markov
process \((X_t)_{t\in[0,T]}\), if it can be constructed, describes the
trajectory of a particle subject to the attractive one-point interaction and can be used to understand its interaction with the singular point. Table~\ref{tab:diffusions_comparison} lists four such existing diffusions together with their corresponding driving families.
\begin{table}[htbp] 
\centering
\caption{One-point interaction diffusions and their driving families}
\label{tab:diffusions_comparison}
\begin{tabularx}{\linewidth}{l X X}
\toprule
\textbf{Diffusions} & \textbf{Driving families \((d=2)\)} & \textbf{Driving families \((d=3)\)} \\
\midrule
\textbf{Ground-state diffusion (GSD)}
&
\(\displaystyle
H_t^\gamma(x)
=
e^{\gamma t}
\sqrt{\frac{2\gamma}{\pi}}\,
K_0\!\left(\sqrt{2\gamma}|x|\right)
\)
&
\(\displaystyle
H_t^\gamma(x)
=
e^{\gamma^2t/2}
\sqrt{\frac{\gamma}{2\pi}}\,
\frac{e^{-\gamma|x|}}{|x|}
\)
\\
\addlinespace
\textbf{Total-mass diffusion (TMD)}
&
\(\displaystyle
H_t^\gamma(x)
=
\int_{\mathbb R^2}
p_t^\gamma(x,y)\,dy
\)
&
\(\displaystyle
H_t^\gamma(x)
=
\int_{\mathbb R^3}
p_t^\gamma(x,y)\,dy
\)
\\
\bottomrule
\end{tabularx}
\end{table}

Recall that a standard \(d\)-dimensional Brownian motion started away from the origin almost surely never hits the origin for \(d\ge2\): although a two-dimensional Brownian motion comes arbitrarily close to the origin, it never reaches it, whereas a three-dimensional Brownian motion is transient and ultimately escapes to infinity. This raises the natural question of whether the attractive one-point interaction is sufficiently strong to alter this behavior and force the associated diffusion to visit the singular point with positive probability. While the singular form of their transition densities suggests such behavior, its verification is not immediate from the construction of these diffusions alone. In~\cite{CM}, a one-point interaction diffusion, which we refer to here as the two-dimensional total-mass diffusion (2d TMD), was constructed and studied, and its transition density is of the form~\eqref{DoobTransformedKernel} with \(H_t^\gamma(x)=\int_{\mathbb R^2}p_t^\gamma(x,y)\,dy\). The 2d TMD also arises naturally in connection with the second moment measures associated with the continuum polymer measures corresponding to the critical two-dimensional stochastic heat flow (2d SHF)~\cite{CSZ5}; see~\cite[Sec.~5.3]{CM2}. For recent works on the 2d SHF, see~\cite{LZ,Tsai,Chen3,GuTsai,CT,NakashimaSHF,BCT}. In particular,~\cite[Thm.~2.4]{CM} establishes that the 2d TMD hits the origin with positive probability and further describes the process conditioned to avoid the origin up to a prescribed time horizon. A key tool in this analysis is the submartingale introduced in~\cite[Prop.~3.4]{CM}, whose increasing component grows only when the diffusion visits the origin. The corresponding near-origin analysis, using an analogous submartingale, was extended recently to the case \(d=3\) in~\cite{MianPathwise} for the 3d TMD constructed in~\cite{CKMV2}. See also~\cite{CKMV1} for the related heat equation with a compactly supported potential, and~\cite{FitzsimmonsLi,Nakashima,WangSHE} for other related works in the three-dimensional setting.

The main goal of the present work is to construct the analogous submartingale for the finite-horizon restrictions of the infinite-horizon ground-state diffusions (GSDs) constructed in~\cite{Chen2} for \(d=2\) and in~\cite[Thm.~3.3]{CKMV2} for \(d=3\). The GSDs have transition densities of the form~\eqref{DoobTransformedKernel} corresponding to the ground-state transform \(H_t^\gamma(x)=e^{\lambda_\gamma t}\psi_\gamma(x)\), where \(\lambda_\gamma\) is the simple positive eigenvalue of the operator \(L^\gamma\) with normalized eigenfunction
\begin{align}\label{EigenFunctionEigenValueUnified}
\psi_\gamma(x)
=
\frac{
\sqrt{2\lambda_\gamma}
}{
\pi^{(d-1)/2}|x|^{d/2-1}
}
K_{d/2-1}\!\left(
\sqrt{2\lambda_\gamma}\,|x|
\right),
\qquad x \in \R^d\setminus \{0\},
\end{align}
see also~\eqref{EigenFunctionEigenValue}. We use this submartingale to show that these GSDs have positive probability of hitting the origin on \( [0,T] \), and that, conditionally on \(\{\tau\le T\}\), the first hitting time \(\tau\) follows a truncated generalized inverse Gaussian (GIG) distribution. Moreover, on the event that the diffusions stay away from the origin, unlike the TMDs, the conditional processes are not standard Brownian motion, but instead admit a representation in terms of a regularized drift and a continuous martingale. While these properties are known for the 2d GSD, the present submartingale-based approach provides an alternative verification; moreover, the formulation developed here treats the cases \(d=2\) and \(d=3\) within a unified framework.

In the case \(d=2\),~\cite[Prop.~2.6]{CM} and~\cite[Eq.~(2.15) and Props.~2.6 and~2.8]{Chen2} establish that these singular diffusions satisfy the following SDE:
\begin{align}\label{SDE}
dX_t
=
dW_t
+
\nabla\log H_{T-t}^\gamma(X_t)\,dt,
\end{align}
where \(W\) is a two-dimensional Brownian motion. In contrast, the 3d TMD satisfies the above SDE only locally on intervals during which the process remains away from the origin; see~\cite[Prop.~2.1]{MianPathwise}. For the 3d GSD,~\cite[Remark after Thm.~3.3]{CKMV2} identifies the generator
\[
\mathcal Lf(x)
=
\frac12\Delta f(x)
-
\left(1+\frac1{|x|}\right)
\frac{\partial f}{\partial r}(x),
\qquad
x\neq0,
\]
for suitable functions \(f\). This suggests, in particular, that the 3d GSD satisfies the SDE~\eqref{SDE} away from the origin, since in this case the drift is given by
\(b^\gamma(x)=-(\gamma+|x|^{-1})x/|x|\); see~\eqref{Drfitfunction}. As in the 3d TMD case, this drift may fail to be integrable near the origin. Indeed, for fixed \(t>0\), \(x\in\R^3\setminus\{0\}\), and \(y\to0\), both the point-interaction heat kernel \(p_t^\gamma(x,y)\) and the ground state \(\psi_\gamma(y)\) are of order \(|y|^{-1}\). Consequently,
\(q_t^\gamma(x,y)\asymp |y|^{-2}\) as \(y\to0\) by~\eqref{DefTransDensity}, and
\[
\int_0^T
\int_{\{y\in\mathbb R^3:|y|\le1\}}
q_t^\gamma(x,y)
\left(\gamma+\frac1{|y|}\right)
\,dy\,dt
\]
may fail to be finite. Consequently, standard existence theories for SDEs with singular drifts (e.g.,~\cite{Fukushima,Streit,Trutnau,Krylov,Zhang,Jin,Kinzebulatov})  may not be applied directly in this setting. See Lemma~\ref{lemmaLocalizedSDEAwayOrigin} for a precise statement concerning the SDE representation of the 3d GSD.

\subsection{List of special functions and notation}
Fix \(d\in\{2,3\}\) and \(T,\gamma>0\). For notational convenience, we suppress the dependence on \(d\) in the notation for the special functions appearing in this work. When a definition depends on the dimension, the dimension is specified by the Euclidean space to which its arguments belong; when the same formula applies in both dimensions, no explicit reference to the dimension is made. The key special functions are summarized below.
\begin{enumerate}[(i)]
\item For \(t>0\), the kernel \(h_t^\gamma(x,y)\) appearing in the definition~\eqref{FundamentalSolution} of the one-point interaction kernel \(p_t^\gamma(x,y)\) is given by
\begin{align}\label{DefPerturbedKernel}
h_t^\gamma(x,y)
:=
\begin{cases}
\displaystyle
2\pi\gamma
\int_{0<r<s<t}
g_r(x)
\nu'\bigl(\gamma(s-r)\bigr)
g_{t-s}(y)
\,ds\,dr,
&
x,y\in\mathbb R^2\setminus\{0\}, \\[4mm]
\displaystyle
\frac{t}{|x||y|}
\bar g_t\bigl(|x|+|y|\bigr)
+
\frac{\gamma t}{|x||y|}
\int_0^\infty
e^{\gamma u}
\bar g_t\bigl(u+|x|+|y|\bigr)
\,du,
&
x,y\in\mathbb R^3\setminus\{0\}.
\end{cases}
\end{align}
where, recalling that $g_t(x) \coloneqq (2\pi t)^{-d/2} e^{-x^2/(2t)}$ is the $d$-dimensional free heat kernel, the function $\bar{g}_t(r) \coloneqq (2\pi t)^{-3/2} e^{-r^2/(2t)}$ for $r \ge 0$ denotes the radial version of the three-dimensional free heat kernel, and $\nu'(a)$ is the derivative of the Volterra function given by $\nu(a) = \int_{0}^{\infty} \frac{a^s}{\Gamma(s+1)} \, \mathrm{d}s$. The kernel~\eqref{FundamentalSolution} corresponds to~\cite[Eq.~(2.4)]{CM} with the parameter identification \(\gamma=\lambda\) in the case \(d=2\), whereas in \(d=3\) it corresponds to~\cite[Eq.~(6)]{CKMV2} with the same parameter \(\gamma\). It also appears in several related works; see, for instance,~\cite{Fleischmann,APT,GrummtKolb,Chen1,GQT,CSZ5}.

\item The normalized positive eigenfunction \(\psi_\gamma\) of the one-point interaction self-adjoint operator \(L^\gamma\) and its corresponding ground-state eigenvalue \(\lambda_\gamma\) are given by
\begin{align}\label{EigenFunctionEigenValue}
\psi_\gamma(x) := 
\begin{cases}
\displaystyle
\sqrt{\frac{2\gamma}{\pi}} K_0(\sqrt{2\gamma}|x|), & x \in \mathbb{R}^2 \setminus \{0\}, \\[4mm]
\displaystyle
        \sqrt{\frac{\gamma}{2\pi}} \frac{e^{-\gamma|x|}}{|x|}, & x \in \mathbb{R}^3 \setminus \{0\},
\end{cases}
\quad \quad \text{with} \quad  \lambda_\gamma := 
\begin{cases}
\displaystyle
\gamma, & d = 2, \\[4mm]
\displaystyle
\frac{\gamma^2}{2}, & d = 3,
    \end{cases}
\end{align}
where $K_\nu$ denotes the modified Bessel function of the second kind of order $\nu$; see~\eqref{DefCompleteBessel}. Using the identity $K_{1/2}(r)=\sqrt{\pi/(2r)}\,e^{-r}$ for $r>0$, the eigenfunction $\psi_\gamma$ can be written in the unified form~\eqref{EigenFunctionEigenValueUnified}.

\item We denote the positive ground-state driving family by
\(\Psi_t^\gamma(x):=e^{\lambda_\gamma t}\psi_\gamma(x)\). Taking
\(H_t^\gamma=\Psi_t^\gamma\) in~\eqref{DoobTransformedKernel}, the
corresponding Doob-transformed transition density is time-homogeneous
and takes the form
\begin{align}\label{DefTransDensity}
q_t^\gamma(x,y)
:=
e^{-\lambda_\gamma t}
\frac{\psi_\gamma(y)}{\psi_\gamma(x)}
p_t^\gamma(x,y),
\qquad
t>0,\quad x,y\in\R^d\setminus\{0\}.
\end{align}

\item The time-independent ground-state drift is
\begin{align} \label{Drfitfunction} 
b^\gamma(x)
:=
\nabla\log\psi_\gamma(x)
=
\begin{cases}
\displaystyle
-\sqrt{2\gamma}\,
\frac{
K_1\bigl(\sqrt{2\gamma}\,|x|\bigr)
}{
K_0\bigl(\sqrt{2\gamma}\,|x|\bigr)
}
\frac{x}{|x|},
&
x\in\mathbb R^2\setminus\{0\},
\\[4mm]
\displaystyle
-\left(
\gamma+\frac1{|x|}
\right)
\frac{x}{|x|},
&
x\in\mathbb R^3\setminus\{0\},
\end{cases}
\end{align}
where $\nabla = (\partial_{x_1}, \dots, \partial_{x_d})$ represents the gradient operator and the regularized (time-dependent) drift defined for $t > 0$ by
\begin{align} \label{RegularizedDrfitfunction}
\widehat b_t^\gamma(x)
:=
\nabla\log(g_t*\psi_\gamma)(x)
=
\begin{cases}
\displaystyle
-\sqrt{2\gamma}\,
\frac{
K_1\bigl(\sqrt{2\gamma}\,|x|,\gamma t\bigr)
}{
K_0\bigl(\sqrt{2\gamma}\,|x|,\gamma t\bigr)
}
\frac{x}{|x|},
&
x\in\mathbb R^2\setminus\{0\},
\\[4mm]
\displaystyle
-\gamma\,
\frac{
K_{3/2}\bigl(\gamma|x|,\frac{\gamma^2t}{2}\bigr)
}{
K_{1/2}\bigl(\gamma|x|,\frac{\gamma^2t}{2}\bigr)
}
\frac{x}{|x|},
&
x\in\mathbb R^3\setminus\{0\},
\end{cases}
\end{align}
where \(K_\nu(\cdot,\cdot)\) and \(K_\nu\) denote the complete and incomplete modified Bessel functions of the second kind, respectively, as defined in~\eqref{DefIncompleteBessel} and~\eqref{DefCompleteBessel}.

\item For \(t\ge0\) and \(x \in \R^d \setminus \{0\}\), define
\begin{align} \label{Sfunction}
S_t^\gamma(x)
:=
\frac{(g_t * \Psi_0^\gamma)(x)}{\Psi_t^\gamma(x)}
=
e^{-\lambda_\gamma t}
\frac{(g_t*\psi_\gamma)(x)}
{\psi_\gamma(x)} .
\end{align}
Moreover, we extend \(S_t^\gamma\) to the origin by setting
\(S_t^\gamma(0):=0\) for \(t>0\), and we have
\(S_0^\gamma(x)=1\) for every \(x\neq0\). The function
\(S_t^\gamma\) satisfies the partial differential
equation~\eqref{PDEforS} on \(\mathbb R^d\setminus\{0\}\).
\end{enumerate}

\subsection{Canonical realization of finite-horizon ground-state diffusions in \texorpdfstring{$d\in\{2,3\}$}{Lg} }  \label{SecCanonicalGSD}

Let $d \in \{2,3\}$ and fix $T, \gamma > 0$. Let $\Omega = C([0,T]; \mathbb{R}^d)$ be the canonical path space of continuous functions equipped with the uniform metric and Borel \(\sigma\)-algebra \(\mathcal B(\Omega)\), and let $X = \{X_t\}_{t \in [0,T]}$ denote the canonical evaluation process $X_t(\omega) = \omega(t)$. We equip $\Omega$ with the natural filtration $\{\mathcal{F}_t\}_{t \in [0,T]}$, where $\mathcal{F}_t \coloneqq \sigma(X_s : 0 \le s \le t)$.

The probability measure $\mathbb{P}_x^{T,\gamma}$ appearing in the lemma below is the restriction to the finite time horizon $[0,T]$ of the corresponding $d$-dimensional ground-state diffusion law. In the current canonical setting, we comment on how this restriction is obtained from the corresponding infinite-horizon law in Appendix~\ref{ProofAppLemmaCP}.

\begin{lemma}[\cite{Chen2, CKMV2}]\label{LemmaCP}
Fix $d \in \{2,3\}$, $T, \gamma > 0$, and $x \in \mathbb{R}^d \setminus \{0\}$. There exists a unique probability measure $\mathbb{P}_x^{T,\gamma}$ on $(\Omega, \mathcal{B}(\Omega))$ under which the coordinate process $X = \{X_t\}_{t \in [0,T]}$ satisfies $\mathbb{P}_x^{T,\gamma}(X_0 = x) = 1$ and is a time-homogeneous Markov process with respect to the natural filtration $\{\mathcal{F}_t\}_{t \in [0,T]}$, having transition density $q_t^\gamma(x,y)$ given by~\eqref{DefTransDensity}. 
\end{lemma}

Recall from~\eqref{Drfitfunction} that the (time-independent) drift \(b^\gamma:\mathbb R^d\setminus\{0\}\to\mathbb R^d\) is given by \(b^\gamma(x)=\nabla\log\psi_\gamma(x)\) for \(x\neq0\). For \(x\in\mathbb R^3\setminus\{0\}\), let \(\mathcal N_x\) denote the
collection of all subsets of \(\mathbb P_{x}^{T,\gamma}\)-null sets and define
\begin{align}\label{Augmented}
\mathcal F_t^{T,x}
:=
\sigma\big\{
\mathcal F_t\cup\mathcal N_x
\big\},
\qquad
\mathcal B_x^T
:=
\sigma\big\{
\mathcal B(\Omega)\cup\mathcal N_x
\big\}.
\end{align}
The measure \(\mathbb P_{x}^{T,\gamma}\) extends uniquely to the augmented \(\sigma\)-algebra \(\mathcal B_x^T\), and we use the same symbol for the extension. In dimension \(d=2\), the SDE representation below follows directly from~\cite{Chen2}. For reference purposes, we record the SDE representations in both
dimensions in a single lemma. The proof of the following lemma for $d=3$ is in Appendix~\ref{ProoflemmaLocalizedSDEAwayOrigin}. 
\begin{lemma}\label{lemmaLocalizedSDEAwayOrigin}
Fix \(d\in\{2,3\}\), \(T,\gamma>0\), and
\(x\in\mathbb R^d\setminus\{0\}\). There exists a \(d\)-dimensional
standard Brownian motion
\(\{W_t^{T,\gamma}\}_{t\in[0,T]}\) on the probability space
\((\Omega,\mathcal B_x^T,\mathbb P_x^{T,\gamma})\), with respect to the
filtration \(\{\mathcal F_t^{T,x}\}_{t\in[0,T]}\), such that the coordinate
process \(X=\{X_t\}_{t\in[0,T]}\) satisfies
\begin{align}\label{CanonicalGSDSDE}
dX_t
=
dW_t^{T,\gamma}
+
b^\gamma(X_t)\,dt,
\qquad
X_0=x,
\end{align}
in the following sense:
\begin{enumerate}[(i)]
\item If \(d=2\), then~\eqref{CanonicalGSDSDE} holds on \([0,T]\),
\(\mathbb P_x^{T,\gamma}\)-almost surely~\cite[Prop.~4.2\((4^\circ)\)]{Chen3}.

\item If \(d=3\), then~\eqref{CanonicalGSDSDE} holds locally on every
time interval on which \(X\) remains away from
the origin.
\end{enumerate}
\end{lemma}

\subsubsection{A submartingale and the probability of avoiding the origin} \label{SectionMainResults}

For $t \in [0,T]$, recall the function $S_t^\gamma : \mathbb{R}^d \setminus \{0\} \to [0,1]$ given by~\eqref{Sfunction}, where $d \in \{2,3\}$. Fix
\(x\in\mathbb R^d\setminus\{0\}\). Since
\(\Psi_t^\gamma(x)=e^{\lambda_\gamma t}\psi_\gamma(x)\), we have
\(\frac{\partial}{\partial t}\Psi_t^\gamma(x)
=
\lambda_\gamma\Psi_t^\gamma(x)\) and
\(\frac12\Delta\Psi_t^\gamma(x)
=
\lambda_\gamma\Psi_t^\gamma(x)\), where the second identity uses that
\(\psi_\gamma\) is an eigenfunction of \(L^\gamma\) with eigenvalue
\(\lambda_\gamma\), and \(L^\gamma\) agrees with \(\frac12\Delta\) away
from the origin. Moreover, since the free heat kernel \(g_t\) satisfies
the heat equation and
\((g_t*\Psi_0^\gamma)(x)=\Psi_t^\gamma(x)S_t^\gamma(x)\), we have
\[
\frac{\partial}{\partial t}
\big(g_t*\Psi_0^\gamma\big)(x)
=
\frac12\Delta\big(g_t*\Psi_0^\gamma\big)(x)
=
\frac12\Psi_t^\gamma(x)\Delta S_t^\gamma(x)
+
\nabla S_t^\gamma(x)\cdot\nabla\Psi_t^\gamma(x)
+
\frac12S_t^\gamma(x)\Delta\Psi_t^\gamma(x).
\]
Differentiating~\eqref{Sfunction} with respect to \(t\) and using the
above identities, we obtain
\begin{align}\label{PDEforS}
\frac{\partial}{\partial t}S_t^\gamma(x)
=
\frac12\Delta S_t^\gamma(x)
+
b^\gamma(x)\cdot\nabla S_t^\gamma(x),
\qquad
t>0,\quad x\in\mathbb R^d\setminus\{0\},
\end{align}
where \(b^\gamma=\nabla\log\psi_\gamma\) is the ground-state drift
defined in~\eqref{Drfitfunction}. Moreover, recalling that the regularized
drift is given by
\(\widehat b_t^\gamma=\nabla\log(g_t*\psi_\gamma)\), it follows
from~\eqref{Sfunction} that
\begin{align}\label{GradS}
\nabla S_t^\gamma(x)
=
-S_t^\gamma(x)
\left(
b^\gamma(x)-\widehat b_t^\gamma(x)
\right),
\qquad
t>0,\quad x\in\mathbb R^d\setminus\{0\}.
\end{align}

The proof of the following proposition is in Appendix~\ref{ProofPropSubMart}.
\begin{proposition}\label{PropSubMart}
Fix \(d\in\{2,3\}\), \(T,\gamma>0\), and
\(x\in\mathbb R^d\setminus\{0\}\). The process
\(\mathbf S^{T,\gamma}=\{\mathbf S_t^{T,\gamma}\}_{t\in[0,T]}\), defined by
\[
\mathbf S_t^{T,\gamma}
:=
S_{T-t}^{\gamma}(X_t),
\qquad
t\in[0,T].
\]
is a bounded, continuous
\(\mathbb P_x^{T,\gamma}\)-submartingale with respect to
\(\{\mathcal F_t^{T,x}\}_{t\in[0,T]}\). Moreover, if
\(\mathbf S^{T,\gamma}
=
\mathbf M^{T,\gamma}+\mathbf A^{T,\gamma}\)
denotes its Doob--Meyer decomposition, then the increasing component
\(\mathbf A^{T,\gamma}\) is constant during the excursions of \(X\) away
from the origin, while the martingale component satisfies
\(\mathbf M^{T,\gamma}\in L^2(\mathbb P_x^{T,\gamma})\) and is given by
\begin{align}\label{MartingaleFormula}
\mathbf M_t^{T,\gamma}
=
S_T^\gamma(x)
+
\int_0^t
\nabla S_{T-s}^\gamma(X_s)\cdot dW_s^{T,\gamma},
\qquad
t\in[0,T].
\end{align}
\end{proposition}

\begin{remark}
The submartingale characterization in the above proposition is
analogous to those obtained for the corresponding TMDs
in~\cite[Prop.~3.4]{CM} and~\cite[Prop.~2.2]{MianPathwise}. The main
difference is that, in the present GSD setting, the additional
regularized drift \(\widehat b_t^\gamma\) appears. For the TMDs, the
corresponding regularized drift is zero ($\widehat b_t^\gamma = 0$), so the identity
\(\nabla S_t^\gamma=-S_t^\gamma(b^\gamma-\widehat b_t^\gamma)\) reduces to $\nabla S_t^\gamma = -S_t^\gamma b^\gamma$, which is precisely the integrand appearing
in the martingale component of the corresponding submartingale
decomposition.
\end{remark}

Recall that the first time at which the process
\(X=\{X_t\}_{t\in[0,T]}\) visits the origin is
\[
\tau
:=
\inf\left\{
t\in[0,T]:X_t=0
\right\},
\]
with the convention that \(\inf\varnothing=\infty\).
Also recall the function
\(
S_t^\gamma:\mathbb R^d\setminus\{0\}\to[0,1]
\)
defined in~\eqref{Sfunction}. By~\eqref{EigenFunctionEigenValueUnified}, Lemma~\ref{LemmaGaussianGroundStateConvolution}, and part~(i) of the theorem below, the probability that \(X\) avoids the origin over the finite time interval \([0,T]\) is precisely \(S_T^\gamma(x)\). The proof of the following corollary is given in Appendix~\ref{SubsectionCorGroundStateVisitation}.

\begin{corollary}\label{CorGroundStateVisitation}
Fix \(d\in\{2,3\}\), \(T,\gamma>0\), and \(x\in\mathbb R^d\setminus\{0\}\). The following hold.
\begin{enumerate}[(i)]
\item The probability that the ground-state diffusion does not hit the
origin by time \(T\) is given by
\begin{align}
\mathbb P_x^{T,\gamma}[\tau>T]
=
\frac{
K_{d/2-1}\!\left(
\sqrt{2\lambda_\gamma}\,|x|,
\lambda_\gamma T
\right)
}{
K_{d/2-1}\!\left(
\sqrt{2\lambda_\gamma}\,|x|
\right)
}.\nonumber
\end{align}

\item Conditionally on the event \(\{\tau\le T\}\), the first hitting
time \(\tau\) has density on \((0,T]\) given by
\begin{align}
\frac{
\mathbb P_x^{T,\gamma}
\left[
\tau\in dt
\,\middle|\,
\tau\le T
\right]
}{dt}
=
\frac{\left(
\frac{|x|}{\sqrt{2\lambda_\gamma}}
\right)^{d/2-1}}{
K_{d/2-1}
\left(
\sqrt{2\lambda_\gamma}|x|
\right)
-
K_{d/2-1}
\left(
\sqrt{2\lambda_\gamma}|x|,
\lambda_\gamma T
\right)
}\,
\frac{1}{2t^{d/2}}
e^{-\lambda_\gamma t-\frac{|x|^2}{2t}}
,
\qquad
0<t\le T. \nonumber
\end{align}
\end{enumerate}
\end{corollary}

\begin{remark}
Recall that a generalized inverse Gaussian random variable
\(\xi\sim\operatorname{GIG}(\nu;a,b)\), where
\(\nu\in\mathbb R\) and \(a,b>0\), has density
\[
\frac{\mathbb P[\xi\in dt]}{dt}
=
\frac{1}{2K_\nu(ab)}
\left(\frac{b}{a}\right)^\nu
t^{\nu-1}
e^{-\frac12\left(\frac{a^2}{t}+b^2t\right)},
\qquad
t>0.
\]
The formulas in Corollary~\ref{CorGroundStateVisitation} show that the
first hitting time of the origin has the generalized inverse Gaussian
distribution
\[
\tau
\sim
\operatorname{GIG}\left(
1-\frac d2;
|x|,
\sqrt{2\lambda_\gamma}
\right).
\]
Indeed, part~(i) gives the tail probability of this distribution at \(T\),
while part~(ii) gives its density restricted to \((0,T]\) and normalized
by the probability of \(\{\tau\le T\}\). In particular, when \(d=2\),
\(\lambda_\gamma=\gamma\), and hence
\(\tau\sim\operatorname{GIG}\bigl(0;|x|,\sqrt{2\gamma}\bigr)\).
This two-dimensional hitting-time distribution is already known; see
\cite[Eq.~(7.5)]{Chen2} and~\cite[p.~884]{DonatiMartinYor}. The present
argument provides an alternative derivation through the submartingale
characterization established in Proposition~\ref{PropSubMart}, while
treating the cases \(d=2\) and \(d=3\) within a unified framework.
\end{remark}

\subsubsection{Conditioned dynamics}

The proof of the following lemma is given in Section~\ref{SubsectionLemSubMART}.

\begin{lemma}\label{LemGroundStateVisitation}
Fix \(d\in\{2,3\}\), \(T,\gamma>0\), and \(x\in\mathbb R^d\setminus\{0\}\). The following hold.
\begin{enumerate}[(i)]

\item
Under the conditional law \(\widehat{\mathbb P}_{x}^{T,\gamma} := \mathbb P_x^{T,\gamma} [\,\cdot\,|\,\tau>T]\), the coordinate process
\(X=\{X_t\}_{t\in[0,T]}\) satisfies
\begin{align}
X_t
=
x
+
\int_0^t
\widehat b_{T-s}^{\gamma}(X_s)\,ds
+
\widehat M_t^{T,\gamma},
\qquad
t\in[0,T],\nonumber
\end{align}
where
\(\{\widehat M_t^{T,\gamma}\}_{t\in[0,T]}\) is a continuous
\(d\)-dimensional martingale with respect to
\(\{\mathcal F_t^{T,x}\}_{t\in[0,T]}\).

\item Let \(\rho\) be an
\(\{\mathcal F_t^{T,x}\}_{t\in[0,T]}\)-stopping time taking values in
\([0,T]\) such that \(\mathbb P_x^{T,\gamma}[\rho<\tau]>0\). Then the stopped coordinate process \(\{X_{t\wedge\rho}\}_{t\in[0,T]}\) has the same law under $\widehat{\mathbb P}_{x}^{T,\gamma,\rho} := \mathbb P_x^{T,\gamma}
\bigl[\,\cdot\,\bigm|\rho<\tau\bigr]$ as it does under the path measure
\[
\frac{
\bigl(S_{T-\rho}^{\gamma}(X_\rho)\bigr)^{-1}
}{
\widehat{\mathbb E}_{x}^{T,\gamma}
\left[
\bigl(S_{T-\rho}^{\gamma}(X_\rho)\bigr)^{-1}
\right]
}
\,\widehat{\mathbb P}_{x}^{T,\gamma}.
\]
\end{enumerate}
\end{lemma}

\begin{remark}
Part~(i) of Lemma~\ref{LemGroundStateVisitation} highlights an important distinction between the GSDs and the corresponding TMDs under conditioning to avoid the origin. For the TMDs, conditioning on \(\{\tau>T\}\) removes the drift completely, so that the conditioned coordinate process is a standard \(d\)-dimensional Brownian motion started from \(x\); see~\cite[Thm.~2.4(ii)]{CM} for \(d=2\) and~\cite[Thm.~2.4(ii)]{MianPathwise} for \(d=3\). In contrast, for the GSDs the conditioning does not eliminate the drift entirely. Instead, part~(i) shows that under \(\widehat{\mathbb P}_x^{T,\gamma}\), the coordinate process retains the time-dependent regularized drift
\(\widehat b_{T-t}^{\gamma}=\nabla\log(g_{T-t}*\psi_\gamma)\), whereas the TMD has the constant initial profile \(1\), so that \(\nabla\log(g_{T-t}*1)=0\).
\end{remark}

\subsection{Organization of the paper}

The remainder of the paper is organized as follows.
\begin{itemize}

\item[--]
Section~\ref{SecRegularizedDrift} studies the regularized drift
\(\widehat b_t^\gamma\) and establishes the estimates needed later in the paper, treating the cases \(d=2\) and \(d=3\) separately in Sections~\ref{SecRegularizedDrift2d} and~\ref{SecRegularizedDrift3d}. We also record the relevant Bessel-function identities and estimates, together with the Gaussian convolution formula for the ground-state eigenfunction in Lemma~\ref{LemmaGaussianGroundStateConvolution}.

\item[--]
Section~\ref{SqureIntegrability} establishes the square-integrability
estimates required for the martingale \(\mathbf M^{T,\gamma}\) given by~\eqref{MartingaleFormula}, again
treating the cases \(d=2\) and \(d=3\) separately in Sections~\ref{SqureIntegrability2d} and~\ref{SqureIntegrability3d}.

\item[--]
Section~\ref{SubsectionLemSubMART} proves
Lemma~\ref{LemGroundStateVisitation}, concerning the dynamics of the
GSD conditioned to avoid the origin and the corresponding stopping-time
change of measure.

\item[--]
Appendix~\ref{AppProofsLemmas} contains the proofs of
Lemmas~\ref{LemmaCP}, \ref{lemmaLocalizedSDEAwayOrigin}, and
\ref{LemmaGaussianGroundStateConvolution}.

\item[--]
In Appendix~\ref{ProofPropSubMart}, we prove Proposition~\ref{PropSubMart} concerning the submartingale characterization.

\item[--]
Finally, Section~\ref{SubsectionCorGroundStateVisitation} contains the proof of
Corollary~\ref{CorGroundStateVisitation}.

\end{itemize}

\section{The regularized drift \texorpdfstring{$\widehat b_t^\gamma(x)$}{Lg} }  \label{SecRegularizedDrift}

The proof of the following lemma for \(d=2\) is given at the end of Section~\ref{SecRegularizedDrift2d}, while for \(d=3\) it follows directly from Lemma~\ref{LemmaHatDriftBound}, since in this case
\(\lvert b^\gamma(x)\rvert=\gamma+\lvert x\rvert^{-1}\).

\begin{lemma}\label{CorRegularizedDriftComparison}
Fix \(d\in\{2,3\}\) and \(T,\gamma>0\). There exists a constant
\(C=C(\gamma,T)>0\) such that, for all \(0<t\le T\) and all
\(x\in\mathbb R^d\setminus\{0\}\),
\begin{align} 
\bigl|\widehat b_t^\gamma(x)\bigr|
\le
C\,\bigl|b^\gamma(x)\bigr|. \nonumber
\end{align} 
\end{lemma}

We use the following integral representations. For $\nu \in \R$, $z>0$, and $y\ge0$, the incomplete modified Bessel function of the second kind of order $\nu$ is defined by
\begin{align} \label{DefIncompleteBessel}
K_\nu(z,y)
:=
\frac{1}{2}
\left(\frac{z}{2}\right)^\nu
\int_y^\infty
u^{-\nu-1}
e^{-u-\frac{z^2}{4u}}
\,du,
\end{align}
while the modified Bessel function of the second kind of order $\nu$ is obtained by taking $y=0$, namely,
\begin{align} \label{DefCompleteBessel}
K_\nu(z)
:=
K_\nu(z,0)
=
\frac{1}{2}
\left(\frac{z}{2}\right)^\nu
\int_0^\infty
u^{-\nu-1}
e^{-u-\frac{z^2}{4u}}
\,du.
\end{align}
Differentiating~\eqref{DefIncompleteBessel} with respect to $z$ yields
\begin{align}
\frac{\partial}{\partial z} K_\nu(z,y)
=
\frac{\nu}{z} K_\nu(z,y) - K_{\nu+1}(z,y). \nonumber
\end{align}
In particular, \(\frac{\partial}{\partial z}K_0(z,y)=-K_1(z,y)\),
\(K_\nu'(z)=\frac{\nu}{z}K_\nu(z)-K_{\nu+1}(z)\), and \(K_0'(z)=-K_1(z)\). The functions \(K_\nu(r)\) satisfy the following asymptotic behaviors~\cite[Eqs.~(5.16.4)\textendash(5.16.5), p.~136]{Lebedev}:
\begin{align}
K_\nu(r)
&\sim
\begin{cases}
\displaystyle
\log\frac1r,
&
r\downarrow0,\quad \nu=0,
\\[2mm]
\displaystyle
\frac1r,
&
r\downarrow0,\quad \nu=1,
\end{cases}
\label{AsympBesselOrigin}
\\
K_\nu(r)
&\sim
\sqrt{\frac{\pi}{2r}}\,
e^{-r},
\qquad
r\to\infty,\quad \nu\ge0.
\label{AsympBesselInfinity}
\end{align}
The proof of following lemma is in Appendix~\ref{ProofLemmaGaussianGroundStateConvolution}

\begin{lemma}\label{LemmaGaussianGroundStateConvolution}
For every \(t>0\) and \(x\in\mathbb R^d\setminus\{0\}\), with
\(d\in\{2,3\}\), the Gaussian convolution of the ground-state
eigenfunction is given by
\begin{align}
(g_t*\psi_\gamma)(x)
=
\frac{
\sqrt{2\lambda_\gamma}
}{
\pi^{(d-1)/2}|x|^{d/2-1}
}
e^{\lambda_\gamma t}
K_{d/2-1}\!\left(
\sqrt{2\lambda_\gamma}\,|x|,
\lambda_\gamma t
\right). \nonumber
\end{align}
\end{lemma}

\subsection{The regularized drift: \texorpdfstring{$d=2$}{Lg} }  \label{SecRegularizedDrift2d}

We prove Lemma~\ref{CorRegularizedDriftComparison} for \(d=2\) at the end of this section. Fix \(\gamma>0\). Recall from~\eqref{Drfitfunction} that the
(time-independent) drift \(b^\gamma:\mathbb R^2\setminus\{0\}\to\mathbb R^2\) is given by \(b^\gamma(x)=\nabla\log K_0\!\bigl(\sqrt{2\gamma}\,|x|\bigr)\) for
\(x \in \R^2 \setminus \{0\}\). There exist constants \(c=c(\gamma)>0\) and
\(C=C(\gamma)>0\) such that, for all
\(x\in\mathbb R^2\setminus\{0\}\),
\begin{align}\label{BoundGroundStateDrift2d}
\frac{c}{
|x|\left( 1+\log \dfrac1{|x|} \right)
}
\mathbf 1_{\{|x|\le1/2\}}
+
c\mathbf 1_{\{|x|>1/2\}}
\le
|b^\gamma(x)|
\le
\frac{C}{
|x|\left( 1+\log\dfrac1{|x|} \right)
}
\mathbf 1_{\{|x|\le1/2\}}
+
C\mathbf 1_{\{|x|>1/2\}} .
\end{align}
Indeed, for \(\beta=2\gamma\), the drift \(\mathbf b_\beta\) defined
in~\cite[Eq.~(1.15)]{Chen2} satisfies
\(b^\gamma=\frac12\mathbf b_{2\gamma}\) on
\(\mathbb R^2\setminus\{0\}\), so the estimate on
\(0<|x|\le1/2\) follows from~\cite[Prop.~2.13\((1^\circ)\)]{Chen2}
and the comparability of \(\log |x|^{-1}\) and
\(1+\log |x|^{-1}\) on this region. For \(|x|>1/2\), the estimate follows from the positivity and
continuity of \(K_1/K_0\) on \((0,\infty)\), together with
\(K_1(r)/K_0(r)\to1\) as \(r\to\infty\), which follows
from~\eqref{AsympBesselInfinity}.

For each \(t>0\), recall from~\eqref{RegularizedDrfitfunction} that
\(\widehat b_t^\gamma:\mathbb R^2\setminus\{0\}\to\mathbb R^2\) is defined by \(\widehat b_t^\gamma:=\nabla\log(g_t*\psi_\gamma)\). The following lemma gives a bound on the regularized drift \(\widehat b_t^\gamma\).

\begin{lemma}\label{LemmaRegularizedDriftBound2d}
Let \(T, \gamma>0\). There exists a constant
\(C=C(\gamma,T)>0\) such that, for all \(0<t\le T\) and all
\(x\in\mathbb R^2\setminus\{0\}\),
\begin{align}
\bigl|\widehat b_t^\gamma(x)\bigr|
\le{}&
\frac{
C\,|x|
}{
t\left(
1+\log^+\dfrac1{\gamma t}
\right)
}
\mathbf 1_{\left\{
0<|x|<
\min\left\{
\sqrt{2t},\frac12
\right\}
\right\}}
+
\frac{
C
}{
|x|
\left(
1+\log\dfrac1{|x|}
\right)
}
\mathbf 1_{\left\{
\sqrt{2t}\le |x|\le\frac12
\right\}}
+
C\mathbf 1_{\left\{
|x|>\frac12
\right\}}. \nonumber
\end{align}
\end{lemma}

\begin{proof}
Let \(x\in\mathbb R^2\setminus\{0\}\). Using the definition~\eqref{DefIncompleteBessel} of the incomplete modified Bessel function $K_\nu(z,y)$ in~\eqref{RegularizedDrfitfunction}, and applying the change of variables \(u=\gamma s\), we obtain
\begin{align}\label{RegularizedDriftIntegralRatio}
\bigl|\widehat b_t^\gamma(x)\bigr|
=
\sqrt{2\gamma}\,
\frac{
\displaystyle
\big(
\sqrt{2\gamma}\,|x|
\big)
\int_{\gamma t}^\infty
u^{-2}
e^{-u-\frac{\gamma|x|^2}{2u}}
\,du
}{
\displaystyle
\frac12
\int_{\gamma t}^\infty
u^{-1}
e^{-u-\frac{\gamma|x|^2}{2u}}
\,du
}
=
|x|\,
\frac{
\displaystyle
\int_t^\infty
s^{-2}
e^{-\gamma s-\frac{|x|^2}{2s}}
\,ds
}{
\displaystyle
\int_t^\infty
s^{-1}
e^{-\gamma s-\frac{|x|^2}{2s}}
\,ds
}.
\end{align}
We estimate the final quotient in~\eqref{RegularizedDriftIntegralRatio} separately in the following three regions. \vspace{.2cm}

\noindent\textit{Case 1: $0<|x|<\min\{\sqrt{2t},1/2\}$.} For every \(s\ge t\), we have $|x|^2/(2s) \le |x|^2/(2t) <1$, and hence
\begin{align}\label{SmallRegionDenominator}
\int_t^\infty
\frac1s
e^{-\gamma s-\frac{|x|^2}{2s}}\,ds
&\ge
e^{-1}
\int_t^\infty
\frac{e^{-\gamma s}}s\,ds
\ge
C(\gamma,T)
\left(
1+\log^+\frac1{\gamma t}
\right),
\end{align}
for some constant \(C=C(\gamma,T)>0\), where the last inequality
follows by considering separately the cases
\(0<\gamma t\le\frac12\) and \(\gamma t>\frac12\). In the first case,
\(t<1/\gamma\), and the desired estimate is obtained by using
\(\int_t^\infty \ge \int_t^{1/\gamma}+\int_{1/\gamma}^{2/\gamma}\),
together with the estimates
\[
\int_t^{1/\gamma}
\frac{e^{-\gamma s}}s\,ds
\ge
e^{-1}
\int_t^{1/\gamma}\frac{ds}{s}
=
e^{-1}
\log\frac1{\gamma t},
\qquad
\int_{1/\gamma}^{2/\gamma}
\frac{e^{-\gamma s}}{s}\,ds
=
\int_1^2\frac{e^{-u}}u\,du
\ge
\frac12\int_1^2e^{-u}\,du
>0,
\]
and \(\log\frac1{\gamma t}=\log^+\frac1{\gamma t}\) since
\(\gamma t\le\frac12<1\). In the second case, \(\frac12<\gamma t\le\gamma T\), and the change of
variables \(u=\gamma s\) gives
\begin{align*}
\int_t^\infty
\frac{e^{-\gamma s}}s\,ds
&\ge
\int_T^\infty
\frac{e^{-\gamma s}}s\,ds
=
\int_{\gamma T}^\infty
\frac{e^{-u}}u\,du
=
E(\gamma T)
\ge
\frac{E(\gamma T)}{1+\log2}
\left(
1+\log^+\frac1{\gamma t}
\right),
\end{align*}
where \(E(r):=\int_r^\infty e^{-u}u^{-1}\,du\), \(r>0\), denotes the exponential integral. Since \(\gamma T>0\) is fixed, \(E(\gamma T)\) is a finite strictly positive constant depending only on \(\gamma\) and \(T\). The final inequality follows from \(\log^+\frac1{\gamma t}\le\log2\), since \(\gamma t>\frac12\).

For the numerator in~\eqref{RegularizedDriftIntegralRatio}, we have
\[
\int_t^\infty
\frac1{s^2}
e^{-\gamma s-\frac{|x|^2}{2s}}\,ds
\le
\int_t^\infty\frac{ds}{s^2}
=
\frac1t.
\]
Substituting the above estimate together with~\eqref{SmallRegionDenominator} into~\eqref{RegularizedDriftIntegralRatio} yields the desired bound in the case \(0<|x|<\min\{\sqrt{2t},1/2\}\). \vspace{.2cm}

\noindent\textit{Case 2:
\(\sqrt{2t}\le |x|\le 1/2\).} In this case, we have \(t\le |x|^2\) and \(|x|^2<1\), so that
\([|x|^2,1]\subset[t,\infty)\). Moreover, for every \(s\in[|x|^2,1]\), we have \(|x|^2/(2s)\le1\) and \(\gamma s\le\gamma\). Therefore,
\begin{align}
\int_t^\infty
\frac1s
e^{-\gamma s-\frac{|x|^2}{2s}}
\,ds
\ge
\int_{|x|^2}^1
\frac1s
e^{-\gamma s-\frac{|x|^2}{2s}}
\,ds
\ge
e^{-(\gamma+1)}
\int_{|x|^2}^1\frac{ds}{s}
=
2e^{-(\gamma+1)}
\log\frac1{|x|}
\ge
C(\gamma)
\left(
1+\log\frac1{|x|}
\right), \nonumber
\end{align}
where $C(\gamma) := 2[1+\log 2]^{-1} \,e^{-(\gamma+1)}\log 2 >0$
and the last inequality uses that \(\log\frac1{|x|}\ge\log2\) since \(0<|x|\le1/2\).

For the numerator, using \(e^{-\gamma s}\le1\) and the change of variables \(u=|x|^2/(2s)\), we obtain
\begin{align*}
\int_t^\infty
\frac1{s^2}
e^{-\gamma s-\frac{|x|^2}{2s}}
\,ds
&\le
\int_0^\infty
\frac1{s^2}
e^{-\frac{|x|^2}{2s}}
\,ds
=
\frac2{|x|^2}
\int_0^\infty e^{-u}\,du
=
\frac2{|x|^2}.
\end{align*}
Substituting the above estimates into~\eqref{RegularizedDriftIntegralRatio} yields the desired bound for \(\sqrt{2t}\le|x|\le1/2\). \vspace{.2cm}

\noindent\textit{Case 3: \(|x|>1/2\).}
We first consider the subregion \(1/2<|x|<\sqrt{2t}\). Since the proof of Case~1 uses only the condition \(|x|<\sqrt{2t}\), the same argument applies in the present subregion and yields the first inequality below.
\[
\bigl|\widehat b_t^\gamma(x)\bigr|
\le
\frac{
C(\gamma,T)|x|
}{
t\left(
1+\log^+\frac1{\gamma t}
\right)
}
\le
\frac{C(\gamma,T)|x|}{t}
\le
8 C(\gamma,T) \sqrt{2T}
\]
The second inequality follows from the elementary estimate
\(1+L\ge1\) for \(L\ge0\), whereas the last inequality uses that
\(t > |x|^2/2 > 1/8\) and \(|x|<\sqrt{2t} \le \sqrt{2T}\).

It remains to consider the region \(|x|>1/2\) and \(t\le |x|^2/2\). Set
\[
D_t(x)
:=
\int_t^\infty
\frac1s
e^{-\gamma s-\frac{|x|^2}{2s}}
\,ds,
\qquad
N_t(x)
:=
|x|
\int_t^\infty
\frac1{s^2}
e^{-\gamma s-\frac{|x|^2}{2s}}
\,ds.
\]
We claim that there exists a constant \(C=C(\gamma,T)>0\) such that
\(D_t(x)\ge C D_0(x)\) throughout this region. Observe that
\[
D_0(x)-D_t(x)
=
\int_0^t
\frac1s
e^{-\gamma s-\frac{|x|^2}{2s}}
\,ds
\le
\int_0^T
\frac1s
e^{-\frac{|x|^2}{2s}}
\,ds
=
\int_{|x|^2/(2T)}^\infty
\frac{e^{-u}}u
\,du
\le
\frac{2T}{|x|^2}
e^{-\frac{|x|^2}{2T}},
\]
where the second equality uses the change of variables
\(u=|x|^2/(2s)\). On the other hand, using~\eqref{AsympBesselInfinity} for $\nu=0$, there exist constants
\(C_3=C_3(\gamma)>0\) and
\(R_0=R_0(\gamma,T)>\max\{1/2,\sqrt{2T}\}\) such that
\[
D_0(x)
=
2K_0\left(\sqrt{2\gamma}|x|\right)
\ge
C_3|x|^{-1/2}e^{-\sqrt{2\gamma}|x|},
\qquad
|x|\ge R_0.
\]
Consequently,
\[
\frac{D_0(x)-D_t(x)}{D_0(x)}
\le
\frac{2T}{C_3}
|x|^{-3/2}
e^{-\frac{|x|^2}{2T}+\sqrt{2\gamma}|x|},
\qquad
|x|\ge R_0.
\]
Since
\(
|x|^{-3/2}
e^{-\frac{|x|^2}{2T}+\sqrt{2\gamma}|x|}
\to0
\)
as \(|x|\to\infty\), by increasing \(R_0\) if necessary, we obtain
\(
\frac{D_0(x)-D_t(x)}{D_0(x)}
\le\frac12
\)
for all \(|x|\ge R_0\). It follows that
\[
D_t(x)
\ge
\frac12D_0(x),
\qquad
|x|\ge R_0.
\]

It remains to consider \(1/2<|x|\le R_0\) and
\(0<t\le\min\{T,|x|^2/2\}\). Since \(D_t(x)\) is decreasing in \(t\), we have
\(D_t(x)\ge D_T(x)\), which implies the first inequality below:
\[
D_t(x)
\ge
D_0(x)\,\frac{D_T(x)}{D_0(x)}
\ge
D_0(x)\,
\inf_{1/2\le|x|\le R_0}
\frac{D_T(x)}{D_0(x)}
=
C_4(\gamma, T) \, D_0(x),
\]
where the second inequality uses that \(D_T(x)/D_0(x)\) is positive and continuous on the compact set \(\{x:1/2\le|x|\le R_0\}\), and therefore the infimum $C_4(\gamma, T) := \inf_{1/2\le|x|\le R_0} \frac{D_T(x)}{D_0(x)}$ is a positive finite constant depending only on $\gamma$ and $T$. Defining $C_5(\gamma, T) := \min\{1/2, C_4(\gamma, T)\}>0$, we conclude that \(D_t(x)\ge C_5(\gamma, T) D_0(x)\) for \(|x|>1/2\) and \(t\le |x|^2/2\). Consequently,
\[
\bigl|\widehat b_t^\gamma(x)\bigr|
\overset{\eqref{RegularizedDriftIntegralRatio}}{=}
\frac{N_t(x)}{D_t(x)}
\le
\frac{1}{C_5(\gamma, T)} \, \frac{N_0(x)}{D_0(x)}
\overset{\eqref{Drfitfunction}}{=}
\frac{1}{C_5(\gamma, T)} \, |b^\gamma(x)|
\le
C(\gamma, T),
\]
where the first inequality also uses that \(N_t(x)\le N_0(x)\), and the last inequality uses~\eqref{BoundGroundStateDrift2d} in the case \(|x|>\frac12\). This establishes the desired bound in Case~3, and completes the proof of the lemma.
\end{proof}

\begin{proof}[Proof of Lemma~\ref{CorRegularizedDriftComparison} for $d=2$]
Recall the estimates~\eqref{BoundGroundStateDrift2d} for
\(\bigl|b^\gamma(x)\bigr|\). Therefore, the desired comparison follows
immediately from Lemma~\ref{LemmaRegularizedDriftBound2d} when
\(\sqrt{2t}\le |x|\le 1/2\) or \(|x|>1/2\). It remains to consider
\(0<|x|<\min\{\sqrt{2t},1/2\}\), which we divide into the two cases
\(0<|x|^2\le t\) and \(t<|x|^2<2t\).

When \(0<|x|^2\le t\) and \(0<|x|<1/2\), we have
\[
\log\frac1{|x|}
=
\frac12\log\frac1t
+
\frac12\log\frac{t}{|x|^2}
\le
\frac12\log^+\frac1t
+
\frac12\log\frac{t}{|x|^2},
\]
where the inequality uses
\(\log t^{-1}\le\log^+t^{-1}\). Hence, we have the first inequality below:
\begin{align*}
\frac{|x|^2}{t}
\left(
1+\log\frac1{|x|}
\right)
&\le
\frac{|x|^2}{t}
\left(
1+\frac12\log^+\frac1t
\right)
+
\frac12
\frac{|x|^2}{t}
\log\frac{t}{|x|^2}
\\
&\le
1+\frac12\log^+\frac1t
+
\frac12
\sup_{0<u\le1}
u\log\frac1u
\\
&\le
C\left(
1+\log^+\frac1t
\right)
\le
C(\gamma,T)
\left(
1+\log^+\frac1{\gamma t}
\right),
\end{align*}
where in the second inequality we used
\(u=|x|^2/t\in(0,1]\) and the fact that the function
\(u\mapsto u\log(1/u)\) is continuous on \((0,1]\), tends to \(0\) as
\(u\downarrow0\), and hence extends continuously to \([0,1]\). The last inequality follows from the fixed value of \(\gamma>0\).

Also, if \(t<|x|^2<2t\) and \(0<|x|<1/2\), then we obtain the same estimate as follows
\[
\frac{|x|^2}{t}
\left(
1+\log\frac1{|x|}
\right)
\le
2
\left(
1+\log^+\frac1{\sqrt t}
\right)
\le
C(\gamma,T)
\left(
1+\log^+\frac1{\gamma t}
\right).
\]
Combining the above observations with
Lemma~\ref{LemmaRegularizedDriftBound2d} in the case
\(0<|x|<\min\{\sqrt{2t},1/2\}\), we obtain the first inequality below:
\[
\bigl|\widehat b_t^\gamma(x)\bigr|
\le
\frac{C}{
|x|
\left(
1+\log\frac1{|x|}
\right)
}
\le
C\,|b^\gamma(x)|,
\]
where the second inequality follows from the lower bound in~\eqref{BoundGroundStateDrift2d}.
\end{proof}

\subsection{The regularized drift: \texorpdfstring{$d=3$}{Lg} }  \label{SecRegularizedDrift3d}

Since \(|b^\gamma(x)|=\gamma+|x|^{-1}\) for \(x\in\mathbb R^3 \setminus \{0\}\) by~\eqref{Drfitfunction}, the following lemma yields Lemma~\ref{CorRegularizedDriftComparison} for \(d=3\).

\begin{lemma}\label{LemmaHatDriftBound}
Let \(T, \gamma>0\). There exists a constant \(C=C(\gamma,T)>0\) such that, for all
\(0<t\le T\) and all \(x\in\mathbb R^3\setminus\{0\}\),
\[
|\widehat b_t^\gamma(x)|
\le
C\left(
\gamma
+
\frac1{\sqrt t}\mathbf 1_{\{|x|\le\sqrt t\}}
+
\frac1{|x|}\mathbf 1_{\{|x|>\sqrt t\}}
\right).
\]
\end{lemma}

\begin{proof}
Using \(\nabla_xg_t(x-y)=-\nabla_yg_t(x-y)\) and integrating by parts in the variable \(y\), we obtain
\[
\nabla(g_t*\psi_\gamma)(x)
=
\int_{\mathbb R^3}
g_t(x-y)\nabla\psi_\gamma(y)\,dy 
=
-\int_{\mathbb R^3}
g_t(x-y)\left(\gamma+\frac1{|y|}\right)
\frac{y}{|y|}\psi_\gamma(y)dy,
\]
where the second equality uses~\eqref{EigenFunctionEigenValue}. Since \(\widehat b_t^\gamma=\nabla\log(g_t*\psi_\gamma)\) on $\R^3 \setminus \{0\}$ by~\eqref{RegularizedDrfitfunction}, we have
\begin{align} \label{RegularizedDriftInequality}
|\widehat b_t^\gamma(x)|
\le
\gamma
+
\frac{
\int_{\mathbb R^3}
g_t(x-y)\frac{1}{|y|}\psi_\gamma(y)\,dy
}{
\int_{\mathbb R^3}
g_t(x-y)\psi_\gamma(y)\,dy
}.
\end{align}
It remains to estimate the last ratio. We first consider the case
\(|x|\le \sqrt t\). Since \(|x|\le \sqrt t\) and \(0<t\le T\), we have
\[
|x-y|^2\le 4t, 
\quad \textup{and} \quad 
\gamma |y| \le \gamma \sqrt{T}
\qquad \text{for } |y|\le \sqrt t,
\]
and therefore using the definitions of $g_t(x-y)$ and \(\psi_\gamma(y)\), we obtain
\[
\int_{\mathbb R^3}
g_t(x-y)\psi_\gamma(y)\,dy
\ge
\frac{C}{t^{3/2}}
\int_{|y|\le \sqrt t}
\frac{1}{|y|}\,dy
=
\frac{C}{t^{3/2}}
\int_0^{\sqrt t} r\,dr
=
\frac{C}{\sqrt{t}}.
\]
for some constant $C=C(\gamma,T)>0$. On the other hand, since \(|x|\le \sqrt t\), we have $|x-y|^2\ge \frac12 |y|^2-|x|^2 \ge \frac12 |y|^2-t$. Therefore,
\[
\int_{\mathbb R^3}
g_t(x-y)\frac{\psi_\gamma(y)}{|y|}\,dy
\le
\frac{C(\gamma)}{t^{3/2}}
\int_{\mathbb R^3}
e^{-\frac{|y|^2}{4t}}
\frac{1}{|y|^2}\,dy
=
\frac{C(\gamma)}{t^{3/2}}
\int_0^\infty e^{-\frac{r^2}{4t}}\,dr
=
\frac{C(\gamma)}{t}
\int_0^\infty e^{-u^2}\,du
\le
\frac{C(\gamma)}{t},
\]
where the last equality uses the change of variables $r=2\sqrt{t}u$ and the final inequality uses that $\int_0^\infty e^{-u^2}\,du=\sqrt{\pi}/2$. Substituting these observations into~\eqref{RegularizedDriftInequality}, we obtain, for
\(|x|\le \sqrt t\),
\[
|\widehat b_t^\gamma(x)|
\le
\gamma
+
\frac{C(\gamma,T)}{\sqrt t}
\le
\gamma
+
\frac{C(\gamma,T)}{|x|}.
\]
It remains to consider the case \(|x|>\sqrt t\). We split the numerator in~\eqref{RegularizedDriftInequality} as
\begin{align}\label{IntNumerator}
\int_{\mathbb R^3}
g_t(x-y)\frac{\psi_\gamma(y)}{|y|}\,dy
&=
\int_{\{|y|\ge |x|/2\}}
g_t(x-y)\frac{\psi_\gamma(y)}{|y|}\,dy
+
\int_{\{|y|< |x|/2\}}
g_t(x-y)\frac{\psi_\gamma(y)}{|y|}\,dy \nonumber \\
&\le
\frac{2}{|x|}
\int_{\{|y|\ge |x|/2\}}
g_t(x-y)\psi_\gamma(y)\,dy 
+ 
\frac{C(\gamma)}{t^{3/2}}e^{-\frac{|x|^2}{8t}}
\int_{\{|y|<|x|/2\}}
\frac{1}{|y|^2}\,dy \nonumber \\
&\le
\frac{2}{|x|}
\int_{\mathbb R^3}
g_t(x-y)\psi_\gamma(y)\,dy 
+ 
C(\gamma)\frac{|x|}{t^{3/2}}e^{-\frac{|x|^2}{8t}},
\end{align}
where the first inequality holds since on the region \(\{|y|\ge |x|/2\}\), we have \(|y|^{-1}\le 2|x|^{-1}\), and if \(|y|<|x|/2\), then \(|x-y|\ge |x|/2\). Observe that,
\[
\frac{|x|}{t^{3/2}}e^{-\frac{|x|^2}{8t}}
=
\frac{e^{-\frac32\gamma |x|}}{|x|^2}
\bigg[\left(\frac{|x|}{\sqrt t}\right)^3
e^{-\frac{|x|^2}{8t}+\frac32\gamma |x|} \bigg]
\le
C(\gamma,T)\frac{e^{-\frac32\gamma |x|}}{|x|^2},
\]
where the inequality follows by setting \(u=|x|/\sqrt t>1\), using \(t\le T\), and observing that the function $u\longmapsto u^3 \exp \big( -\frac{u^2}{8} + \frac32\gamma\sqrt T\,u \big)$ is bounded on \([1,\infty)\). On the other hand, $|y| \le |x|+|x-y| \le \frac{3}{2}|x|$ for $|x|>\sqrt t$ and $|y-x|\le \sqrt t/2$. Therefore,
\begin{align*}
\int_{\mathbb R^3}g_t(x-y)\psi_\gamma(y)\,dy
\ge
\int_{\{|y-x|\le \sqrt t/2\}}
g_t(x-y)\psi_\gamma(y)\,dy \ge
\frac{C(\gamma)}{t^{3/2}}
\frac{e^{-\frac32\gamma |x|}}{|x|}
\int_{\{|y-x|\le \sqrt t/2\}} \,dy 
\ge
C(\gamma,T)\,\frac{e^{-\frac32\gamma |x|}}{|x|},
\end{align*}
where the final inequality uses that the volume of the ball $\{|y-x|\le \sqrt t/2\}$ is $Ct^{3/2}$. Substituting the above observations into~\eqref{IntNumerator}, we obtain
\[
\int_{\mathbb R^3}
g_t(x-y)\frac{\psi_\gamma(y)}{|y|}\,dy
\le
\frac{C(\gamma,T)}{|x|}
\int_{\mathbb R^3}
g_t(x-y)\psi_\gamma(y)\,dy .
\]
Therefore,~\eqref{RegularizedDriftInequality} implies
\[
|\widehat b_t^\gamma(x)|
\le
\gamma+\frac{C(\gamma,T)}{|x|}
\le
C(\gamma,T)\left(\gamma+\frac1{|x|}\right),
\]
for \(|x|>\sqrt t\), which completes the proof.
\end{proof}

\section{Square integrability of the martingale \texorpdfstring{$\mathbf M^{T,\gamma}$}{Lg} }  \label{SqureIntegrability}

In this section, we establish the integrability estimates needed for the martingale \(\mathbf M^{T,\gamma}\) defined in~\eqref{MartingaleFormula} and for the time-integrated regularized drift \(\int_0^t \widehat b_{T-s}^{\gamma}(X_s)\,ds\) appearing in Lemma~\ref{LemGroundStateVisitation}. The main estimates are collected in the following lemma.

\begin{lemma} \label{LemSquareIntegrability}
Fix \(d\in\{2,3\}\), \(T,\gamma>0\), and
\(x\in\mathbb R^d\setminus\{0\}\). Then
\[
\int_0^T
\int_{\mathbb R^d}
q_s^\gamma(x,y)
\,
\bigl|
 S_{T-s}^{\gamma}(y)
b^{\gamma}(y)
\bigr|^2
\,dy\,ds
<\infty .
\]
Moreover, the same estimate holds with
$S_{T-s}^{\gamma}(y)b^{\gamma}(y)$ replaced by either
$S_{T-s}^{\gamma}(y)\widehat b_{T-s}^{\gamma}(y)$ or
$\widehat b_{T-s}^{\gamma}(y)$.
\end{lemma}

\subsection{Square integrability of the martingale: \texorpdfstring{$d=2$}{Lg} } \label{SqureIntegrability2d}
Recall from~\eqref{EigenFunctionEigenValue} that the normalized positive eigenfunction \(\psi_\gamma:\mathbb R^2\setminus\{0\}\to(0,\infty)\) is given by
\[
\psi_\gamma(x)
=
\sqrt{\frac{2\gamma}{\pi}}\,
K_0\!\left(\sqrt{2\gamma}\,|x|\right)
\]
Thus, on \(0<|x|\le1/2\), the function \(\psi_\gamma(x)\) is comparable to \(1+\log |x|^{-1}\) by the case \(\nu=0\) in~\eqref{AsympBesselOrigin}, together with the positivity and continuity of \(K_0\) on compact subintervals of \((0,\infty)\).
Likewise,~\eqref{AsympBesselInfinity} shows that, for large \(|x|\),
the function \(\psi_\gamma(x)\) behaves like $|x|^{-1/2} e^{-\sqrt{2\gamma}|x|}$ and the positivity and continuity of \(K_0\) extend this comparison to the entire region \(|x|>1/2\). Consequently, there exists a constant \(C=C(\gamma)>1\) such that, for all \(x\in\mathbb R^2\setminus\{0\}\),
\begin{align}
\psi_\gamma(x)
&\ge
\frac1C
\left[
\left(
1+\log \dfrac1{|x|}
\right)
\mathbf 1_{\{|x|\le1/2\}}
+
\frac{
e^{-\sqrt{2\gamma}|x|}
}{
\sqrt{|x|}
}
\mathbf 1_{\{|x|>1/2\}}
\right], \label{LowerBoundGroundState} \\
\psi_\gamma(x)
&\le
C
\left[
\left(
1+\log\dfrac1{|x|}
\right)
\mathbf 1_{\{|x|\le1/2\}}
+
\frac{
e^{-\sqrt{2\gamma}|x|}
}{
\sqrt{|x|}
}
\mathbf 1_{\{|x|>1/2\}}
\right]. \label{UpperBoundGroundState}
\end{align}

\begin{lemma}
Let \(T,\gamma>0\). There exists a constant
\(C=C(\gamma,T)>0\) such that, for all \(0<t\le T\) and all
\(y\in\mathbb R^2\setminus\{0\}\),
\begin{align}\label{CombinedGroundStateBound}
\frac{
\bigl|(g_t*\psi_\gamma)(y)b^\gamma(y)\bigr|^2
}{
\psi_\gamma(y)
}
\le
C\Bigg(
\frac{
\left(1+\log^+\frac1{\gamma t}\right)^2
}{
|y|^2
\left(1+\log^+\frac1{|y|}\right)^3
}
\mathbf 1_{\{|y|^2<2t\}}
+
\frac{
1
}{
|y|^2
\left(1+\log^+\frac1{|y|}\right)
}
\mathbf 1_{\{|y|^2\ge2t\}}
\Bigg).
\end{align}
\end{lemma}

\begin{proof}
Using~\eqref{DefIncompleteBessel}, we have, for all \(z,a>0\),
\[
K_0(z,a)
=
\frac12
\int_a^\infty
\frac1u
e^{-u-\frac{z^2}{4u}}
\,du
\le
\frac12\int_a^\infty\frac{e^{-u}}u\,du
\le
C\left(1+\log^+\frac1a\right),
\]
where the last inequality follows by splitting the integral over
\([a,1]\) and \([1,\infty)\) when \(0<a<1\), while for \(a\ge1\) the integral
\(\int_a^\infty e^{-u}u^{-1}\,du\) is bounded by a universal constant. Consequently, it follows from Lemma~\ref{LemmaGaussianGroundStateConvolution} that, for all \(y\in\mathbb R^2\), 
\begin{align}\label{BoundGaussianBesselConvolution}
(g_t*\psi_\gamma)(y)
=
e^{\gamma t}
\sqrt{\frac{2\gamma}{\pi}}\,
K_0\!\left(\sqrt{2\gamma}\,|y|,\gamma t\right)
\le
C\left(
1+\log^+\frac1{\gamma t}
\right),
\end{align}
where \(C=C(\gamma,T)>0\), since \(0<t\le T\). 

By~\eqref{LowerBoundGroundState},
there exists \(C=C(\gamma,T)>0\) such that
\begin{align}
\frac1{\psi_\gamma(y)}
\le
\frac{C}{1+\log^+\frac1{|y|}},
\qquad
0<|y|\le\sqrt{2T}.
\end{align}
Indeed, if \(\sqrt{2T}\le1/2\), then the estimate follows directly from the first part of~\eqref{LowerBoundGroundState}. If \(\sqrt{2T} > 1/2\), then for \(1/2<|y|\le\sqrt{2T}\), the second part of~\eqref{LowerBoundGroundState} yields a uniform bound depending only on \(\gamma\) and \(T\), while \(1+\log^+\frac1{|y|}\) is bounded above by \(1+\log2\). Enlarging the constant if necessary gives the desired estimate.

Similarly, using~\eqref{BoundGroundStateDrift2d} and considering the cases
\(\sqrt{2T}\le 1/2\), \(1/2<\sqrt{2T}\le1\), and \(\sqrt{2T}>1\), we obtain
\begin{align}\label{DriftBoundOnBoundedRegion}
|b^\gamma(y)|
\le
\frac{C}{
|y|\left(
1+\log^+\frac1{|y|}
\right)},
\qquad
0<|y|\le\sqrt{2T}.
\end{align}
Therefore, if \(|y|^2<2t\), then \(|y|\le\sqrt{2T}\), and hence~\eqref{BoundGaussianBesselConvolution}--\eqref{DriftBoundOnBoundedRegion} give
\begin{align}\label{CombinedGroundStateInnerBound}
\frac{
\bigl|(g_t*\psi_\gamma)(y)b^\gamma(y)\bigr|^2
}{
\psi_\gamma(y)
}
\le
C
\frac{
\left(
1+\log^+\frac1{\gamma t}
\right)^2
}{
|y|^2
\left(
1+\log^+\frac1{|y|}
\right)^3
}.
\end{align}

We next consider the region \(|y|^2\ge2t\). Since $(g_t*\psi_\gamma)(y)=e^{\gamma t}\psi_\gamma(y)S_t^\gamma(y)$ by~\eqref{Sfunction} and \(0\le S_t^\gamma(y)\le1\), we have the first inequality below:
\begin{align}\label{GroundStateDriftProductBound}
\frac{
\bigl|(g_t*\psi_\gamma)(y)b^\gamma(y)\bigr|^2
}{
\psi_\gamma(y)
}
\le
e^{2\gamma T}
\psi_\gamma(y)|b^\gamma(y)|^2
\le
\frac{C}{
|y|^2\left(1+\log\dfrac1{|y|}\right)
}
\mathbf 1_{\{|y|\le1/2\}}
+
\frac{C}{|y|^2}
\mathbf 1_{\{|y|>1/2\}},
\end{align}
for some constant \(C=C(\gamma,T)>0\). For the second inequality, observe that for \(0<|y|\le1/2\), the desired bound follows directly from~\eqref{BoundGroundStateDrift2d} and~\eqref{UpperBoundGroundState}, whereas for
\(|y|>1/2\), again using~\eqref{BoundGroundStateDrift2d} and~\eqref{UpperBoundGroundState}, we obtain
\begin{align}
\psi_\gamma(y)|b^\gamma(y)|^2
\le
\frac{C}{\sqrt{|y|}}e^{-\sqrt{2\gamma}|y|}
=
\frac{C}{|y|^2}
\left(
|y|^{3/2}e^{-\sqrt{2\gamma}|y|}
\right)
\le
\frac{C}{|y|^2}
\sup_{s>1/2}
s^{3/2}e^{-\sqrt{2\gamma}s}
\le
\frac{C}{|y|^2}, \nonumber
\end{align}
where the last inequality follows since
\(s\mapsto s^{3/2}e^{-\sqrt{2\gamma}s}\) is bounded on
\((1/2,\infty)\). By considering separately the cases
\(0<|y|\le1/2\), \(1/2<|y|\le1\), and \(|y|>1\),
and increasing the constant if necessary,
\eqref{GroundStateDriftProductBound} yields
\begin{align}\label{CombinedGroundStateOuterBound}
\frac{
\bigl|(g_t*\psi_\gamma)(y)b^\gamma(y)\bigr|^2
}{
\psi_\gamma(y)
}
\le
\frac{C}{
|y|^2
\left(
1+\log^+\frac1{|y|}
\right)
},
\qquad
|y|^2\ge2t.
\end{align}
Combining~\eqref{CombinedGroundStateInnerBound} and
\eqref{CombinedGroundStateOuterBound} proves
\eqref{CombinedGroundStateBound}.
\end{proof}

We now prove Lemma~\ref{LemSquareIntegrability} in the case \(d=2\).
Before proceeding with the proof, we recall an estimate for the function \(h_s^\gamma\) defined in~\eqref{DefPerturbedKernel}. Fix \(L=T\gamma\). Since \(s\gamma\le T\gamma=L\) for every \(0<s<T\), it follows from~\cite[Prop.~8.17]{CM} that there exists a constant
\(C=C(\gamma,T)>0\) such that, for all \(x,y\in\R^2\) and \(0<s<T\),
\begin{align}\label{2dPointInteractionTerm}
h_s^\gamma(x,y)
\leq
\frac{C}{s}
\frac{1}{1+\log^+\frac{1}{s\gamma}}
\left(
e^{-\frac{|x|^2}{2s}}
+
\log^+\frac{2s}{|x|^2}
\right)
\left(
e^{-\frac{|y|^2}{2s}}
+
\log^+\frac{2s}{|y|^2}
\right).
\end{align}
For simplicity, we adopt the following convention throughout the proof.
\begin{notation}
Throughout the proof, \(T>0\), \(\gamma>0\), and
\(x\in\mathbb R^2\setminus\{0\}\) are fixed. For nonnegative quantities \(f\) and \(g\), we write \(f\preceq g\) (respectively,
\(f\succeq g\)) if there exists a constant \(C=C(\gamma,T,x)>0\) such that \(f\le Cg\) (respectively, \(f\ge Cg\)).
\end{notation}

\begin{proof}[Proof of Lemma~\ref{LemSquareIntegrability} for $d=2$]
Fix \(T,\gamma>0\) and \(x\in\mathbb R^2\setminus\{0\}\). Using~\eqref{DefTransDensity},~\eqref{Drfitfunction}, and~\eqref{Sfunction}, we may write the inner integral appearing in
Lemma~\ref{LemSquareIntegrability} as
\begin{align*}
\int_{\R^2}
q_s^\gamma(x,y)\,
\bigl|S_{T-s}^{\gamma}(y)b^\gamma(y)\bigr|^2\,dy
&=
e^{-\gamma(2T-s)}
\frac{1}{\psi_\gamma(x)}
\int_{\R^2}
p_s^\gamma(x,y)\,
\frac{
\bigl|(g_{T-s}*\psi_\gamma)(y)b^\gamma(y)\bigr|^2
}{
\psi_\gamma(y)
}
\,dy \\
&=
e^{-\gamma(2T-s)}
\frac{1}{\psi_\gamma(x)}
\int_{\R^2}
g_s(x,y)
\frac{
\bigl|(g_{T-s}*\psi_\gamma)(y)b^\gamma(y)\bigr|^2
}{
\psi_\gamma(y)
}
\,dy
\\
&\qquad
+
e^{-\gamma(2T-s)}
\frac{1}{\psi_\gamma(x)}
\int_{\R^2}
h_s^\gamma(x,y)
\frac{
\bigl|(g_{T-s}*\psi_\gamma)(y)b^\gamma(y)\bigr|^2
}{
\psi_\gamma(y)
}
\,dy,
\end{align*}
where the second equality uses~\eqref{FundamentalSolution}. We denote the last terms above by \(J_s^{T,\gamma}(x)\) and \(K_s^{T,\gamma}(x)\), respectively. In the analysis below, we will first consider the term \(K_s^{T,\gamma}(x)\) and show that $\int_0^T K_s^{T,\gamma}(x)\,ds<\infty$. 

It follows from~\eqref{2dPointInteractionTerm} that, for \(0<s< T\),
\begin{align} \label{KayXGroundState}
K_s^{T,\gamma}(x)
\leq\;&
\frac{
C(\gamma,T) e^{-\gamma(2T-s)}
}{
s\psi_\gamma(x)
\left(
1+\log^+\frac{1}{s\gamma}
\right)
}
\left(
e^{-\frac{|x|^2}{2s}}
+
\log^+\frac{2s}{|x|^2}
\right)
\int_{\R^2}
\left(
e^{-\frac{|y|^2}{2s}}
+
\log^+\frac{2s}{|y|^2}
\right)
\frac{
\bigl|(g_{T-s}*\psi_\gamma)(y)b^\gamma(y)\bigr|^2
}{
\psi_\gamma(y)
}
\,dy
\nonumber\\
\preceq\;&
\frac{1}{
s\left(
1+\log^+\frac{1}{s\gamma}
\right)
}
\int_{\R^2}
\left(
e^{-\frac{|y|^2}{2s}}
+
\log^+\frac{2s}{|y|^2}
\right)
\frac{
\bigl|(g_{T-s}*\psi_\gamma)(y)b^\gamma(y)\bigr|^2
}{
\psi_\gamma(y)
}
\,dy,
\end{align}
where the second inequality follows from the definition of the notation \(\preceq\), since \(T,\gamma>0\) and \(x\in\R^2\setminus\{0\}\) are fixed, and \(0<s< T\). Applying~\eqref{CombinedGroundStateBound}
in~\eqref{KayXGroundState} yields that $K_s^{T,\gamma}(x)
\preceq
\widehat K_s^{T,\gamma}
+
\check K_s^{T,\gamma}$, where
\begin{align*}
\widehat K_s^{T,\gamma}
:={}&
\frac{1}{
s\left(
1+\log^+\frac1{s\gamma}
\right)
}
\int_{|y|^2<2(T-s)}
\frac{
\left(
e^{-\frac{|y|^2}{2s}}
+
\log^+\frac{2s}{|y|^2}
\right)
\left(
1+\log^+\frac1{\gamma(T-s)}
\right)^2
}{
|y|^2
\left(
1+\log^+\frac1{|y|}
\right)^3
}
\,dy,
\\
\check K_s^{T,\gamma}
:={}&
\frac{1}{
s\left(
1+\log^+\frac1{s\gamma}
\right)
}
\int_{|y|^2\ge2(T-s)}
\frac{
e^{-\frac{|y|^2}{2s}}
+
\log^+\frac{2s}{|y|^2}
}{
|y|^2
\left(
1+\log^+\frac1{|y|}
\right)
}
\,dy.
\end{align*}
Through going to the radial variable $R=\frac{|y|^2}{2(T-s)}$ we can express \(\widehat K_s^{T,\gamma}\) and
\(\check K_s^{T,\gamma}\) as
\begin{align}
\widehat K_s^{T,\gamma}
={}&
\frac{\pi
\left(
1+\log^+\frac1{\gamma(T-s)}
\right)^2
}{
s\left(
1+\log^+\frac1{s\gamma}
\right)
}
\int_0^1
\frac{
e^{-\frac{T-s}{s}R}
+
\log^+\frac{s}{(T-s)R}
}{
R
\left(
1+\log^+\frac1{\sqrt{2(T-s)R}}
\right)^3
}
\,dR, \nonumber \\
\check K_s^{T,\gamma}
={}&
\frac{\pi}{
s\left(
1+\log^+\frac1{s\gamma}
\right)
}
\int_1^\infty
\frac{
e^{-\frac{T-s}{s}R}
+
\log^+\frac{s}{(T-s)R}
}{
R
\left(
1+\log^+\frac1{\sqrt{2(T-s)R}}
\right)
}
\,dR. \label{SecondK}
\end{align}

The integral of \(\widehat K_s^{T,\gamma}\) over \(s\in[0,T]\) has the
following expression:
\begin{align}\label{Yabuz3GroundState}
\int_0^T \widehat K_s^{T,\gamma}\,ds
={}&
\pi\int_0^T
\int_0^1
\frac{
\left(
e^{-\frac{T-s}{s}R}
+
\log^+\frac{s}{(T-s)R}
\right)
\left(
1+\log^+\frac{1}{\gamma(T-s)}
\right)^2
}{
s
\left(
1+\log^+\frac{1}{s\gamma}
\right)
R
\left(
1+\log^+\frac{1}{\sqrt{2(T-s)R}}
\right)^3
}
\,dR\,ds
\nonumber\\
={}&
\pi\int_0^1
\frac{1}{R}
\int_0^1
\frac{
\left(
e^{-\frac{1-a}{a}R}
+
\log^+\frac{a}{(1-a)R}
\right)
\left(
1+\log^+\frac{1}{\gamma T(1-a)}
\right)^2
}{
a
\left(
1+\log^+\frac{1}{\gamma Ta}
\right)
\left(
1+\log^+\frac{1}{\sqrt{2T(1-a)R}}
\right)^3
}
\,da\,dR.
\end{align}
The second equality changes to the integration variable
\(a=\frac{s}{T}\) and swaps the order of integration. We denote the
inner integral on the last line of~\eqref{Yabuz3GroundState} by
\(\mathrm{(A)}\). We will show below that
\begin{align}\label{PreLookGroundState}
\mathrm{(A)}
\preceq
\frac{
1+\log\left(
1+\log\frac{1}{R}
\right)
}{
\left(
1+\log\frac{1}{R}
\right)^2
},
\qquad 0<R<1.
\end{align}
Notice that applying~\eqref{PreLookGroundState}
in~\eqref{Yabuz3GroundState} gives
\begin{align*}
\int_0^T \widehat K_s^{T,\gamma}\,ds
&\preceq
\int_0^1
\frac{
1+\log\left(
1+\log\frac{1}{R}
\right)
}{
R
\left(
1+\log\frac{1}{R}
\right)^2
}
\,dR
=
\int_1^\infty \frac{1+\log u}{u^2}\,du=2<\infty,
\end{align*}
where the first equality follows from the change of variables $u=1+\log R^{-1}$.

Towards proving~\eqref{PreLookGroundState}, we split
\(\mathrm{(A)}\) over \(a\in(0,1/2)\) and
\(a\in[1/2,1)\). For \(0<a\le1/2\), we have
\(1/2\le1-a\le1\), and therefore
\begin{align*}
1+\log^+\frac{1}{\gamma T(1-a)}
\preceq 1,
\quad \textup{and} \quad
1+\log^+\frac{1}{\sqrt{2T(1-a)R}}
\succeq
1+\log\frac{1}{R}.
\end{align*}
Moreover, since \(1-a\ge \frac12\), we have $e^{-\frac{1-a}{a}R}
\le
e^{-\frac{R}{2a}}$ and $\log^+\frac{a}{(1-a)R}
\le
\log^+\frac{2a}{R}$. It follows that
\begin{align}\label{AGroundStateFirstHalf}
\int_0^{\frac12}
&\frac{
\left(
e^{-\frac{1-a}{a}R}
+
\log^+\frac{a}{(1-a)R}
\right)
\left(
1+\log^+\frac{1}{\gamma T(1-a)}
\right)^2
}{
a
\left(
1+\log^+\frac{1}{\gamma Ta}
\right)
\left(
1+\log^+\frac{1}{\sqrt{2T(1-a)R}}
\right)^3
}
\,da \nonumber \\
&\preceq
\frac{1}{
\left(
1+\log\frac{1}{R}
\right)^3
}
\Bigg(
\int_0^{\frac12}
\frac{e^{-\frac{R}{2a}}}{a}\,da
+
\int_0^{\frac12}
\frac{
\log^+\frac{2a}{R}
}{
a
\left(
1+\log^+\frac{1}{\gamma Ta}
\right)
}
\,da
\Bigg).
\end{align}
For the second integral on the right-hand side, we observe that
\begin{align}
\int_0^{\frac12}
\frac{
\log^+\frac{2a}{R}
}{
a
\left(
1+\log^+\frac{1}{\gamma Ta}
\right)
}
\,da
={}&
\int_{\frac{R}{2}}^{\frac12}
\frac{
\log\frac{2a}{R}
}{
a
\left(
1+\log^+\frac{1}{\gamma Ta}
\right)
}
\,da
\nonumber\\
\leq{}&
\log\frac{1}{R}
\int_{\frac{R}{2}}^{\frac12}
\frac{1}{
a
\left(
1+\log^+\frac{1}{\gamma Ta}
\right)
}
\,da
\preceq{}
\left(
\log\frac{1}{R}
\right)
\left(
1+
\log\left(
1+\log\frac{1}{R}
\right)
\right). \nonumber
\end{align}
The last inequality follows from $\int
\frac{1}{
a\left(1+\log\frac{1}{a}\right)
}
\,da
=
-\log\left(1+\log\frac{1}{a}\right)$. Also,
\begin{align*}
\int_0^{\frac12}
\frac{e^{-\frac{R}{2a}}}{a}\,da
&\leq
\int_0^1
\frac{e^{-\frac{R}{2a}}}{a}\,da
=
E\left(\frac{R}{2}\right)
\preceq
1+\log\frac{1}{R},
\end{align*}
where \(E\) is the exponential integral defined by $E(x)
:=
\int_0^1
\frac{e^{-x/a}}{a}\,da$ for $x>0$
and satisfies $E(x)\preceq 1+\log\frac1x$ for $0<x\le1$. Substituting the above observations into~\eqref{AGroundStateFirstHalf}, we obtain
\begin{align}\label{AGroundStateFirstHalfBound}
&\int_0^{\frac12}
\frac{
\left(
e^{-\frac{1-a}{a}R}
+
\log^+\frac{a}{(1-a)R}
\right)
\left(
1+\log^+\frac{1}{\gamma T(1-a)}
\right)^2
}{
a
\left(
1+\log^+\frac{1}{\gamma Ta}
\right)
\left(
1+\log^+\frac{1}{\sqrt{2T(1-a)R}}
\right)^3
}
\,da
\preceq
\frac{
1+\log\left(
1+\log\frac{1}{R}
\right)
}{
\left(
1+\log\frac{1}{R}
\right)^2
},
\end{align}
where, in the last step, we used that both $1+L$ and $L\bigl(1+\log(1+L)\bigr)$ are bounded above by $(1+L)\bigl(1+\log(1+L)\bigr)$, with $L:=\log\frac1R\ge0$, since $0<R<1$.

We next consider the integral over \(a\in[\frac12,1)\). Since
\(a\ge\frac12\), we have $
a^{-1}\big(
1+\log^+\frac{1}{\gamma Ta}
\big)^{-1}
\preceq 1$. We first suppose that \(0<R\le\frac12\) and split the integral over $\frac12\le a\le1-R$ and $1-R<a<1$. When \(\frac12\le a\le1-R\), we have
\begin{align*}
e^{-\frac{1-a}{a}R}
+
\log^+\frac{a}{(1-a)R}
&\preceq
1+\log\frac{1}{R},
\\
1+\log^+\frac{1}{\gamma T(1-a)}
&\preceq
1+\log\frac{1}{1-a},
\\
1+\log^+\frac{1}{\sqrt{2T(1-a)R}}
&\succeq
1+\log\frac{1}{R},
\end{align*}
where the first inequality uses that \(1-a\ge R\) and the next two inequalities use that \(1-a\le 1/2\). It follows that
\begin{align}\label{AGroundStateSecondHalfOne}
\int_{\frac12}^{1-R}
\frac{
\left(
e^{-\frac{1-a}{a}R}
+
\log^+\frac{a}{(1-a)R}
\right)
\left(
1+\log^+\frac{1}{\gamma T(1-a)}
\right)^2
}{
a
\left(
1+\log^+\frac{1}{\gamma Ta}
\right)
\left(
1+\log^+\frac{1}{\sqrt{2T(1-a)R}}
\right)^3
}
\,da
&\preceq
\frac{1}{
\left(
1+\log\frac{1}{R}
\right)^2
}
\int_{\frac12}^{1}
\left(
1+\log\frac{1}{1-a}
\right)^2
\,da
\nonumber\\
&=
\frac{1}{
\left(
1+\log\frac{1}{R}
\right)^2
}
\int_{\log 2}^{\infty}
(1+u)^2e^{-u}
\,du
\nonumber\\
&\preceq
\frac{1}{
\left(
1+\log\frac{1}{R}
\right)^2
}.
\end{align}
Here, the equality follows by first using the change of variables \(r=1-a\) and then \(u=\log(1/r)\); the resulting integral is finite.

On the other hand, when \(1-R<a<1\), we have
\begin{align*}
e^{-\frac{1-a}{a}R}
+
\log^+\frac{a}{(1-a)R}
&\preceq
1+\log\frac{1}{(1-a)R},
\\
1+\log^+\frac{1}{\gamma T(1-a)}
&\preceq
1+\log\frac{1}{1-a}
\le
1+\log\frac{1}{(1-a)R},
\\
1+\log^+\frac{1}{\sqrt{2T(1-a)R}}
&\succeq
1+\log\frac{1}{(1-a)R}.
\end{align*}
Consequently,
\begin{align}\label{AGroundStateSecondHalfTwo}
&\int_{1-R}^{1}
\frac{
\left(
e^{-\frac{1-a}{a}R}
+
\log^+\frac{a}{(1-a)R}
\right)
\left(
1+\log^+\frac{1}{\gamma T(1-a)}
\right)^2
}{
a
\left(
1+\log^+\frac{1}{\gamma Ta}
\right)
\left(
1+\log^+\frac{1}{\sqrt{2T(1-a)R}}
\right)^3
}
\,da
\preceq
\int_{1-R}^{1}1\,da
=
R
\preceq
\frac{1}{
\left(
1+\log\frac{1}{R}
\right)^2
}.
\end{align}
The last inequality uses that the function $R\mapsto R\left(1+\log\frac1R\right)^2$ is bounded on $(0,1)$, since it is continuous on $(0,1)$ and tends to $0$ as $R\to0$. 

We now suppose that \(\frac12<R<1\). In this case \(1-R<\frac12\),
so that \(a>1-R\) for every \(a\in[\frac12,1)\). The preceding
estimates (the estimates above~\eqref{AGroundStateSecondHalfTwo}) therefore give
\begin{align}\label{AGroundStateSecondHalfLargeR}
&\int_{\frac12}^{1}
\frac{
\left(
e^{-\frac{1-a}{a}R}
+
\log^+\frac{a}{(1-a)R}
\right)
\left(
1+\log^+\frac{1}{\gamma T(1-a)}
\right)^2
}{
a
\left(
1+\log^+\frac{1}{\gamma Ta}
\right)
\left(
1+\log^+\frac{1}{\sqrt{2T(1-a)R}}
\right)^3
}
\,da
\preceq
\int_{\frac12}^{1}1\,da
\preceq
\frac{1}{
\left(
1+\log\frac{1}{R}
\right)^2
},
\end{align}
where the last inequality follows since $1
\le
1+\log\frac{1}{R}
\le
1+\log 2$ for $1/2<R<1$. Combining~\eqref{AGroundStateFirstHalfBound} with
\eqref{AGroundStateSecondHalfOne}--\eqref{AGroundStateSecondHalfLargeR},
and using that \(1\le 1+\log\bigl(1+\log\frac1R\bigr)\) for \(0<R<1\),
we obtain~\eqref{PreLookGroundState}.

It remains to estimate \(\int_0^T \check K_s^{T,\gamma}\,ds\), which, by~\eqref{SecondK}, can be written as
\begin{align}
\int_0^T \check K_s^{T,\gamma}\,ds
&=
\int_0^T\int_1^\infty
\frac{\pi}{
s\left(
1+\log^+\frac1{s\gamma}
\right)
}
\frac{
e^{-\frac{T-s}{s}R}
+
\log^+\frac{s}{(T-s)R}
}{
R
\left(
1+\log^+\frac1{\sqrt{2(T-s)R}}
\right)
}
\,dR\,ds
\nonumber\\
&=
\pi
\int_1^\infty
\frac{1}{R}
\int_0^1
\frac{
e^{-\frac{1-a}{a}R}
+
\log^+\frac{a}{(1-a)R}
}{
a\left(
1+\log^+\frac{1}{\gamma Ta}
\right)
\left(
1+\log^+\frac{1}{\sqrt{2T(1-a)R}}
\right)
}
\,da\,dR
\nonumber\\
&\le
\pi
\int_1^\infty
\frac{1}{R}
\int_0^1
\left(
e^{-\frac{1-a}{a}R}
+
\log^+\frac{a}{(1-a)R}
\right)
\frac{da}{a}
\,dR, \nonumber
\end{align}
where the second equality follows by changing variables to \(a=\frac{s}{T}\) and swapping the order of integration, while the inequality follows from the elementary bound \(1+\log^+L\ge1\). For the exponential term, using the change of variables
\(u=\frac{1-a}{a}\), we have
\begin{align*}
\int_0^1
e^{-\frac{1-a}{a}R}
\frac{da}{a}
&=
\int_0^\infty
\frac{e^{-Ru}}{1+u}
\,du
\leq
\int_0^\infty
e^{-Ru}
\,du
=
\frac{1}{R}.
\end{align*}
Similarly, using the change of variables \(u=\frac{1-a}{a}\), we have
\begin{align*}
\int_0^1
\log^+\frac{a}{(1-a)R}
\frac{da}{a}
=
\int_0^\infty
\frac{
\log^+\frac{1}{Ru}
}{
1+u
}
\,du
=
\int_0^{1/R}
\frac{
\log\frac{1}{Ru}
}{
1+u
}
\,du
\leq
\int_0^{1/R}
\log\frac{1}{Ru}
\,du
=
\frac{1}{R}.
\end{align*}
Therefore, $\int_0^T \check K_s^{T,\gamma}\,ds
\preceq
\int_1^\infty
\frac{1}{R^2}
\,dR
<\infty$. Combining this with the estimate for
\(\widehat K_s^{T,\gamma}\), we conclude that $\int_0^T K_s^{T,\gamma}\,ds<\infty$.

For the Gaussian part, using $g_s(x,y)
=
\frac{1}{2\pi s}
e^{-\frac{|x-y|^2}{2s}}$, we have, for \(0<s<T\),
\begin{align}\label{JayXGroundState}
J_s^{T,\gamma}(x)
\preceq
\frac{1}{s}
\int_{\R^2}
e^{-\frac{|x-y|^2}{2s}}
\frac{
\bigl|(g_{T-s}*\psi_\gamma)(y)b^\gamma(y)\bigr|^2
}{
\psi_\gamma(y)
}
\,dy.
\end{align}
Next, we split the
\(y\)-integration in~\eqref{JayXGroundState} over $\big\{|y|<|x|/2\big\}$ and $\big\{|y|\ge |x|/2\big\}$. Accordingly, applying~\eqref{CombinedGroundStateBound}, we write $J_s^{T,\gamma}(x)
\preceq
J_{s,1}^{T,\gamma}(x)
+
J_{s,2}^{T,\gamma}(x)$, where
\begin{align*}
J_{s,1}^{T,\gamma}(x)
:={}&
\frac1s
\int_{|y|<|x|/2}
e^{-\frac{|x-y|^2}{2s}}
\Bigg(
\frac{
\left(
1+\log^+\frac1{\gamma(T-s)}
\right)^2
}{
|y|^2
\left(
1+\log^+\frac1{|y|}
\right)^3
}
\mathbf 1_{\{|y|^2<2(T-s)\}}
+
\frac{
1
}{
|y|^2
\left(
1+\log^+\frac1{|y|}
\right)
}
\mathbf 1_{\{|y|^2\ge2(T-s)\}}
\Bigg)
\,dy,
\\
J_{s,2}^{T,\gamma}(x)
:={}&
\frac1s
\int_{|y|\ge |x|/2}
e^{-\frac{|x-y|^2}{2s}}
\Bigg(
\frac{
\left(
1+\log^+\frac1{\gamma(T-s)}
\right)^2
}{
|y|^2
\left(
1+\log^+\frac1{|y|}
\right)^3
}
\mathbf 1_{\{|y|^2<2(T-s)\}}
+
\frac{
1
}{
|y|^2
\left(
1+\log^+\frac1{|y|}
\right)
}
\mathbf 1_{\{|y|^2\ge2(T-s)\}}
\Bigg)
\,dy.
\end{align*}
For \(J_{s,2}^{T,\gamma}(x)\), using that
\(|y|\ge |x|/2\) on the domain of integration and the elementary bound
\(1+a\ge1\) for \(a\ge0\), we obtain
\begin{align*}
J_{s,2}^{T,\gamma}(x)
&\preceq
\frac1{s|x|^2}
\left[
\left(
1+\log^+\frac1{\gamma(T-s)}
\right)^2
\mathbf 1_{\{|x|^2<8(T-s)\}}
+1
\right]
\int_{\mathbb R^2}
e^{-\frac{|x-y|^2}{2s}}
\,dy
\preceq
1,
\end{align*}
where the second inequality uses that
\(\int_{\mathbb R^2}e^{-\frac{|x-y|^2}{2s}}\,dy=2\pi s\) and that
\(\log^+\frac1{\gamma(T-s)}
\le
\log^+\frac8{\gamma|x|^2} \preceq 1\) whenever
\(|x|^2<8(T-s)\). Consequently,
\(\int_0^T J_{s,2}^{T,\gamma}(x)\,ds<\infty\).

We now estimate \(J_{s,1}^{T,\gamma}(x)\). If
\(|y|<|x|/2\), then $|x-y|
\ge
|x|-|y|
>
|x|/2$ and therefore passing to polar coordinates gives
\begin{align}\label{FirstJGroundState}
J_{s,1}^{T,\gamma}(x)
\preceq{}&
\frac{e^{-\frac{|x|^2}{8s}}}{s}
\Bigg[
\left(
1+\log^+\frac1{\gamma(T-s)}
\right)^2
\int_0^{\min\{|x|/2,\sqrt{2(T-s)}\}}
\frac{
dr
}{
r\left(
1+\log^+\frac1r
\right)^3
}
\nonumber\\
&\hspace{24mm}
+
\int_{\min\{|x|/2,\sqrt{2(T-s)}\}}^{|x|/2}
\frac{
dr
}{
r\left(
1+\log^+\frac1r
\right)
}
\Bigg],
\end{align}
where the second integral is understood to be zero whenever
\(\sqrt{2(T-s)}\ge |x|/2\). Denote the first and second terms inside
the brackets in~\eqref{FirstJGroundState} by
\(\widehat J_{s,1}^{T,\gamma}(x)\) and
\(\check J_{s,1}^{T,\gamma}(x)\), respectively. Then we have
\begin{align*}
\widehat J_{s,1}^{T,\gamma}(x)
\le
\left(
1+\log^+\frac1{\gamma(T-s)}
\right)^2
\int_0^{\sqrt{2(T-s)}}
\frac{
dr
}{
r\left(
1+\log^+\frac1r
\right)^3
}
\preceq 1,
\end{align*}
where the last inequality follows since if \(\sqrt{2(T-s)}\le1\), then \(\log^+\frac1r=\log\frac1r\) for
\(0<r\le\sqrt{2(T-s)}\), and
\begin{align*}
\int_0^{\sqrt{2(T-s)}}
\frac{
dr
}{
r\left(
1+\log\frac1r
\right)^3
}
=
\frac{1}{
2\left(
1+\log\frac1{\sqrt{2(T-s)}}
\right)^2
}\,,
\quad \textup{and} \quad
1+\log^+\frac1{\gamma(T-s)}
\preceq
1+\log\frac1{\sqrt{2(T-s)}}.
\end{align*}
On the other hand, if \(\sqrt{2(T-s)}>1\), then we have
\[
1+\log^+\frac1{\gamma(T-s)}
\preceq 1,
\quad \textup{and} \quad
\int_0^{\sqrt{2(T-s)}}
\frac{
dr
}{
r\left(
1+\log^+\frac1r
\right)^3
}
\le
\int_0^{\sqrt{2T}}
\frac{
dr
}{
r\left(
1+\log^+\frac1r
\right)^3
}
<\infty,
\]
where the finiteness of the latter integral follows by splitting it over \((0,1)\) and \((1,\sqrt{2T})\), yielding an integrable logarithmic singularity near the origin and a finite logarithmic integral away from the origin.

For \(\check J_{s,1}^{T,\gamma}(x)\), there is nothing to prove when
\(\sqrt{2(T-s)}\ge |x|/2\), since in this case
\(\check J_{s,1}^{T,\gamma}(x)=0\). On the other hand, if
\(\sqrt{2(T-s)}<|x|/2\), then
\begin{align*}
\check J_{s,1}^{T,\gamma}(x)
=
\int_{\sqrt{2(T-s)}}^{|x|/2}
\frac{
dr
}{
r\left(
1+\log^+\frac1r
\right)
}
\le
\int_{\sqrt{2(T-s)}}^{|x|}
\frac{dr}{r}
=
\log\frac{|x|}{\sqrt{2(T-s)}}.
\end{align*}
Consequently, in both cases, we have
\(\check J_{s,1}^{T,\gamma}(x)
\le
\log^+\frac{|x|}{\sqrt{2(T-s)}}\).
Therefore, combining this estimate with
\(\widehat J_{s,1}^{T,\gamma}(x)\preceq1\), and integrating over \(s\in(0,T)\), we obtain from~\eqref{FirstJGroundState}
\[
\int_0^T J_{s,1}^{T,\gamma}(x)\,ds
\preceq
\int_0^{T}
\frac{e^{-\frac{|x|^2}{8s}}}{s}
\left(
1+\log^+\frac{|x|}{\sqrt{2(T-s)}}
\right)\,ds
<\infty.
\]
Indeed, after splitting the integral at \(T/2\), the exponential factor \(e^{-|x|^2/(8s)}\) ensures integrability near \(s=0\), whereas near \(s=T\), the factor \(s^{-1}e^{-|x|^2/(8s)}\) is bounded and \(\log^+\!\bigl(|x|/\sqrt{2(T-s)}\bigr)\) is integrable. Consequently, \(\int_0^T J_s^{T,\gamma}(x)\,ds<\infty\), which completes the proof for \(S_{T-s}^{\gamma}(y)b^\gamma(y)\). By
Lemma~\ref{CorRegularizedDriftComparison}, the same estimate also
holds with \(S_{T-s}^{\gamma}(y)b^\gamma(y)\) replaced by \(S_{T-s}^{\gamma}(y)\widehat b_{T-s}^\gamma(y)\). \vspace{.2cm}

\noindent \textit{\textbf{The regularized drift case.}} It remains to prove the estimate for \(\widehat b_{T-s}^{\gamma}(y)\). Using the definitions of \(p^\gamma\) and \(q^\gamma\), we have, for \(0<s<T\),
\begin{align*}
\int_{\mathbb R^2}
q_s^\gamma(x,y)
\bigl|\widehat b_{T-s}^\gamma(y)\bigr|^2\,dy
=
\frac{e^{-\gamma s}}{\psi_\gamma(x)}
\int_{\mathbb R^2}
g_s(x,y)
\psi_\gamma(y)
\bigl|\widehat b_{T-s}^\gamma(y)\bigr|^2
\,dy
+
\frac{e^{-\gamma s}}{\psi_\gamma(x)}
\int_{\mathbb R^2}
h_s^\gamma(x,y)
\psi_\gamma(y)
\bigl|\widehat b_{T-s}^\gamma(y)\bigr|^2
\,dy.
\end{align*}
We denote the two terms on the right-hand side by
\(\widetilde J_s^{T,\gamma}(x)\) and
\(\widetilde K_s^{T,\gamma}(x)\), respectively. Combining~\eqref{UpperBoundGroundState} with~Lemma~\ref{LemmaRegularizedDriftBound2d}, we obtain, for all
\(0<u\le T\) and all \(y\in\mathbb R^2\setminus\{0\}\),
\begin{align}\label{RegularizedDriftWeightedBound}
\psi_\gamma(y)
\bigl|\widehat b_u^\gamma(y)\bigr|^2
\preceq{}&
\frac{
|y|^2
\left(
1+\log\frac1{|y|}
\right)
}{
u^2
\left(
1+\log^+\frac1{\gamma u}
\right)^2
}
\mathbf 1_{\left\{
0<|y|<
\min\{\sqrt{2u},1/2\}
\right\}}
+
\frac{
1
}{
|y|^2
\left(
1+\log\frac1{|y|}
\right)
}
\mathbf 1_{\left\{
\sqrt{2u}\le |y|\le1/2
\right\}}
\nonumber\\
&+
|y|^{-1/2}
e^{-\sqrt{2\gamma}|y|}
\mathbf 1_{\{|y|>1/2\}}.
\end{align}
Here and below, the second region is understood to be empty whenever
\(\sqrt{2u}>1/2\). Thus, using $g_s(x,y) = (2\pi s)^{-1} e^{-|x-y|^2/(2s)}$ together with the estimate~\eqref{RegularizedDriftWeightedBound}, we obtain
\begin{align}
\widetilde J_s^{T,\gamma}(x)
\preceq
\frac{1}{s}
\int_{\mathbb R^2}
e^{-\frac{|x-y|^2}{2s}}
\psi_\gamma(y)
\bigl|\widehat b_{T-s}^{\gamma}(y)\bigr|^2
\,dy
\preceq
\widetilde J_{s,1}^{T,\gamma}(x)
+
\widetilde J_{s,2}^{T,\gamma}(x)
+
\widetilde J_{s,3}^{T,\gamma}(x),
\label{RegularizedGaussianDecomposition}
\end{align}
where the three terms correspond to the three regions in
\eqref{RegularizedDriftWeightedBound} and are estimated below separately.

We first estimate
\(\widetilde J_{s,1}^{T,\gamma}(x)\) by considering the cases
\(0<s\le T/2\) and \(T/2<s<T\). Observe that, for
\(0<s\le T/2\), we have \(T/2\le T-s\le T\), and hence $(T-s) \big(
1+\log^+\frac1{\gamma(T-s)}
\big) \succeq 1 $. Moreover, the function
\(y\mapsto |y|^2\bigl(1+\log|y|^{-1}\bigr)\) is bounded on
\(0<|y|\le1/2\). Therefore,
\begin{align*}
\widetilde J_{s,1}^{T,\gamma}(x)
&:=
\frac{1}{
s(T-s)^2
\left(
1+\log^+\frac1{\gamma(T-s)}
\right)^2
}
\int_{\left\{
0<|y|<
\min\{\sqrt{2(T-s)},1/2\}
\right\}}
e^{-\frac{|x-y|^2}{2s}}
|y|^2
\left(
1+\log\frac1{|y|}
\right)
\,dy
\\
&\preceq
\frac1s
\int_{\mathbb R^2}
e^{-\frac{|x-y|^2}{2s}}
\,dy
\preceq
1,
\qquad
0<s\le\frac T2.
\end{align*}
On the other hand, for \(T/2<s<T\), we have \(s^{-1}\preceq1\).
Dropping the Gaussian factor and using polar coordinates, we obtain
\begin{align*}
\widetilde J_{s,1}^{T,\gamma}(x)
&\preceq
\frac{1}{
(T-s)^2
\left(
1+\log^+\frac1{\gamma(T-s)}
\right)^2
}
\int_0^{\min\{\sqrt{2(T-s)},1/2\}}
r^3
\left(
1+\log\frac1r
\right)
\,dr.
\end{align*}
For
\(a=\min\{\sqrt{2(T-s)},1/2\}\), integration by parts gives
\begin{align*}
\int_0^a
r^3\left(1+\log\frac1r\right)\,dr
=
\frac{a^4}{4}
+
\frac{a^4}{4}\log\frac1a
+
\frac{a^4}{16}
\preceq
a^4\left(1+\log\frac1a\right)
\preceq
(T-s)^2
\left(
1+\log^+\frac1{\gamma(T-s)}
\right),
\end{align*}
where the last inequality follows by splitting into the cases
\(a=\sqrt{2(T-s)}\) and \(a=1/2\): in the former case,
\(a^4\bigl(1+\log a^{-1}\bigr)\) is bounded by the right-hand side after enlarging the implicit constant, while in the latter case
\(1/8\le T-s\le T\), so
\(a^4\bigl(1+\log a^{-1}\bigr)\preceq1\), and the estimate follows immediately. Consequently,
\begin{align*}
\widetilde J_{s,1}^{T,\gamma}(x)
\preceq
\frac{1}{
\left(
1+\log^+\frac1{\gamma(T-s)}
\right)
}
\preceq 1,
\qquad T/2<s<T.
\end{align*}
Combining the above estimates over the two cases, we conclude that
\(\int_0^T \widetilde J_{s,1}^{T,\gamma}(x)\,ds<\infty\).

For \(\widetilde J_{s,2}^{T,\gamma}(x)\), it suffices to assume that
\(T-s\le1/8\), since otherwise the integration region is empty. If
\(0<s\le T/2\), then $\sqrt{T} \le |y| \le 1/2$ throughout the
integration region. Therefore,
\begin{align}
\widetilde J_{s,2}^{T,\gamma}(x)
:=
\frac{1}{s}
\int_{\left\{
\sqrt{2(T-s)}\le |y|\le1/2
\right\}}
\frac{
e^{-\frac{|x-y|^2}{2s}}
}{
|y|^2
\left(
1+\log\frac1{|y|}
\right)
}
\,dy
\le 
\frac{1}{T}
\int_{\mathbb R^2}
\frac{e^{-\frac{|x-y|^2}{2s}}}{s}
\,dy
\preceq
1, \nonumber
\end{align}
and hence $\int_0^{T/2}\widetilde J_{s,2}^{T,\gamma}(x)\,ds<\infty$.
Now suppose that \(T/2<s<T\). Then,
\begin{align*}
\widetilde J_{s,2}^{T,\gamma}(x)
\le
\frac{2}{T}
\int_{\left\{
\sqrt{2(T-s)}\le |y|\le1/2
\right\}}
\frac{dy}{|y|^2}
=
\frac{4\pi}{T}
\int_{\sqrt{2(T-s)}}^{1/2}
\frac{dr}{r}
=
\frac{4\pi}{T}
\bigg(\log\frac{1}{2} +
\log\frac{1}{\sqrt{2(T-s)}} \bigg)
\preceq 
\left(
1+\log\frac1{T-s}
\right).
\end{align*}
Consequently,
\begin{align*}
\int_{T/2}^T
\widetilde J_{s,2}^{T,\gamma}(x)\,ds
\preceq 
\int_{\max\left\{
T/2,T-1/8
\right\}}^T
\left(
1+\log\frac1{T-s}
\right)
\,ds
=
\int_0^{\min\left\{
T/2,1/8
\right\}}
\left(
1+\log\frac1u
\right)
\,du
<\infty.
\end{align*}
The above estimates yield \(\int_0^T \widetilde J_{s,2}^{T,\gamma}(x)\,ds<\infty\).

For \(\widetilde J_{s,3}^{T,\gamma}(x)\), since $|y|>\frac12$, we have
\begin{align*}
\widetilde J_{s,3}^{T,\gamma}(x)
:=
\frac{1}{s}
\int_{\{|y|>1/2\}}
e^{-\frac{|x-y|^2}{2s}}
|y|^{-1/2}
e^{-\sqrt{2\gamma}|y|}
\,dy
\le
\sqrt{2}
\int_{\mathbb R^2}
\frac{e^{-\frac{|x-y|^2}{2s}}}{s}
\,dy
\preceq
1.
\end{align*}
Consequently, $\int_0^T \widetilde J_{s,3}^{T,\gamma}(x)\,ds <\infty$.

We next estimate \(\widetilde K_s^{T,\gamma}(x)\). Since
\(s\gamma<T\gamma\) for every \(0<s<T\), applying~\eqref{2dPointInteractionTerm} with \(L=T\gamma\) yields the first inequality below.
\begin{align}
\widetilde K_s^{T,\gamma}(x)
&:=
\frac{e^{-\gamma s}}{\psi_\gamma(x)}
\int_{\mathbb R^2}
h_s^\gamma(x,y)
\psi_\gamma(y)
\bigl|\widehat b_{T-s}^\gamma(y)\bigr|^2
\,dy
\nonumber\\
&\preceq
\frac{
e^{-\frac{|x|^2}{2s}}
+
\log^+\frac{2s}{|x|^2}
}{
s
\left(
1+\log^+\frac1{s\gamma}
\right)
}
\int_{\mathbb R^2}
\left(
e^{-\frac{|y|^2}{2s}}
+
\log^+\frac{2s}{|y|^2}
\right)
\psi_\gamma(y)
\bigl|\widehat b_{T-s}^\gamma(y)\bigr|^2
\,dy
\nonumber\\
&\preceq
\widetilde K_{s,1}^{T,\gamma}(x)
+
\widetilde K_{s,2}^{T,\gamma}(x)
+
\widetilde K_{s,3}^{T,\gamma}(x), \nonumber
\end{align}
where the three terms on the right-hand side correspond, respectively, to the three regions appearing in~\eqref{RegularizedDriftWeightedBound}, and will be estimated separately below.

For the first term, set $a:=\min\{\sqrt{2(T-s)},1/2\}$, then
\begin{align} \label{KoneEstimate}
\widetilde K_{s,1}^{T,\gamma}(x)
&:={}
\frac{
e^{-\frac{|x|^2}{2s}}
+
\log^+\frac{2s}{|x|^2}
}{
s(T-s)^2
\left(
1+\log^+\frac1{s\gamma}
\right)
\left(
1+\log^+\frac1{\gamma(T-s)}
\right)^2
}
\int_{\{0<|y|<a\}}
\left(
e^{-\frac{|y|^2}{2s}}
+
\log^+\frac{2s}{|y|^2}
\right)
|y|^2
\left(
1+\log\frac1{|y|}
\right)
\,dy \nonumber\\
&\preceq
\frac{
e^{-\frac{|x|^2}{2s}}
+
\log^+\frac{2s}{|x|^2}
}{
s(T-s)^2
\left(
1+\log^+\frac1{s\gamma}
\right)
\left(
1+\log^+\frac1{\gamma(T-s)}
\right)^2
}
\int_0^{a}
r^3
\left(
1+\log\frac1r
\right)^2
\,dr,
\end{align}
where the inequality uses that
$e^{-\frac{|y|^2}{2s}}+\log^+\frac{2s}{|y|^2}
\preceq 1+\log\frac1{|y|}$ since \(0<|y|\le1/2\) and \(0<s<T\).
Using the change of variables \(r=au\), we obtain
\begin{align*}
\int_0^a
r^3
\left(
1+\log\frac1r
\right)^2
\,dr
&=
a^4
\int_0^1
u^3
\left(
1+\log\frac1a+\log\frac1u
\right)^2
\,du
\\
&\preceq
a^4
\left(
1+\log\frac1a
\right)^2
+
a^4
\\
&\preceq
a^4
\left(
1+\log\frac1a
\right)^2
\preceq
(T-s)^2
\left(
1+\log^+\frac1{\gamma(T-s)}
\right)^2.
\end{align*}
Here, the first inequality follows from
\((A+B)^2\le 2A^2+2B^2\), with
\(A=1+\log a^{-1}\) and \(B=\log u^{-1}\), while the second inequality
uses \(1+\log a^{-1}\ge1\) for \(0<a\le1/2\). The final inequality
follows as in \(\widetilde J_{s,1}^{T,\gamma}(x)\), by considering
separately the cases \(T-s\le1/8\) (for which
\(a=\sqrt{2(T-s)}\)) and \(T-s>1/8\) (for which \(a=1/2\)).
Using the above estimate and the fact that
$1+\log^+\frac1{s\gamma}\ge1$, we obtain from~\eqref{KoneEstimate} that
\begin{align}
\int_0^T
\widetilde K_{s,1}^{T,\gamma}(x)
\,ds
\preceq
\int_0^T
\frac{e^{-\frac{|x|^2}{2s}}}{s}
\,ds
+
\int_0^T
\frac{\log^+\frac{2s}{|x|^2}}{s}
\,ds
=
\int_{\frac{|x|^2}{2T}}^\infty
\frac{e^{-u}}{u}
\,du
+
\mathbf 1_{\{|x|^2<2T\}}
\int_{\frac{|x|^2}{2}}^T
\frac{\log\frac{2s}{|x|^2}}{s}
\,ds
<\infty.
\end{align}
Here, the equality uses the change of variables
\(u=|x|^2/(2s)\) in the first integral and the fact that
\(\log^+\frac{2s}{|x|^2}=0\) for \(0<s\le |x|^2/2\).
The finiteness follows since \(e^{-u}/u\) is integrable on
\(\bigl[|x|^2/(2T),\infty\bigr)\) for \(x\neq0\), while, when
\(|x|^2<2T\), the second integrand is continuous on the finite interval
\(\bigl[|x|^2/2,T\bigr]\).

For the second term, the region of integration is nonempty only when
\(T-s\le1/8\). Moreover, since \(0<|y|\le1/2\) and \(0<s<T\), we have
\(e^{-\frac{|y|^2}{2s}}
+
\log^+\frac{2s}{|y|^2}
\preceq
1+\log\frac1{|y|}\). Therefore, using polar coordinates, we obtain
\begin{align}
\widetilde K_{s,2}^{T,\gamma}(x)
&:=
\frac{
e^{-\frac{|x|^2}{2s}}
+
\log^+\frac{2s}{|x|^2}
}{
s
\left(
1+\log^+\frac1{s\gamma}
\right)
}
\int_{\left\{
\sqrt{2(T-s)}\le |y|\le1/2
\right\}}
\frac{
e^{-\frac{|y|^2}{2s}}
+
\log^+\frac{2s}{|y|^2}
}{
|y|^2
\left(
1+\log\frac1{|y|}
\right)
}
\,dy
\nonumber\\
&\preceq
\frac{
e^{-\frac{|x|^2}{2s}}
+
\log^+\frac{2s}{|x|^2}
}{s}
\int_{\sqrt{2(T-s)}}^{1/2}
\frac{dr}{r}\,
\mathbf 1_{\{T-s\le1/8\}}
\nonumber\\
&\preceq
\frac{
e^{-\frac{|x|^2}{2s}}
+
\log^+\frac{2s}{|x|^2}
}{s}
\left(
1+\log^+\frac1{T-s}
\right)
\mathbf 1_{\{T-s\le1/8\}}. \nonumber
\end{align}
Consequently, splitting the \(s\)-integral over
\((0,T/2]\) and \((T/2,T)\), we obtain
\begin{align}
\int_0^T
\widetilde K_{s,2}^{T,\gamma}(x)
\,ds
&\preceq
\int_0^{T/2}
\frac{
e^{-\frac{|x|^2}{2s}}
+
\log^+\frac{2s}{|x|^2}
}{
s
}
\,ds
+
\int_{T/2}^T
\left(
1+\log^+\frac1{T-s}
\right)
\,ds
<\infty. \nonumber
\end{align}
Here, the first integral is finite by the same argument used for
\(\widetilde K_{s,1}^{T,\gamma}(x)\), while the second is finite since \(\log^+(1/u)\) is integrable near \(u=0\).

For the third term, since \(0<s<T\) and \(|y|>1/2\), we have $e^{-\frac{|y|^2}{2s}}+\log^+\frac{2s}{|y|^2} \le 1+\log^+(8T)$. Therefore, using polar coordinates, we obtain
\begin{align}
\widetilde K_{s,3}^{T,\gamma}(x)
&:=
\frac{
e^{-\frac{|x|^2}{2s}}
+
\log^+\frac{2s}{|x|^2}
}{
s
\left(
1+\log^+\frac1{s\gamma}
\right)
}
\int_{\{|y|>1/2\}}
\left(
e^{-\frac{|y|^2}{2s}}
+
\log^+\frac{2s}{|y|^2}
\right)
|y|^{-1/2}
e^{-\sqrt{2\gamma}|y|}
\,dy
\nonumber\\
&\preceq
\frac{
e^{-\frac{|x|^2}{2s}}
+
\log^+\frac{2s}{|x|^2}
}{s}
\int_{1/2}^{\infty}
r^{1/2}
e^{-\sqrt{2\gamma}r}
\,dr
\preceq
\frac{
e^{-\frac{|x|^2}{2s}}
+
\log^+\frac{2s}{|x|^2}
}{s}. \nonumber
\end{align}
Here the finiteness of the radial integral follows by choosing \(R>1/2\)
sufficiently large so that
\(r^{1/2}\le e^{\frac{\sqrt{2\gamma}}{2}r}\) for all \(r\ge R\), and
then splitting the integral over \([1/2,R]\) and \([R,\infty)\). Consequently, $\int_0^T
\widetilde K_{s,3}^{T,\gamma}(x)
\,ds < \infty$ by the same argument used for
\(\widetilde K_{s,1}^{T,\gamma}(x)\).

Combining the above estimates with, we conclude that
\(\int_0^T\widetilde K_s^{T,\gamma}(x)\,ds<\infty\), which completes the proof.
\end{proof}

\subsection{Square integrability of the martingale: \texorpdfstring{$d=3$}{Lg} } \label{SqureIntegrability3d}

For each \(\gamma>0\) and each \(T>0\), there exists a constant \(C=C(\gamma,T)>0\) such that for all \(0<t\le T\) and all \(x,y\in\R^3\setminus\{0\}\),
\begin{align}\label{PtgammaBounds}
g_t(x-y)
\le
p_t^\gamma(x,y)
\le
g_t(x-y)
+
\frac{C}{\sqrt{t}}
\frac{1}{|x||y|}
e^{-\frac{|x|^2}{2t}}
e^{-\frac{|y|^2}{2t}} ,
\end{align}
where \(p_t^\gamma(x,y)\) is defined in~\eqref{FundamentalSolution} and \(g_t(x-y):= (2\pi t)^{-3/2} e^{-|x-y|^2/(2t)}\) denotes the three-dimensional free heat kernel. Indeed, this follows from~\cite[Lem.~8]{Fleischmann} using the kernel identification \(p_t^\gamma(x,y)=P^\alpha(t/2;x,y)\), together with the time rescaling \(t\mapsto t/2\) and the parameter identification
\(\gamma=-4\pi\alpha\).

We begin by proving the following auxiliary lemma that will be used in the proof of Lemma~\ref{LemSquareIntegrability}. 

\begin{lemma}
Let \(T,\gamma>0\). Then there exists a constant
\(C=C(\gamma,T)>0\) such that, for all \(0<t\le T\) and all
\(x\in\mathbb R^3\setminus\{0\}\),
\begin{align}\label{BoundProductSDriftBound}
\bigl|S_t^\gamma(x)b^\gamma(x)\bigr|
\le
C\left(
1+
\frac{1}{\sqrt t}\mathbf 1_{\{|x|\le \sqrt t\}}
+
\frac{1}{|x|}\mathbf 1_{\{|x|>\sqrt t\}}
\right).
\end{align}
\end{lemma}

\begin{proof}
Recall from~\eqref{Drfitfunction} that $|b^\gamma(x)|=\gamma+|x|^{-1}$. On the set \(\{|x|>\sqrt t\}\), we use the bound \(0\le S_t^\gamma(x)\le 1\), which follows from~\eqref{Sfunction} and Lemma~\ref{LemmaGaussianGroundStateConvolution} since \(0\le \operatorname{erfc}(\cdot)\le 2\), to obtain
\begin{align*}
\bigl|S_t^\gamma(x)b^\gamma(x)\bigr|
\le
\gamma+\frac1{|x|}
\le
C(\gamma,T)\Big(1+\frac1{|x|}\Big).
\end{align*}
On the set \(\{|x|\le\sqrt t\}\), we split the integral into the
regions \(\{|y|\le2\sqrt t\}\) and \(\{|y|>2\sqrt t\}\). Then
\begin{align}
(g_t*\psi_\gamma)(x) 
\le
C\int_{\mathbb R^3}
g_t(x-y)\frac1{|y|}
\,dy
&=
C\int_{\{|y|\le2\sqrt t\}}
g_t(x-y)\frac1{|y|}
\,dy
+
C\int_{\{|y|>2\sqrt t\}}
g_t(x-y)\frac1{|y|}
\,dy
\nonumber\\
&\le
\frac{C}{t^{3/2}}
\int_{\{|y|\le2\sqrt t\}}
\frac1{|y|}
\,dy
+
\frac{C}{t^{3/2}}
\int_{\{|y|>2\sqrt t\}}
e^{-\frac{|y|^2}{8t}}
\frac1{|y|}
\,dy
\nonumber\\
&=
\frac{C}{t^{3/2}}
\int_0^{2\sqrt t} r\,dr
+
\frac{C}{t^{3/2}}
\int_{2\sqrt t}^{\infty}
r e^{-\frac{r^2}{8t}}
\,dr
\le
\frac{C}{\sqrt t}, \nonumber
\end{align}
where $C=C(\gamma)>0$. Here, on \(\{|y|>2\sqrt t\}\), the assumption
\(|x|\le\sqrt t\) gives $|x-y| \ge |y|-|x| \ge |y|/2$ which yields the second inequality. Using the above, we have
\begin{align*}
\bigl|S_t^\gamma(x)b^\gamma(x)\bigr|
=
e^{-\gamma^2t/2}
\frac{(g_t*\psi_\gamma)(x)}{\psi_\gamma(x)}
\left(\gamma+\frac1{|x|}\right)
\le
\frac{1}{\psi_\gamma(x)}
\frac{C(\gamma)}{\sqrt t}
\left(\gamma+\frac1{|x|}\right)
\le
|x|e^{\gamma |x|}
\frac{C(\gamma)}{\sqrt t}
\left(\gamma+\frac1{|x|}\right),
\end{align*}
where the final inequality follows from~\eqref{EigenFunctionEigenValue}. Finally, using that $|x|\le \sqrt t$ and $t \le T$, we conclude that the above is bounded by $C(\gamma,T)\big(1+\frac1{\sqrt t}\big)$.
\end{proof}

\begin{proof}[Proof of Lemma~\ref{LemSquareIntegrability} for $d=3$]
Using~\eqref{DefTransDensity} and~\eqref{BoundProductSDriftBound}, we obtain
\begin{align*}
\int_0^T
\int_{\mathbb R^3}
q_s^\gamma(x,y)
\,
\bigl|
 S_{T-s}^{\gamma}(y)
b^{\gamma}(y)
\bigr|^2
\,dy\,ds  
&\le
C
\int_0^T
\int_{\{|y|\le \sqrt{T-s}\}}
e^{-\frac{\gamma^2s}{2}}
p_s^\gamma(x,y)\psi_\gamma(y)
\left(1+\frac1{\sqrt{T-s}}\right)^2
\,dy\,ds
\\
&\quad+
C
\int_0^T
\int_{\{|y|> \sqrt{T-s}\}}
e^{-\frac{\gamma^2s}{2}}
p_s^\gamma(x,y)\psi_\gamma(y)
\left(1+\frac1{|y|}\right)^2
\,dy\,ds ,
\end{align*}
where \(C=C(\gamma,T,x)>0\). Let us call the two integrals by $I_1$ and $I_2$. Below we will bound each separately. \vspace{.2cm}

\noindent \textit{Estimating \(I_1\)}. For \(I_1\), we first split the time integral over $[0,T]$ into $[0,T/2]$ and $[T/2,T]$. On \([0,T/2]\), we have \(T-s\ge T/2\), and hence $\big(1+\frac1{\sqrt{T-s}}\big)^2 \le C(T)$. Moreover, since \(\psi_\gamma\) is an eigenfunction of \(L^\gamma\) with eigenvalue \(\gamma^2/2\), and \(p_s^\gamma(x,y)\) is the corresponding fundamental solution, we have
\begin{align}\label{GroundStateIdentity}
\int_{\mathbb R^3}
p_s^\gamma(x,y)\psi_\gamma(y)\,dy
=
e^{\gamma^2s/2}\psi_\gamma(x).
\end{align}
Therefore, we obtain
\begin{align}
\int_0^{T/2}
\int_{\{|y|\le \sqrt{T-s}\}}
e^{-\frac{\gamma^2s}{2}}
p_s^\gamma(x,y)\psi_\gamma(y)
\left(1+\frac1{\sqrt{T-s}}\right)^2
\,dy\,ds
&\le
C
\int_0^{T/2}
e^{-\frac{\gamma^2s}{2}}
\int_{\mathbb R^3}
p_s^\gamma(x,y)\psi_\gamma(y)
\,dy\,ds
\nonumber\\
&\le
C\psi_\gamma(x). \nonumber
\end{align}

On \([T/2,T]\), using~\eqref{EigenFunctionEigenValue} and~\eqref{PtgammaBounds}, we have $p_s^\gamma(x,y)\psi_\gamma(y) \le C \big( |y|^{-1} + |y|^{-2} \big)$, where recall that \(C=C(\gamma,T,x)>0\). Hence,
\[
\int_{\{|y|\le \sqrt{T-s}\}}
p_s^\gamma(x,y)\psi_\gamma(y)\,dy
\le
C
\int_{\{|y|\le \sqrt{T-s}\}}
\left(
\frac1{|y|}
+
\frac1{|y|^2}
\right)dy
\le
C\sqrt{T-s},
\]
where the last inequality follows by passing to polar coordinates and using \(\sqrt{T-s}+2\le C(T)\) on \([T/2,T]\). Consequently,
\[
I_1
\le
C\psi_\gamma(x)
+
C
\int_{T/2}^{T}
\left(1+\frac1{\sqrt{T-s}}\right)^2
\sqrt{T-s}\,ds
<\infty .
\]

\noindent \textit{Estimating \(I_2\)}. Again we split the time integral into $[0,T/2]$ and $[T/2,T]$. On \([0,T/2]\), we have \(|y|>\sqrt{T-s}\ge \sqrt{T/2}\), and therefore $(1+|y|^{-1})^2\le C(T)$. Hence,~\eqref{GroundStateIdentity} implies
\[
\int_0^{T/2}
\int_{\{|y| > \sqrt{T-s}\}}
e^{-\frac{\gamma^2s}{2}}
p_s^\gamma(x,y)\psi_\gamma(y)
\left(1+\frac1{|y|}\right)^2
\,dy\,ds
\le
C \psi_\gamma(x) < \infty.
\]
On \([T/2,T]\), using~\eqref{EigenFunctionEigenValue} and~\eqref{PtgammaBounds}, there exist constants
\(C=C(\gamma,T,x)>0\) and \(c=c(T)>0\) such that
\[
e^{-\frac{\gamma^2s}{2}}p_s^\gamma(x,y)\psi_\gamma(y)
\le
C
e^{-c_T|y|^2}
\left(
\frac1{|y|}
+
\frac1{|y|^2}
\right).
\]
Therefore, we have the first inequality below.
\begin{align} \label{SecondInt}
\int_{\{|y|>\sqrt{T-s}\}}
e^{-\frac{\gamma^2s}{2}}
p_s^\gamma(x,y)\psi_\gamma(y)
\left(1+\frac1{|y|}\right)^2
\,dy
&\le
C
\int_{\{|y|>\sqrt{T-s}\}}
e^{-c_T|y|^2}
\left(
\frac1{|y|}
+
\frac1{|y|^2}
\right)
\left(1+\frac1{|y|}\right)^2
\,dy \nonumber \\
&=
C
\int_{\sqrt{T-s}}^\infty
e^{-c_Tr^2}
(r+1)
\left(1+\frac1r\right)^2
\,dr \nonumber \\
&\le
C
\left(
\frac{C}{\sqrt{T-s}}
+
\int_{\sqrt T}^\infty
e^{-c_Tr^2}(r+1)
\left(1+\frac1r\right)^2
\,dr
\right) \nonumber \\
&\le
\frac{C}{\sqrt{T-s}} 
\end{align}
The second inequality follows by splitting the integral over
\([\sqrt{T-s},\sqrt T]\) and \([\sqrt T,\infty)\), and observing that
\[
\int_{\sqrt{T-s}}^{\sqrt T}
e^{-c_Tr^2}(r+1)\left(1+\frac1r\right)^2\,dr
\le
C\int_{\sqrt{T-s}}^{\sqrt T}\frac{dr}{r^2}
=
\frac{C}{\sqrt{T-s}}-\frac{C}{\sqrt T}
\le
\frac{C}{\sqrt{T-s}}.
\]
For the final inequality in~\eqref{SecondInt}, we use the observation
\[
\int_{\sqrt T}^{\infty}
e^{-c_Tr^2}(r+1)\left(1+\frac1r\right)^2\,dr
\le
C
\le
\frac{C}{\sqrt{T-s}},
\]
where the last inequality follows from the fact that
\(\sqrt{T-s}\le\sqrt{T/2}\) for \(s\in[T/2,T]\). Since the right-hand side of~\eqref{SecondInt} is integrable over \(s\in[T/2,T]\), combining this estimate with the one on \([0,T/2]\) yields \(I_2<\infty\).

Since \(0\le S_t^\gamma\le 1\), Lemma~\ref{LemmaHatDriftBound} implies that the
bound~\eqref{BoundProductSDriftBound} holds for both
\(\bigl|S_t^\gamma(x)\widehat b_t^\gamma(x)\bigr|\) and
\(\bigl|\widehat b_t^\gamma(x)\bigr|\), for all \(0<t\le T\). Therefore, the
same estimate holds with \(S_{T-s}^{\gamma}(y)b^\gamma(y)\) replaced by either
\(S_{T-s}^{\gamma}(y)\widehat b_{T-s}^{\gamma}(y)\) or
\(\widehat b_{T-s}^{\gamma}(y)\).
\end{proof}

\section{Proof of Lemma~\ref{LemGroundStateVisitation}} \label{SubsectionLemSubMART}

\begin{proof}
\noindent Part (i). Recall that
\(\widehat{\mathbb P}_{x}^{T,\gamma}
:=
\mathbb P_x^{T,\gamma}
[\,\cdot\,|\,\tau>T]\). Recall from the proof of part~(i) of Corollary~\ref{CorGroundStateVisitation} that
\(\{\mathbf S_{t\wedge\tau}^{T,\gamma}\}_{t\in[0,T]}\)
is a bounded \(\mathbb P_x^{T,\gamma}\)-martingale. Hence, for every
\(t\in[0,T]\), we have
\begin{align}
\mathbb E_x^{T,\gamma}
\left[
\mathbf 1_{\{\tau>T\}}
\,\big|\,
\mathcal F_t^{T,x}
\right]
&\overset{\eqref{1OComp}}{=}
\mathbb E_x^{T,\gamma}
\left[
\mathbf S_{\tau\wedge T}^{T,\gamma}
\,\big|\,
\mathcal F_t^{T,x}
\right]
=
\mathbf S_{t\wedge\tau}^{T,\gamma}. \nonumber
\end{align}
Thus, the density process of
\(\widehat{\mathbb P}_{x}^{T,\gamma}\) with respect to
\(\mathbb P_x^{T,\gamma}\) is
\begin{align} \label{Zdefine}
Z_t^{T,\gamma}
:=
\frac{
d\widehat{\mathbb P}_{x}^{T,\gamma}
}{
d\mathbb P_x^{T,\gamma}
}
\bigg|_{\mathcal F_t^{T,x}}
=
\frac{
\mathbb E_x^{T,\gamma}
\left[
\mathbf 1_{\{\tau>T\}}
\,\big|\,
\mathcal F_t^{T,x}
\right]
}{
\mathbb P_x^{T,\gamma}[\tau>T]
}
=
\frac{
\mathbf S_{t\wedge\tau}^{T,\gamma}
}{
S_T^\gamma(x)
},
\qquad
0\le t\le T,
\end{align}
where the last equality also uses~\eqref{PartOneFinal}. Consequently, the process
\(\{Z_t^{T,\gamma}\}_{t\in[0,T]}\) is a bounded continuous
\(\mathbb P_x^{T,\gamma}\)-martingale, with
\(Z_0^{T,\gamma}=1\). Hence, for \(0\le t\le T\), we have
\begin{align}\label{Emmy}
dZ_t^{T,\gamma}
=
\frac{1}{S_T^\gamma(x)}
\,d\mathbf S_{t\wedge\tau}^{T,\gamma}
=
\mathbf 1_{\{t<\tau\}}
\frac{1}{S_T^\gamma(x)}
\nabla S_{T-t}^{\gamma}(X_t)\cdot dW_t^{T,\gamma}
=
Z_t^{T,\gamma}
\nabla\log S_{T-t}^{\gamma}(X_t)\cdot dW_t^{T,\gamma},
\end{align}
where the second equality follows from Proposition~\ref{PropSubMart}
and the martingale representation~\eqref{MartingaleFormula}, together
with the fact that \(\mathbf A_{t\wedge\tau}^{T,\gamma}=0\) for
\(t\in[0,T]\). The last equality uses the identity
\(\mathbf S_{t\wedge\tau}^{T,\gamma}
=
\mathbf 1_{\{t<\tau\}}S_{T-t}^{\gamma}(X_t)\), which may be proved
exactly as~\eqref{1OComp}, together with the above representation of
\(Z_t^{T,\gamma}\). Under \(\widehat{\mathbb P}_{x}^{T,\gamma}\), consider the process
\begin{align}\label{WidehatMartingale}
\widehat M_t^{T,\gamma}
:=
X_t-X_0
-
\int_0^t
\widehat b_{T-s}^{\gamma}(X_s)\,ds,
\qquad
t\in[0,T].
\end{align}
Since the conditional law $\widehat{\mathbb P}_{x}^{T,\gamma}$ has the density \(Z_s^{T,\gamma}\) defined
in~\eqref{Zdefine}, Tonelli's theorem gives
\begin{align}\label{RegDriftSecondMoment}
\widehat{\mathbb E}_{x}^{T,\gamma}
\left[
\int_0^T
\left|
\widehat b_{T-s}^{\gamma}(X_s)
\right|^2
\,ds
\right]
&=
\frac{1}{S_T^\gamma(x)}
\int_0^T
\mathbb E_x^{T,\gamma}
\left[
\mathbf S_{s\wedge\tau}^{T,\gamma}
\left|
\widehat b_{T-s}^{\gamma}(X_s)
\right|^2
\right]
\,ds
\nonumber\\
&\le
\frac{1}{S_T^\gamma(x)}
\int_0^T
\mathbb E_x^{T,\gamma}
\left[
S_{T-s}^{\gamma}(X_s)
\left|
\widehat b_{T-s}^{\gamma}(X_s)
\right|^2
\right]
\,ds
\nonumber\\
&\le
\frac{1}{S_T^\gamma(x)}
\int_0^T
\int_{\mathbb R^d}
q_s^\gamma(x,y)
\left|
\widehat b_{T-s}^{\gamma}(y)
\right|^2
\,dy\,ds
<
\infty. 
\end{align}
Here, the first inequality uses
\(\mathbf S_{s\wedge\tau}^{T,\gamma}
=
\mathbf 1_{\{s<\tau\}}S_{T-s}^{\gamma}(X_s)
\le
S_{T-s}^{\gamma}(X_s)\).
The second inequality uses that \(q_s^\gamma(x,y)\) is the transition
density of \(X_s\) under \(\mathbb P_x^{T,\gamma}\), together with
\(0\le S_{T-s}^{\gamma}(y)\le1\). The finiteness follows directly from
Lemma~\ref{LemSquareIntegrability} with \(\widehat b_{T-s}^{\gamma}(y)\).  Consequently, H\"older's inequality yields that \(\int_0^T |\widehat b_{T-s}^{\gamma}(X_s)|\,ds<\infty\), \(\widehat{\mathbb P}_{x}^{T,\gamma}\)-almost surely, and hence the process \(\widehat M=\{\widehat M_t^{T,\gamma}\}_{t\in[0,T]}\) is well defined.

To show that $\widehat M$ is a \(\widehat{\mathbb P}_{x}^{T,\gamma}\)-martingale, let \(0\le t\le T\). Applying It\^o's product formula to the
\(\mathbb R^d\)-valued process \(\{Z_t^{T,\gamma}\widehat M_t^{T,\gamma}\}_{t\in[0,T]}\), we obtain
\begin{align}\label{ZWhat}
d\bigl(Z_t^{T,\gamma}\widehat M_t^{T,\gamma}\bigr)
&=
\widehat M_t^{T,\gamma}\,dZ_t^{T,\gamma}
+
Z_t^{T,\gamma}\,d\widehat M_t^{T,\gamma}
+
dZ_t^{T,\gamma}\,d\widehat M_t^{T,\gamma}.
\end{align}
Since \(Z_t^{T,\gamma}=0\) for \(t\in[\tau,T]\), it suffices to analyze
\eqref{ZWhat} on the stochastic interval \([0,\tau)\). On this interval, substituting the SDE~\eqref{CanonicalGSDSDE} for \(X\) into
\eqref{WidehatMartingale}, and using $\nabla\log S_{T-t}^{\gamma}(X_t) = \widehat b_{T-t}^{\gamma}(X_t)-b^\gamma(X_t)$, which follows from~\eqref{GradS}, we obtain $d\widehat M_t^{T,\gamma}
=
dW_t^{T,\gamma}
-
\nabla\log S_{T-t}^{\gamma}(X_t)\,dt$. Consequently, using~\eqref{Emmy}, we obtain
\[
dZ^{T,\gamma} \, d\widehat M^{T,\gamma}
=
Z_t^{T,\gamma}
\nabla\log S_{T-t}^{\gamma}(X_t)\,dt
=
Z_t^{T,\gamma}\,dW_t^{T,\gamma}
-
Z_t^{T,\gamma}\,d\widehat M_t^{T,\gamma},
\]
where the second equality again uses the expression for $d\widehat M_t^{T,\gamma}$. Substituting this into~\eqref{ZWhat}, we obtain, on
\([0,\tau)\),
\[
d\bigl(Z_t^{T,\gamma}\widehat M_t^{T,\gamma}\bigr)
=
\widehat M_t^{T,\gamma}\,dZ_t^{T,\gamma}
+
Z_t^{T,\gamma}\,dW_t^{T,\gamma}.
\]
Since \(Z_t^{T,\gamma}\widehat M_t^{T,\gamma}=0\) for
\(t\in[\tau,T]\), the process
\(\{Z_t^{T,\gamma}\widehat M_t^{T,\gamma}\}_{t\in[0,T]}\) is constant
on \([\tau,T]\). It therefore follows that
\(\{Z_t^{T,\gamma}\widehat M_t^{T,\gamma}\}_{t\in[0,T]}\) is a local
\(\mathbb P_x^{T,\gamma}\)-martingale. Since
\(0\le Z_t^{T,\gamma}\le 1/S_T^\gamma(x)\), it follows from
\eqref{WidehatMartingale}, the elementary inequality
\(|a+b+c|^2\le3(|a|^2+|b|^2+|c|^2)\), and H\"older's inequality that
\begin{align} \label{ZMUniformBound}
\sup_{0\le t\le T}
\mathbb E_x^{T,\gamma}
\left[
\left|
Z_t^{T,\gamma}\widehat M_t^{T,\gamma}
\right|^2
\right]
&\le
\frac{3}{S_T^\gamma(x)^2}
\Bigg(
\sup_{0\le t\le T}
\mathbb E_x^{T,\gamma}
\left[
|X_t|^2
\right]
+
|x|^2
+
T\,
\mathbb E_x^{T,\gamma}
\left[
\int_0^T
\left|
\widehat b_{T-s}^{\gamma}(X_s)
\right|^2
\,ds
\right]
\Bigg). 
\end{align}
The last term in~\eqref{ZMUniformBound} is finite by the final estimate in~\eqref{RegDriftSecondMoment}, since
\[
\mathbb E_x^{T,\gamma}
\left[
\int_0^T
\left|
\widehat b_{T-s}^{\gamma}(X_s)
\right|^2
\,ds
\right]
=
\int_0^T
\int_{\mathbb R^d}
q_s^\gamma(x,y)
\left|
\widehat b_{T-s}^{\gamma}(y)
\right|^2
\,dy\,ds
<
\infty.
\]
Moreover, the first term in~\eqref{ZMUniformBound} is finite when \(d=3\) by the
uniform second-moment estimate established in the proof of
Lemma~\ref{LemmaL2IncrementY}. Thus, it remains to control this term
when \(d=2\). For this purpose, recall from the proof of
Lemma~\ref{LemmaCP} that, under the parameter identification
\(\beta=2\gamma\) and the time-space transformation
\(\widetilde X_t^\gamma=\sqrt{2}\,Z_{t/2}\), the law
\(\mathbb P_x^\gamma\) of \(\widetilde X\) is obtained from
\(\mathbb P_{x/\sqrt{2}}^{(0),(2\gamma)\downarrow}\), and
\(\mathbb P_x^{T,\gamma}\) is its restriction to the time interval
\([0,T]\). Hence, we have the equality below.
\begin{align}
\sup_{0\le t\le T}
\mathbb E_x^{T,\gamma}
\left[
\frac{1}{
K_0\left(\sqrt{2\gamma}|X_t|\right)
}
\right]
&=
\sup_{0\le r\le T/2}
\mathbb E_{x/\sqrt{2}}^{(0),(2\gamma)\downarrow}
\left[
\frac{1}{
K_0\left(\sqrt{4\gamma}|Z_r|\right)
}
\right]
<
\infty. \nonumber
\end{align}
The finiteness follows from~\cite[Eq.~(4.103)]{Chen2} with
\(\beta=2\gamma\), \(z_0=x/\sqrt{2}\), and \(t_0=T/2\). Moreover, by
the asymptotic behaviors~\eqref{AsympBesselOrigin}--\eqref{AsympBesselInfinity}, we have \(\sup_{r>0}r^2K_0\left(\sqrt{2\gamma}r\right)<\infty\).
Consequently, there exists \(C=C(\gamma)>0\) such that
\(|y|^2\le C/K_0\left(\sqrt{2\gamma}|y|\right)\) for
\(y\in\mathbb R^2\setminus\{0\}\). Therefore,
\[
\sup_{0\le t\le T}
\mathbb E_x^{T,\gamma}
\left[
|X_t|^2
\right]
\le
C
\sup_{0\le t\le T}
\mathbb E_x^{T,\gamma}
\left[
\frac{1}{
K_0\left(\sqrt{2\gamma}|X_t|\right)
}
\right]
<
\infty.
\]
Thus, in either dimension, \(\{Z_t^{T,\gamma}\widehat M_t^{T,\gamma}\}_{t\in[0,T]}\) is an
\(L^2\)-bounded local \(\mathbb P_x^{T,\gamma}\)-martingale, and
therefore a \(\mathbb P_x^{T,\gamma}\)-martingale.

Next, we deduce that the process \(\{\widehat M_t^{T,\gamma}\}_{t\in[0,T]}\) is a \(\widehat{\mathbb P}_{x}^{T,\gamma}\)-martingale. For \(0\le s<t\le T\), we have the first equality below by Bayes' rule.
\begin{align}
\widehat{\mathbb E}_{x}^{T,\gamma}
\left[
\widehat M_t^{T,\gamma}
\,\middle|\,
\mathcal F_s^{T,x}
\right]
&=
\frac{
\mathbb E_x^{T,\gamma}
\left[
Z_T^{T,\gamma}\widehat M_t^{T,\gamma}
\,\middle|\,
\mathcal F_s^{T,x}
\right]
}{
\mathbb E_x^{T,\gamma}
\left[
Z_T^{T,\gamma}
\,\middle|\,
\mathcal F_s^{T,x}
\right]
}
\nonumber\\
&=
\frac{
\mathbb E_x^{T,\gamma}
\left[
\mathbb E_x^{T,\gamma}
\left[
Z_T^{T,\gamma}
\,\middle|\,
\mathcal F_t^{T,x}
\right]
\widehat M_t^{T,\gamma}
\,\middle|\,
\mathcal F_s^{T,x}
\right]
}{
\mathbb E_x^{T,\gamma}
\left[
Z_T^{T,\gamma}
\,\middle|\,
\mathcal F_s^{T,x}
\right]
}
\nonumber\\
&=
\frac{
\mathbb E_x^{T,\gamma}
\left[
Z_t^{T,\gamma}\widehat M_t^{T,\gamma}
\,\middle|\,
\mathcal F_s^{T,x}
\right]
}{
Z_s^{T,\gamma}
}
=
\frac{
Z_s^{T,\gamma}\widehat M_s^{T,\gamma}
}{
Z_s^{T,\gamma}
}
=
\widehat M_s^{T,\gamma} \nonumber
\end{align}
Here, the second equality uses that \(\widehat M_t^{T,\gamma}\) is \(\mathcal F_t^{T,x}\)-measurable, the third uses the martingale property of the density process \(\{Z_t^{T,\gamma}\}_{t\in[0,T]}\) defined in~\eqref{Zdefine}, and the fourth uses that \(\{Z_t^{T,\gamma}\widehat M_t^{T,\gamma}\}_{t\in[0,T]}\) is a
\(\mathbb P_x^{T,\gamma}\)-martingale. Since \(\widehat{\mathbb P}_{x}^{T,\gamma}[\tau>T]=1\) and \(s<T\), we have \(s<\tau\) and hence \(Z_s^{T,\gamma}>0\), \(\widehat{\mathbb P}_{x}^{T,\gamma}\)-almost surely. Therefore, the above ratios are well defined. Therefore, \(\{\widehat M_t^{T,\gamma}\}_{t\in[0,T]}\) is a
\(\widehat{\mathbb P}_{x}^{T,\gamma}\)-martingale. \vspace{.2cm}

\noindent Part (ii). Since \(\rho\) is a stopping time taking values in
\([0,T]\) and the density process
\(\{Z_t^{T,\gamma}\}_{t\in[0,T]}\) given in~\eqref{Zdefine} is a bounded
\(\mathbb P_x^{T,\gamma}\)-martingale, the optional sampling theorem yields
\(Z_\rho^{T,\gamma}
=
\mathbb E_x^{T,\gamma}
\big[
Z_T^{T,\gamma}
\,\big|\,
\mathcal F_\rho^{T,x}
\big]\).
Consequently, for every \(A\in\mathcal F_\rho^{T,x}\), we have the final
equality below.
\begin{align}\label{DensityAtRho}
\widehat{\mathbb P}_x^{T,\gamma}[A]
&=
\mathbb E_x^{T,\gamma}
\left[
Z_T^{T,\gamma}\mathbf 1_A
\right]
=
\mathbb E_x^{T,\gamma}
\left[
\mathbb E_x^{T,\gamma}
\left[
Z_T^{T,\gamma}
\,\middle|\,
\mathcal F_\rho^{T,x}
\right]
\mathbf 1_A
\right]
=
\mathbb E_x^{T,\gamma}
\left[
Z_\rho^{T,\gamma}\mathbf 1_A
\right].
\end{align}
Thus, \(Z_\rho^{T,\gamma}\) is the density of
\(\widehat{\mathbb P}_x^{T,\gamma}\) with respect to
\(\mathbb P_x^{T,\gamma}\) on \(\mathcal F_\rho^{T,x}\). Moreover,
by~\eqref{Zdefine} together with
\(\mathbf S_{t\wedge\tau}^{T,\gamma}
=
\mathbf 1_{\{t<\tau\}}S_{T-t}^{\gamma}(X_t)\), we obtain
\begin{align}\label{DensityAtRhoExplicit}
Z_\rho^{T,\gamma}
=
\mathbf 1_{\{\rho<\tau\}}
\frac{
S_{T-\rho}^\gamma(X_\rho)
}{
S_T^\gamma(x)
},
\qquad
\mathbb P_x^{T,\gamma}\text{-almost surely}.
\end{align}
Since \(\widehat{\mathbb P}_x^{T,\gamma}[\tau>T]=1\) and
\(\{\tau>T\}\subseteq\{\rho<\tau\}\), because \(\rho\) takes values in
\([0,T]\), it follows that
\(\widehat{\mathbb P}_x^{T,\gamma}[\rho<\tau]=1\). On
\(\{\rho<\tau\}\), we have \(X_\rho\neq0\), and therefore
\(S_{T-\rho}^\gamma(X_\rho)>0\). On \(\{\rho=T\}\), we have
\(S_0^\gamma(X_T)=1\). Hence
\(\bigl(S_{T-\rho}^\gamma(X_\rho)\bigr)^{-1}\) is finite
\(\widehat{\mathbb P}_x^{T,\gamma}\)-almost surely. Moreover,
\begin{align}\label{InverseSNormalization}
\widehat{\mathbb E}_x^{T,\gamma}
\left[
\bigl(S_{T-\rho}^\gamma(X_\rho)\bigr)^{-1}
\right]
&=
\widehat{\mathbb E}_x^{T,\gamma}
\left[
\mathbf 1_{\{\rho<\tau\}}
\bigl(S_{T-\rho}^\gamma(X_\rho)\bigr)^{-1}
\right]
\nonumber\\
&=
\mathbb E_x^{T,\gamma}
\left[
Z_\rho^{T,\gamma}
\mathbf 1_{\{\rho<\tau\}}
\bigl(S_{T-\rho}^\gamma(X_\rho)\bigr)^{-1}
\right]
=
\frac{1}{S_T^\gamma(x)}
\mathbb P_x^{T,\gamma}[\rho<\tau]
\in(0,\infty).
\end{align}
Here, the first equality uses that
\(\widehat{\mathbb P}_x^{T,\gamma}[\rho<\tau]=1\), the second equality
follows from~\eqref{DensityAtRho}, and the final equality follows
from~\eqref{DensityAtRhoExplicit}. The above quantity is finite and
strictly positive since \(S_T^\gamma(x)>0\) and
\(\mathbb P_x^{T,\gamma}[\rho<\tau]>0\). Thus, the path measure appearing
in the statement is a well-defined probability measure.

Now let \(A\in\mathcal F_\rho^{T,x}\). Then, using~\eqref{DensityAtRho}
and~\eqref{InverseSNormalization}, we obtain the first equality below.
\begin{align}\label{ConditionalLawAtRho}
\frac{
\widehat{\mathbb E}_x^{T,\gamma}
\left[
\mathbf 1_A
\bigl(S_{T-\rho}^\gamma(X_\rho)\bigr)^{-1}
\right]
}{
\widehat{\mathbb E}_x^{T,\gamma}
\left[
\bigl(S_{T-\rho}^\gamma(X_\rho)\bigr)^{-1}
\right]
}
&=
\frac{
S_T^\gamma(x)\,
\mathbb E_x^{T,\gamma}
\left[
Z_\rho^{T,\gamma}
\mathbf 1_A
\mathbf 1_{\{\rho<\tau\}}
\bigl(S_{T-\rho}^\gamma(X_\rho)\bigr)^{-1}
\right]
}{
\mathbb P_x^{T,\gamma}[\rho<\tau]
}
\nonumber\\
&=
\frac{
\mathbb P_x^{T,\gamma}
\left[
A\cap\{\rho<\tau\}
\right]
}{
\mathbb P_x^{T,\gamma}[\rho<\tau]
}
=
\widehat{\mathbb P}_x^{T,\gamma,\rho}[A]
\end{align}
Here, the second equality uses~\eqref{DensityAtRhoExplicit}. Thus, the
restrictions to \(\mathcal F_\rho^{T,x}\) of
\(\widehat{\mathbb P}_x^{T,\gamma,\rho}\) and of the path measure
\[
\frac{
\bigl(S_{T-\rho}^{\gamma}(X_\rho)\bigr)^{-1}
}{
\widehat{\mathbb E}_x^{T,\gamma}
\left[
\bigl(S_{T-\rho}^{\gamma}(X_\rho)\bigr)^{-1}
\right]
}
\,\widehat{\mathbb P}_{x}^{T,\gamma}
\]
coincide. Finally, since the stopped coordinate process
\(\{X_{t\wedge\rho}\}_{t\in[0,T]}\) is
\(\mathcal F_\rho^{T,x}\)-measurable as a
\(C([0,T];\mathbb R^d)\)-valued random variable, it follows
from~\eqref{ConditionalLawAtRho} that its distribution under
\(\widehat{\mathbb P}_x^{T,\gamma,\rho}\) agrees with its distribution
under the path measure displayed above. 
\end{proof}

\begin{appendix}

\section{Proofs of lemmas} \label{AppProofsLemmas}

\subsection{Proof of Lemma~\ref{LemmaCP}} \label{ProofAppLemmaCP}

\begin{proof} 
\noindent \(1^\circ\) (\(d=2\)). Let
\(Z=\{Z_t\}_{t\in[0,\infty)}\) be the complex-valued diffusion
constructed in~\cite[Eq.~(2.15) and Prop.~2.6]{Chen2} under
\(\mathbb P_{x/\sqrt{2}}^{(0),\beta\downarrow}\). By~\cite[Thm.~2.10,
Eq.~(2.21)]{Chen2}, the process
\(\{\sqrt{2}\,Z_t\}_{t\ge0}\) has transition density
\begin{align}\label{ZtransDens}
r_t^\beta(z,y)
:=
e^{-\beta t}
\frac{
K_0\bigl(\sqrt{\beta}|y|\bigr)
}{
K_0\bigl(\sqrt{\beta}|z|\bigr)
}
P_t^\beta(z,y).
\end{align}
Identifying
\(\mathbb C\) with \(\mathbb R^2\) in the usual way, we regard \(Z\) as
an \(\mathbb R^2\)-valued continuous process and
\(\mathbb P_{x/\sqrt{2}}^{(0),\beta\downarrow}\) as its law on
\(C([0,\infty),\mathbb R^2)\). Moreover, the parameter \(\beta\) is
related to the present parameter \(\gamma\) by \(\beta=2\gamma\), and
the time change \(t\mapsto t/2\) converts their normalization of the
Laplacian into the present normalization \(\frac12\Delta\). Therefore,
\(\widetilde X_t^\gamma:=\sqrt{2}\,Z_{t/2}\), for
\(t\in[0,\infty)\), is a time-homogeneous Markov process with law
\(\mathbb P_x^\gamma:=
\mathbb P_{x/\sqrt{2}}^{(0),(2\gamma)\downarrow}
\circ\Phi^{-1}\) on \(C([0,\infty),\mathbb R^2)\), where
\[
\Phi:
C([0,\infty),\mathbb R^2)
\longrightarrow
C([0,\infty),\mathbb R^2),
\qquad
(\Phi\omega)(t)
:=
\sqrt{2}\,\omega(t/2).
\]
Under the changes \(\beta=2\gamma\) and \(t\mapsto t/2\), the integral
kernel \(P_t^\beta(z,y)\) appearing in~\cite[Eqs.~(2.17)--(2.19)]{Chen2}
satisfies \(P_{t/2}^{2\gamma}(z,y)=p_t^\gamma(z,y)\). Substituting this identity into~\eqref{ZtransDens}, we obtain
\(r_{t/2}^{2\gamma}(z,y)=q_t^\gamma(z,y)\). Hence, the transition
density of \(\widetilde X^\gamma\) under \(\mathbb P_x^\gamma\) is
\(q_t^\gamma(z,y)\). Define,
\begin{align}
\mathbb P_x^{T,\gamma}
:=
\mathbb P_x^\gamma\circ\pi_T^{-1},
\qquad
\pi_T:
C([0,\infty),\mathbb R^2)
\longrightarrow
C([0,T],\mathbb R^2),
\qquad
\pi_T(\omega):=\omega|_{[0,T]}.  \nonumber
\end{align}
Thus, under \(\mathbb P_x^{T,\gamma}\), the coordinate process \(\{X_t\}_{t\in[0,T]}\) has the same law as \(\{\widetilde X_t^\gamma\}_{t\in[0,T]}\) under \(\mathbb P_x^\gamma\), and hence has the desired properties. \vspace{.2cm}

\noindent \(2^\circ\) (\(d=3\)). By~\cite[Thm.~3.3]{CKMV2}, there
exists a three-dimensional time-homogeneous Markov process
\(X^1=\{X_t^1\}_{t\in[0,\gamma^2T]}\), started from \(\gamma x\), with
transition density
\[
q_t^1(z,y)
=
e^{-t/2}
\frac{\psi_1(y)}{\psi_1(z)}
p_t^1(z,y),
\qquad
t>0,\quad z,y\in\mathbb R^3\setminus\{0\}.
\]
Let \(\mathbf P_{\gamma x}^{\gamma^2T,1}\) denote its law on
\(C([0,\gamma^2T],\mathbb R^3)\). Define
\begin{align}
\mathbb P_x^{T,\gamma}
:=
\mathbb P_{\gamma x}^{\gamma^2T,1}
\circ\Theta_\gamma^{-1},
\qquad
\Theta_\gamma: C([0,\gamma^2T],\mathbb R^3) \longrightarrow
C([0,T],\mathbb R^3),
\qquad
(\Theta_\gamma\omega)(t)
:=
\frac1\gamma\omega(\gamma^2t). \nonumber
\end{align}
Thus, under \(\mathbf P_x^{T,\gamma}\), the coordinate process
\(\{X_t\}_{t\in[0,T]}\) has the same law as
\(\{\gamma^{-1}X_{\gamma^2t}^1\}_{t\in[0,T]}\) under
\(\mathbf P_{\gamma x}^{\gamma^2T,1}\). In particular,
\(\mathbf P_x^{T,\gamma}(X_0=x)=1\), and the coordinate process is a
time-homogeneous Markov process with transition density $\gamma^3
q_{\gamma^2t}^1(\gamma z,\gamma y)= q_t^\gamma(z,y)$.
\end{proof}

\subsection{Proof of Lemma~\ref{lemmaLocalizedSDEAwayOrigin}} \label{ProoflemmaLocalizedSDEAwayOrigin}
For \(d=2\), the result follows directly from~\cite[Prop.~4.2\((4^\circ)\)]{Chen3}; for completeness, we briefly discuss at the end of this section how the parameter relation \(\beta=2\gamma\) and the time change \(t\mapsto t/2\) yield the SDE in the present notation. For \(d=3\), the proof follows the same argument as in the corresponding 3d TMD case treated in~\cite{MianPathwise}. Accordingly, we first identify the special function $\mathbb Q_t^\gamma$ arising in the present 3d GSD setting and establish the estimates and identities needed in that argument. The remaining steps are the same as in the 3d TMD case, and we indicate the corresponding references where they are used. The proof of the Lemma~\ref{lemmaLocalizedSDEAwayOrigin} is at the end of this section.

For \(t>0\) and \(x\in\mathbb R^3\setminus\{0\}\), recall that \(\Psi_t^\gamma(x)=e^{\gamma^2 t/2}\psi_\gamma(x)\) and define $\hat\Psi_t^\gamma(x):=x\Psi_t^\gamma(x)$. Define the map
\(\mathbb Q_t^\gamma:\mathbb R^3\setminus\{0\}\to\mathbb R^3\) by
\begin{align} \label{DefUpsilon1}
\mathbb Q_t^\gamma(x)
:=
\frac{\big(g_t *\hat \Psi_0^\gamma\big)(x)}{\Psi_t^\gamma(x)}
=
\frac{x\big(g_t * \psi_\gamma\big)(x)}{\Psi_t^\gamma(x)} 
+
\frac{t \nabla \big(g_t * \psi_\gamma\big)(x)}{\Psi_t^\gamma(x)}
\end{align}
where the second equality uses that $\nabla g_t(x-y)=\frac{y-x}{t}g_t(x-y)$.

\begin{lemma}\label{LemmaL2IncrementY}
Fix \(T,\gamma>0\) and \(x\in\mathbb R^3\setminus\{0\}\). Then
\begin{align}
\sup_{0<t\le T}
\int_{\mathbb R^3}
q_t^\gamma(x,y)
\left|
\mathbb Q_{T-t}^{\gamma}(y)-\mathbb Q_T^\gamma(x)
\right|^2
\,dy
<
\infty. \nonumber
\end{align}
\end{lemma}

\begin{proof}
Recall~\eqref{EigenFunctionEigenValue} that $\psi_\gamma(z)=C(\gamma) e^{-\gamma|z|}/|z|$, where $C(\gamma)=\sqrt{\gamma /(2\pi)}$. Thus, using the definition~\eqref{DefUpsilon1}, together with
\(\hat\Psi_0^\gamma(z)=z\psi_\gamma(z)\) and
\(\Psi_u^\gamma(y)=e^{\frac{\gamma^2u}{2}}\psi_\gamma(y)\), we obtain
\begin{align}
\left|\mathbb Q_u^\gamma(y)\right|
&=
\frac{e^{-\frac{\gamma^2u}{2}}}{\psi_\gamma(y)}
\left|
\int_{\mathbb R^3}
g_u(y-z)z\psi_\gamma(z)\,dz
\right| \nonumber \\
&\le
\frac{e^{-\frac{\gamma^2u}{2}}}{\psi_\gamma(y)}
\displaystyle\int_{\mathbb R^3}
g_u(y-z)|z|\psi_\gamma(z)\,dz
=
\frac{C(\gamma)\, e^{-\frac{\gamma^2u}{2}}}{\psi_\gamma(y)}
\int_{\mathbb R^3}
g_u(y-z)e^{-\gamma|z|}
\,dz .
\nonumber
\end{align}
Let \(e:=y/|y|\). Then Cauchy--Schwarz gives
\(e\cdot z\le |e||z|=|z|\), which implies the inequality below:
\begin{align}
\int_{\mathbb R^3}
g_u(y-z)e^{-\gamma|z|}
\,dz
&\le
\int_{\mathbb R^3}
g_u(y-z)e^{-\gamma e\cdot z}
\,dz
\nonumber\\
&=
e^{-\gamma e\cdot y}
\int_{\mathbb R^3}
\frac{1}{(2\pi u)^{3/2}}
e^{
-\frac{|z-y|^2}{2u}
-\gamma e\cdot(z-y)
}
\,dz
\nonumber\\
&=
e^{-\gamma e\cdot y+\frac{\gamma^2u}{2}}
\int_{\mathbb R^3}
\frac{1}{(2\pi u)^{3/2}}
e^{
-\frac{|z-y+\gamma u e|^2}{2u}
}
\,dz
=
e^{-\gamma e\cdot y+\frac{\gamma^2u}{2}}
=
e^{\frac{\gamma^2u}{2}}e^{-\gamma|y|}, \nonumber
\end{align}
where the second equality follows by completing the square and using \(|e|=1\), while the final equality uses \(e\cdot y=|y|\). Substituting the above into the previous display yields
\(\left|\mathbb Q_u^\gamma(y)\right|\le |y|\). This, together with the elementary inequality \( |a-b|^2\le 2|a|^2+2|b|^2 \), yields
\[
\int_{\mathbb R^3}
q_t^\gamma(x,y)
\left|
\mathbb Q_{T-t}^{\gamma}(y)-\mathbb Q_T^\gamma(x)
\right|^2dy
\le
2\int_{\mathbb R^3}
q_t^\gamma(x,y)|y|^2dy
+
2|\mathbb Q_T^\gamma(x)|^2 .
\]
Thus, it remains to show that $\sup_{0<t\le T}
\int_{\mathbb R^3}
q_t^\gamma(x,y)|y|^2\,dy
<\infty$. Using~\eqref{DefTransDensity}, we obtain the equality below.
\begin{align}
\int_{\mathbb R^3}
q_t^\gamma(x,y)|y|^2\,dy
&=
\frac{e^{-\gamma^2t/2}}{\psi_\gamma(x)}
\int_{\mathbb R^3}
|y|^2
\psi_\gamma(y)
p_t^\gamma(x,y)
\,dy \nonumber \\
&\le
\frac{e^{-\gamma^2t/2}}{\psi_\gamma(x)}
\int_{\mathbb R^3}
|y|^2
\psi_\gamma(y)
g_t(x-y)
\,dy 
+
\frac{C e^{-\gamma^2t/2}}{\psi_\gamma(x)\sqrt{t}}
\frac{e^{-\frac{|x|^2}{2t}}}{|x|}
\int_{\mathbb R^3}
|y|
\psi_\gamma(y) e^{-\frac{|y|^2}{2t}}
\,dy , \nonumber
\end{align}
where the inequality follows from the estimate~\eqref{PtgammaBounds}. We estimate the
two integrals on the right-hand side uniformly for \(0<t\le T\). \vspace{.2cm}

\noindent \textit{First integral.} Using~\eqref{EigenFunctionEigenValue}, we have
\[
\int_{\mathbb R^3} |y|^2\psi_\gamma(y)g_t(x-y)\,dy
=
\sqrt{\frac{\gamma}{2\pi}}
\int_{\mathbb R^3}
|y|e^{-\gamma |y|}
g_t(x-y)\,dy
\le
\sqrt{\frac{\gamma}{2\pi}}
\int_{\mathbb R^3}
|y| g_t(x-y)\,dy
=
\sqrt{\frac{\gamma}{2\pi}} \mathbb E\!\left[|x+B_t|\right],
\]
where \(B_t\sim N(0,tI_3)\). Moreover,
\[
\mathbb E\!\left[|x+B_t|\right]
\le
|x|+\mathbb E\!\left[|B_t|\right]
\le
|x|
+
\left(
\mathbb E\!\left[|B_t|^2\right]
\right)^{1/2}
=
|x|+\sqrt{3t}
\le
|x|+\sqrt{3T}.
\]
Therefore, $\sup_{0<t\le T}
\int_{\mathbb R^3}
|y|^2\psi_\gamma(y)g_t(x-y)\,dy
<\infty$.\vspace{.2cm}

\noindent \textit{Second integral.} Again using~\eqref{EigenFunctionEigenValue},
\begin{align}
\int_{\mathbb R^3}
|y|\psi_\gamma(y)e^{-\frac{|y|^2}{2t}}
\,dy
=
\sqrt{\frac{\gamma}{2\pi}}
\int_{\mathbb R^3}
e^{-\gamma|y|}
e^{-\frac{|y|^2}{2t}}
\,dy
\le
\sqrt{\frac{\gamma}{2\pi}}
\int_{\mathbb R^3}
e^{-\frac{|y|^2}{2t}}
\,dy
=
\sqrt{\frac{\gamma}{2\pi}}
(2\pi t)^{3/2}. \nonumber
\end{align}
Consequently,
\begin{align}
\frac{C e^{-\gamma^2t/2}}{\psi_\gamma(x)\sqrt{t}}
\frac{e^{-\frac{|x|^2}{2t}}}{|x|}
\int_{\mathbb R^3}
|y|
\psi_\gamma(y)
e^{-\frac{|y|^2}{2t}}
\,dy
\le
C(\gamma,x)\,
t\,
e^{-\gamma^2t/2}
e^{-\frac{|x|^2}{2t}}
\le
C(\gamma,x)T , \nonumber
\end{align}
uniformly for \(0<t\le T\).
\end{proof}

Given \(\gamma, t>0\), the radial symmetry of \(\mathbb Q_t^\gamma\) implies that there exists a unique scalar
function
\(\overline{\mathbb Q}_t^\gamma:(0,\infty)\to\mathbb R\) such that $\mathbb Q_t^\gamma(x)
=
\overline{\mathbb Q}_t^\gamma(|x|)x$ for $x\in\mathbb R^3\setminus\{0\}$. Consequently, the Jacobian matrix
\(\nabla^\dagger\mathbb Q_t^\gamma(x)\) has the form
\begin{align}
\nabla^\dagger\mathbb Q_t^\gamma(x)
=
\lambda_t^\gamma(x)\,P_x
+
\mu_t^\gamma(x)\,(I-P_x), \nonumber
\end{align}
where \(P_x:\mathbb R^3\to\mathbb R^3\) denotes the orthogonal projection onto \(\operatorname{span}\{x\}\), that is $ P_x=\frac{x\otimes x}{|x|^2}$, and the corresponding eigenvalues are
\begin{align}\label{SigmaForm}
\lambda_t^\gamma(x)
:=
\overline{\mathbb Q}_t^\gamma(|x|)
+
|x|\frac{d}{dr}\overline{\mathbb Q}_t^\gamma(r)\big|_{r=|x|},
\qquad
\mu_t^\gamma(x)
:=
\overline{\mathbb Q}_t^\gamma(|x|),
\end{align} 
For a fixed \(x\in\R^3\setminus\{0\}\), let \(e_1:=\frac{x}{|x|}\),
and extend it to an orthonormal basis
\(\{e_1,e_2,e_3\}\) of \(\mathbb R^3\). Since \(P_x\) is the orthogonal projection onto \(\operatorname{span}\{x\}=\operatorname{span}\{e_1\}\), we have \(P_xe_1=e_1\), \(P_xe_2=0\), and \(P_xe_3=0\). Consequently,
\[
\nabla^\dagger\mathbb Q_t^\gamma(x)e_1
=
\lambda_t^\gamma(x)e_1,
\qquad
\nabla^\dagger\mathbb Q_t^\gamma(x)e_i
=
\mu_t^\gamma(x)e_i,
\quad i=2,3.
\]
Hence, with respect to the basis \(\{e_1,e_2,e_3\}\), the Jacobian \(\nabla^\dagger\mathbb Q_t^\gamma(x)\) is a \(3\times 3\) diagonal matrix with diagonal entries \(\lambda_t^\gamma(x)\), \(\mu_t^\gamma(x)\), and \(\mu_t^\gamma(x)\). Since \(\lambda_t^\gamma(x)>0\) and \(\mu_t^\gamma(x)>0\), the matrix
\(\nabla^\dagger\mathbb Q_t^\gamma(x)\) is invertible with
\begin{align}\label{SigmaInverse}
\bigl(\nabla^\dagger\mathbb Q_t^\gamma(x)\bigr)^{-1}
=
\frac{1}{\lambda_t^\gamma(x)}\,P_x
+
\frac{1}{\mu_t^\gamma(x)}\,(I-P_x).
\end{align}

Thus the function $\mathbb Q_t^\gamma(x)$ has the same eigenvalue decomposition and the inverse property as that of the corresponding transformation in the 3D TMD;~\eqref{SigmaForm}--~\eqref{SigmaInverse} correspond to those of~\cite[Eqs.~(3.2)--(3.3)]{MianPathwise}. Moreover, it can be checked that the function $\mathbb Q_t^\gamma(x)$ also satisfies~\cite[Eqs.~(3.4)--(3.6)]{MianPathwise} in the current setting and hence, using the same argument as in the proof of~\cite[Lem.~3.3]{MianPathwise}, we get that for fixed \(T,\gamma>0\) and \(x\in\mathbb R^3\setminus\{0\}\), the \(\mathbb R^3\)-valued process
\[
Y_t^{T,\gamma}
=
\mathbb Q_{T-t}^{\gamma}(X_t),
\qquad 0\le t\le T,
\]
is a continuous square-integrable \(\mathbb P_x^{T,\gamma}\)-martingale with respect to the augmented filtration \(\{\mathcal F_t^{T,x}\}_{t\in[0,T]}\).

Thus the function \(\mathbb Q_t^\gamma(x)\) has the same eigenvalue decomposition and the inverse property as that of the corresponding transformation in the 3D TMD;~\eqref{SigmaForm}--\eqref{SigmaInverse} correspond to those of~\cite[Eqs.~(3.2)--(3.3)]{MianPathwise}. Moreover, it can be checked that the function \(\mathbb Q_t^\gamma(x)\) also satisfies~\cite[Eqs.~(3.4)--(3.6)]{MianPathwise} in the current setting and hence, using the same argument as in the proof of~\cite[Lem.~3.3]{MianPathwise}, we obtain that, for fixed \(T,\gamma>0\) and \(x\in\mathbb R^3\setminus\{0\}\), the \(\mathbb R^3\)-valued process
\[
Y_t^{T,\gamma}
=
\mathbb Q_{T-t}^{\gamma}(X_t),
\qquad
0\le t\le T,
\]
is a continuous local \(\mathbb P_x^{T,\gamma}\)-martingale with respect to the augmented filtration \(\{\mathcal F_t^{T,x}\}_{t\in[0,T]}\). Since \(Y_0^{T,\gamma}=\mathbb Q_T^\gamma(x)\), the elementary inequality
\(|a+b|^2\le2|a|^2+2|b|^2\) together with Minkowski's inequality and Lemma~\ref{LemmaL2IncrementY}, therefore gives
\begin{align}
\sup_{0\le t\le T}
\left(
\mathbb E_x^{T,\gamma}
\left[
\left|
Y_t^{T,\gamma}
\right|^2
\right]
\right)^{1/2}
&\le
\sup_{0\le t\le T}
\left(
\mathbb E_x^{T,\gamma}
\left[
\left|
Y_t^{T,\gamma}-Y_0^{T,\gamma}
\right|^2
\right]
\right)^{1/2}
+
\left|Y_0^{T,\gamma}\right|
<
\infty. \nonumber
\end{align}
Consequently, \(Y^{T,\gamma}\) is a continuous square-integrable
\(\mathbb P_x^{T,\gamma}\)-martingale. Moreover, for every \(u\in\mathbb R^3\), the quadratic variation of the real-valued martingale \(\{u\cdot Y_t^{T,\gamma}\}_{t\in[0,T]}\) is given by
\begin{align}
\big\langle u\cdot Y^{T,\gamma}\big\rangle_t
=
\int_0^t
\big\|\nabla^\dagger\mathbb Q_{t-s}^{\gamma}(X_s)u\big\|_2^2\,ds,
\qquad 0\le t\le T. \nonumber
\end{align}
Consequently, proceeding as in the proof of~\cite[Cor.~3.5]{MianPathwise}, we obtain that the process
\begin{align}\label{DefSBMW}
W_t^{T,\gamma}
:=
\int_0^t
\mathbf 1_{\{X_s\neq0\}}
\big(\nabla^\dagger\mathbb Q_{t-s}^{\gamma}(X_s)\big)^{-1}
\,dY_s^{T,\gamma},
\qquad t \in [0, T],
\end{align}
is a three-dimensional standard Brownian motion with respect to
\(\{\mathcal F_t^{T,x}\}_{t \in [0, T]}\) under \(\mathbb P_x^{T,\gamma}\), and 
\begin{align} \label{dYWrelation}
dY_t^{T,\gamma} = \nabla^\dagger\mathbb Q_{T-t}^{\gamma}(X_t)\,dW_t^{T,\gamma}
\end{align}

\begin{lemma}\label{LemmaInverseTimeDerivative}
For every \(t>0\) and \(x\in\mathbb R^3\setminus\{0\}\),
\begin{align}
\frac{\partial}{\partial t}\mathbb Q_t^\gamma(x)
=
\frac12\Delta\mathbb Q_t^\gamma(x)
+
\big(\nabla^\dagger\mathbb Q_t^\gamma(x)\big)b^\gamma(x). \nonumber
\end{align}
\end{lemma}

\begin{proof}
Recall that $\Psi_t^\gamma(x)=e^{\gamma^2t/2}\psi_\gamma(x)$, and \(g_t(x-y):= (2\pi t)^{-3/2} e^{-|x-y|^2/(2t)}\) is the free heat kernel, we have
\[
\frac{\partial}{\partial t}\Psi_t^\gamma(x)
=
\frac{\gamma^2}{2}\Psi_t^\gamma(x),
\qquad
\frac{\partial}{\partial t}\big(g_t*\hat\Psi_0^\gamma\big)(x)
=
\frac12\Delta\big(g_t*\hat\Psi_0^\gamma\big)(x).
\]
Therefore, using~\eqref{DefUpsilon1}, we obtain
\begin{align*}
\frac{\partial}{\partial t}\mathbb Q_t^\gamma(x)
&=
\frac{
\Psi_t^\gamma(x)\,
\frac{\partial}{\partial t}\big(g_t*\hat\Psi_0^\gamma\big)(x)
-
\big(g_t*\hat\Psi_0^\gamma\big)(x)\,
\frac{\partial}{\partial t}\Psi_t^\gamma(x)
}{
\big(\Psi_t^\gamma(x)\big)^2
} 
=
\frac{
\frac12\Delta\big(g_t*\hat\Psi_0^\gamma\big)(x)
}{
\Psi_t^\gamma(x)
}
-
\frac{\gamma^2}{2}\mathbb Q_t^\gamma(x).
\end{align*}
Since $\big(g_t*\hat\Psi_0^\gamma\big)(x) = \Psi_t^\gamma(x)\mathbb Q_t^\gamma(x)$, we also have
\begin{align*}
\Delta\big(g_t*\hat\Psi_0^\gamma\big)(x)
=
\Delta\big(\Psi_t^\gamma\mathbb Q_t^\gamma\big)(x) 
=
\Psi_t^\gamma(x)\Delta\mathbb Q_t^\gamma(x)
+
2\big(\nabla^\dagger\mathbb Q_t^\gamma(x)\big)\nabla\Psi_t^\gamma(x)
+
\mathbb Q_t^\gamma(x)\Delta\Psi_t^\gamma(x).
\end{align*}
Using that $b^\gamma(x)=\nabla\log\Psi_t^\gamma(x)$ and $\Delta\Psi_t^\gamma(x)=\gamma^2\Psi_t^\gamma(x)$, the above gives
\begin{align}
\frac{
\Delta\big(g_t*\hat\Psi_0^\gamma\big)(x)
}{
\Psi_t^\gamma(x)
}
=
\Delta\mathbb Q_t^\gamma(x)
+
2\big(\nabla^\dagger\mathbb Q_t^\gamma(x)\big)b^\gamma(x)
+
\gamma^2\mathbb Q_t^\gamma(x). \nonumber
\end{align}
Combining the above observations yields the desired identity.
\end{proof}

Recall that, for \(t\ge0\), the function
\(\mathbb Q_t^\gamma:\mathbb R^3\setminus\{0\}\to\mathbb R^3\setminus\{0\}\)
is defined in~\eqref{DefUpsilon1}. Since
\(\nabla^\dagger\mathbb Q_t^\gamma(x)\)
is invertible for every \(x\neq0\), the map \(\mathbb Q_t^\gamma\) is a
\(C^2\)-diffeomorphism of \(\mathbb R^3\setminus\{0\}\). We denote its inverse by
\[
\overleftarrow{\mathbb Q}_t^\gamma:\mathbb R^3\setminus\{0\}\to\mathbb R^3\setminus\{0\}.
\]
In the following lemma use the conventions that for a twice continuously differentiable map \(\phi:\mathbb R^3\to\mathbb R^3\) the second derivative \(\Hess\phi(x)\) is viewed as an element of
\(\mathbb R^3\otimes M^{3\times3}\) such that
\[
(u^\dagger\otimes I_3)\,\Hess\phi(x)
\]
coincides with the Hessian matrix of the scalar-valued function
\(u^\dagger\phi=u\cdot\phi\) for every \(u\in\mathbb R^3\).

\begin{lemma}\label{LemmaInversePhiIdentities}
For every \(t>0\) and \(x\in\mathbb R^3\setminus\{0\}\),
\begin{align}
&\big(\nabla^\dagger \overleftarrow{\mathbb Q}_t^\gamma\big)\big(\mathbb Q_t^\gamma(x)\big)
\,\nabla^\dagger\mathbb Q_t^\gamma(x)
=
I_3 , \label{PDEUp1}
\\
&
\Big(\frac{\partial}{\partial t}\overleftarrow{\mathbb Q}_t^\gamma\Big)
\big(\mathbb Q_t^\gamma(x)\big)
=
- b^\gamma(x)
+
\frac12
\Tr_2
\Big[
\big(I_3\otimes (\nabla^\dagger\mathbb Q_t^\gamma(x))^\dagger\big)
(\Hess \overleftarrow{\mathbb Q}_t^\gamma)
\big(\mathbb Q_t^\gamma(x)\big)
\big(I_3\otimes \nabla^\dagger\mathbb Q_t^\gamma(x)\big)
\Big]. \label{PDEUp2}
\end{align}
\end{lemma}

\begin{proof}
Since \(\overleftarrow{\mathbb Q}_t^\gamma\) is the inverse of \(\mathbb Q_t^\gamma\) on
\(\mathbb R^3\setminus\{0\}\), we have
\[
\overleftarrow{\mathbb Q}_t^\gamma(\mathbb Q_t^\gamma(x))=x,
\qquad x\in\mathbb R^3\setminus\{0\}.
\]
Differentiating with respect to \(x\) and using the chain rule, we obtain~\eqref{PDEUp1}. To prove~\eqref{PDEUp2}, differentiating the identity \(\overleftarrow{\mathbb Q}_t^\gamma(\mathbb Q_t^\gamma(x))=x\) with respect to \(t\) and applying the chain rule yields the first equality below.
\begin{align}\label{I1}
\Big(\frac{\partial}{\partial t}\overleftarrow{\mathbb Q}_t^\gamma\Big)
\big(\mathbb Q_t^\gamma(x)\big)
=\,&
-
\big(\nabla^\dagger\overleftarrow{\mathbb Q}_t^\gamma\big)
\big(\mathbb Q_t^\gamma(x)\big)
\Big(\frac{\partial}{\partial t}\mathbb Q_t^\gamma\Big)(x) \nonumber\\
=\,&
-\big(\nabla^\dagger\mathbb Q_t^\gamma(x)\big)^{-1}
\Big[\frac12\Delta\mathbb Q_t^\gamma(x)
+
\big(\nabla^\dagger\mathbb Q_t^\gamma(x)\big)b^\gamma(x)\Big] \nonumber\\
=\,&
-\frac12
\big(\nabla^\dagger\mathbb Q_t^\gamma(x)\big)^{-1}
\Delta\mathbb Q_t^\gamma(x)
-
b^\gamma(x)
\end{align}
The second equality follows from Lemma~\ref{LemmaInverseTimeDerivative} and~\eqref{PDEUp1}. Also,
\begin{align*}
\big(\Hess\overleftarrow{\mathbb Q}_t^\gamma\big)
\big(\mathbb Q_t^\gamma(x)\big)
=\,&
-
\Big(
\nabla^\dagger\mathbb Q_t^\gamma(x)
\otimes
\big(\nabla^\dagger\mathbb Q_t^\gamma(x)\big)^\dagger
\Big)^{-1}
\big(\Hess\mathbb Q_t^\gamma\big)(x)
\Big(I_3\otimes\nabla^\dagger\mathbb Q_t^\gamma(x)\Big)^{-1} \\
=\,&
-
\Big(
\big(\nabla^\dagger\mathbb Q_t^\gamma(x)\big)^{-1}
\otimes
\big((\nabla^\dagger\mathbb Q_t^\gamma(x))^\dagger\big)^{-1}
\Big)
\big(\Hess\mathbb Q_t^\gamma\big)(x)
\Big(I_3\otimes\big(\nabla^\dagger\mathbb Q_t^\gamma(x)\big)^{-1}\Big),
\end{align*}
where the second equality uses \((B\otimes C)^{-1}=B^{-1}\otimes C^{-1}\) for invertible matrices \(B\) and \(C\). Using the above identity together with
\[
\big(I_3\otimes(\nabla^\dagger\mathbb Q_t^\gamma(x))^\dagger\big)
\Big(
\big(\nabla^\dagger\mathbb Q_t^\gamma(x)\big)^{-1}
\otimes
\big((\nabla^\dagger\mathbb Q_t^\gamma(x))^\dagger\big)^{-1}
\Big)
=
\big(\nabla^\dagger\mathbb Q_t^\gamma(x)\big)^{-1}\otimes I_3 ,
\]
and
\[
\Big(I_3\otimes\big(\nabla^\dagger\mathbb Q_t^\gamma(x)\big)^{-1}\Big)
\big(I_3\otimes\nabla^\dagger\mathbb Q_t^\gamma(x)\big)
=
\Big(I_3I_3\Big)
\otimes
\Big(
\big(\nabla^\dagger\mathbb Q_t^\gamma(x)\big)^{-1}
\nabla^\dagger\mathbb Q_t^\gamma(x)
\Big)
=
I_3\otimes I_3 ,
\]
we obtain the first equality below.
\begin{align}
\Tr_2
\Big[
\big(I_3\otimes(\nabla^\dagger\mathbb Q_t^\gamma(x))^\dagger\big)
(\Hess\overleftarrow{\mathbb Q}_t^\gamma)(\mathbb Q_t^\gamma(x))
\big(I_3\otimes\nabla^\dagger\mathbb Q_t^\gamma(x)\big)
\Big] 
& =
-\Tr_2
\Big[
\Big(
\big(\nabla^\dagger\mathbb Q_t^\gamma(x)\big)^{-1}
\otimes I_3
\Big)
\big(\Hess\mathbb Q_t^\gamma\big)(x)
\Big] \nonumber\\
&\qquad =
-\big(\nabla^\dagger\mathbb Q_t^\gamma(x)\big)^{-1}
\Tr_2\big[(\Hess\mathbb Q_t^\gamma)(x)\big] \nonumber\\
&\qquad =
-\big(\nabla^\dagger\mathbb Q_t^\gamma(x)\big)^{-1}
\big(\Delta\mathbb Q_t^\gamma\big)(x) \nonumber
\end{align}
Substituting the above in~\eqref{I1} yields~\eqref{PDEUp2}.
\end{proof}

\begin{proof}[Proof of Lemma~\ref{lemmaLocalizedSDEAwayOrigin}]
\noindent \(1^\circ\) (\(d=2\)). Recall from the proof of
Lemma~\ref{LemmaCP} that
\(\mathbb P_{x/\sqrt{2}}^{(0),(2\gamma)\downarrow}\) is the law of the
process \(Z=\{Z_t\}_{t\in[0,\infty)}\) and
\(\widetilde X_t^\gamma:=\sqrt{2}\,Z_{t/2}\). By~\cite[Prop.~4.2\((4^\circ)\)]{Chen3}, there
exists a two-dimensional standard Brownian motion
\(\{B_t\}_{t\ge0}\) under
\(\mathbb P_{x/\sqrt{2}}^{(0),(2\gamma)\downarrow}\) such that
\begin{align}
Z_t
=
\frac{x}{\sqrt{2}}
+
B_t
-
2\sqrt{\gamma}
\int_0^t
\frac{
K_1\bigl(2\sqrt{\gamma}|Z_s|\bigr)
}{
K_0\bigl(2\sqrt{\gamma}|Z_s|\bigr)
}
\frac{Z_s}{|Z_s|}
\,ds,
\qquad t\ge0, \nonumber
\end{align}
Here and below, the integrands are assigned the value \(0\) at the
origin; this choice does not affect the integrals, since the zero set
of \(Z\) has zero Lebesgue measure almost surely. Define
\(\widetilde W_t:=\sqrt{2}\,B_{t/2}\) for \(t\ge0\). By Brownian
scaling, \(\widetilde W\) is a two-dimensional standard Brownian
motion under
\(\mathbb P_{x/\sqrt{2}}^{(0),(2\gamma)\downarrow}\). Using the change
of variables \(s=u/2\), we obtain
\begin{align}
\widetilde X_t^\gamma
=
x+\widetilde W_t
-
\sqrt{2\gamma}
\int_0^t
\frac{
K_1\bigl(\sqrt{2\gamma}|\widetilde X_u^\gamma|\bigr)
}{
K_0\bigl(\sqrt{2\gamma}|\widetilde X_u^\gamma|\bigr)
}
\frac{\widetilde X_u^\gamma}
{|\widetilde X_u^\gamma|}
\,du
=
x+\widetilde W_t
+
\int_0^t
b^\gamma(\widetilde X_u^\gamma)
\,du. \nonumber
\end{align}
Using the coordinate process \(\{X_t\}_{t\in[0,T]}\), define
\begin{align} 
W_t^{T,\gamma}
:=
X_t-x-\int_0^t b^\gamma(X_s)\,ds,
\qquad
t\in[0,T]. \nonumber
\end{align}
Since the pair
\(\{(X_t,W_t^{T,\gamma})\}_{t\in[0,T]}\) under
\(\mathbb P_x^{T,\gamma}\) has the same law as
\(\{(\widetilde X_t^\gamma,\widetilde W_t)\}_{t\in[0,T]}\) under
\(\mathbb P_{x/\sqrt{2}}^{(0),(2\gamma)\downarrow}\),
\(W^{T,\gamma}\) is a two-dimensional standard Brownian motion under
\(\mathbb P_x^{T,\gamma}\), with respect to
\(\{\mathcal F_t^{T,x}\}_{t\in[0,T]}\), and the coordinate process
satisfies the desired SDE. \vspace{.2cm}

\noindent \(2^\circ\) (\(d=3\)). Let $W^{T,\gamma}$ be the three-dimensional Brownian motion in~\eqref{DefSBMW}. Recall that the function \(\overleftarrow{\mathbb Q}_t^\gamma\) denotes the inverse of \(\mathbb Q_t^\gamma\), so that $X_t=\overleftarrow{\mathbb Q}_{T-t}^\gamma\big(Y_t^{T,\gamma}\big)$. Since \(Y_t^{T,\gamma}=\mathbb Q_{T-t}^\gamma(X_t)\) and \(X_s\neq0\) for all \(s\in[a,b]\), the function $(t,y)\mapsto \overleftarrow{\mathbb Q}_{T-t}^\gamma(y)$ is \(C^{1,2}\) along the path \(\{(t,Y_t^{T,\gamma}):t\in[a,b]\}\) and It\^o's formula along with~\eqref{PDEUp2} yields the first equality below.
\begin{align*}
dX_t
=&
\big(\nabla^\dagger\overleftarrow{\mathbb Q}_{T-t}^{\gamma}\big)
\big(Y_t^{T,\gamma}\big)\,dY_t^{T,\gamma}
+
b_{T-t}^{\gamma}(X_t)\,dt \nonumber \\
=&
\big(\nabla^\dagger\overleftarrow{\mathbb Q}_{T-t}^{\gamma}\big)
\big(\mathbb Q_{T-t}^{\gamma}(X_t)\big)
\nabla^\dagger\mathbb Q_{T-t}^{\gamma}(X_t)\,dW_t^{T,\gamma}
+
b_{T-t}^{\gamma}(X_t)\,dt
\end{align*}
where the second equality uses~\eqref{dYWrelation} and that \(Y_t^{T,\gamma}=\mathbb Q_{T-t}^{\gamma}(X_t)\). Finally using~\eqref{PDEUp1}, we obtain the desired SDE on $[a,b]$.
\end{proof}

\subsection{Proof of Lemma~\ref{LemmaGaussianGroundStateConvolution}} \label{ProofLemmaGaussianGroundStateConvolution}

\begin{proof}
Set $\nu:=d/2-1$. By~\eqref{DefCompleteBessel}, followed by the change of variables
\(u=\lambda_\gamma a\), we have
\begin{align}
K_\nu\!\left(\sqrt{2\lambda_\gamma}\,|x|\right)
&=
\frac12
\left(
\frac{\sqrt{2\lambda_\gamma}\,|x|}{2}
\right)^\nu
\int_0^\infty
u^{-\nu-1}
e^{-u-\frac{\lambda_\gamma|x|^2}{2u}}
\,du
=
\frac{
|x|^\nu
}{
2(2\lambda_\gamma)^{\nu/2}
}
\int_0^\infty
a^{-d/2}
e^{-\lambda_\gamma a-\frac{|x|^2}{2a}}
\,da. \nonumber
\end{align}
The identity above, together with the form of \(\psi_\gamma\)
in~\eqref{EigenFunctionEigenValueUnified} and
\(g_t(x)=(2\pi t)^{-d/2}e^{-|x|^2/(2t)}\), yields the first two
equalities below.
\begin{align}
(g_t*\psi_\gamma)(x)
&=
\frac{
\sqrt{2\lambda_\gamma}
}{
\pi^{(d-1)/2}
}
\int_{\mathbb R^d}
g_t(x-y)
\frac{
K_\nu\!\left(\sqrt{2\lambda_\gamma}\,|y|\right)
}{
|y|^\nu
}
\,dy
\nonumber\\
&=
\frac{
\sqrt{2\lambda_\gamma}
}{
2\pi^{(d-1)/2}(2\lambda_\gamma)^{\nu/2}
}
\int_0^\infty
e^{-\lambda_\gamma a}
\int_{\mathbb R^d}
\frac{
e^{-\frac{|x-y|^2}{2t}}
}{
(2\pi t)^{d/2}
}
\frac{
e^{-\frac{|y|^2}{2a}}
}{
a^{d/2}
}
\,dy\,da
\nonumber\\
&=
\frac{
\sqrt{2\lambda_\gamma}
}{
2\pi^{(d-1)/2}(2\lambda_\gamma)^{\nu/2}
}
\int_0^\infty
e^{-\lambda_\gamma a}
\frac{
e^{-\frac{|x|^2}{2(t+a)}}
}{
(t+a)^{d/2}
}
\,da
\nonumber\\
&=
\frac{
\sqrt{2\lambda_\gamma}
}{
2\pi^{(d-1)/2}(2\lambda_\gamma)^{\nu/2}
}
e^{\lambda_\gamma t}
\int_t^\infty
e^{-\lambda_\gamma a}
\frac{
e^{-\frac{|x|^2}{2a}}
}{
a^{d/2}
}
\,da
\nonumber\\
&=
\frac{
\sqrt{2\lambda_\gamma}\,
\lambda_\gamma^\nu
}{
2\pi^{(d-1)/2}(2\lambda_\gamma)^{\nu/2}
}
e^{\lambda_\gamma t}
\int_{\lambda_\gamma t}^\infty
u^{-\nu-1}
e^{-u-\frac{\lambda_\gamma|x|^2}{2u}}
\,du. \nonumber
\end{align}
Here, the third equality follows from the semi-group (convolution)
property of \(d\)-dimensional Gaussian kernels, the fourth equality
follows from the change of variables \(a\mapsto a-t\), and the last
equality follows from the change of variables
\(u=\lambda_\gamma a\). Finally, applying~\eqref{DefIncompleteBessel}
to the last integral and recalling that \(\nu=d/2-1\), we obtain the desired result.
\end{proof}

\section{Proof of Proposition~\ref{PropSubMart}} \label{ProofPropSubMart}

\begin{proof}
By Lemma~\ref{LemmaGaussianGroundStateConvolution} and
definition~\eqref{Sfunction}, \(0\le S_r^\gamma(z)\le1\) for every
\(r\in(0,T]\) and \(z\in\mathbb R^d\). Since
\(S_0^\gamma(z)=1\) for every
\(z\in\mathbb R^d\setminus\{0\}\), the process
\(\mathbf S^{T,\gamma}\) is bounded. Moreover,
\((r,z)\mapsto S_r^\gamma(z)\) is continuous on
\((0,T]\times\mathbb R^d\) after setting \(S_r^\gamma(0):=0\) for
\(r>0\), and extends continuously to
\([0,T]\times(\mathbb R^d\setminus\{0\})\) by setting
\(S_0^\gamma(z):=1\). Therefore, since \(X\) is continuous
and \(\mathbb P_x^{T,\gamma}[X_T=0]=0\), the process
\(\mathbf S^{T,\gamma}\) is continuous. To establish the submartingale property for \(d\in\{2,3\}\), let \(0\le s<t\le T\). By the Markov property of \(\mathbb P_x^{T,\gamma}\), we have
\begin{align}
\mathbb E_x^{T,\gamma}
\big[
\mathbf S_t^{T,\gamma}
\,\big|\,
\mathcal F_s^{T,x}
\big]
&=
\int_{\mathbb R^d}
q_{t-s}^\gamma(X_s,y)
S_{T-t}^{\gamma}(y)
\,dy
\nonumber\\
&=
\frac{
e^{-\lambda_\gamma(T-s)}
}{
\psi_\gamma(X_s)
}
\int_{\mathbb R^d}
g_{T-s}^\gamma(X_s,y)
(g_{T-t}*\psi_\gamma)(y)
\,dy
\nonumber\\
&\ge
\frac{
e^{-\lambda_\gamma(T-s)}
}{
\psi_\gamma(X_s)
}
\int_{\mathbb R^d}
g_{t-s}(X_s-y)
(g_{T-t}*\psi_\gamma)(y)
\,dy
\nonumber\\
&=
\frac{
e^{-\lambda_\gamma(T-s)}
}{
\psi_\gamma(X_s)
}
(g_{T-s}*\psi_\gamma)(X_s)
=
S_{T-s}^{\gamma}(X_s)
=
\mathbf S_s^{T,\gamma}. \nonumber
\end{align}
Here the second equality uses~\eqref{DefTransDensity} and~\eqref{Sfunction},
the inequality uses the bound \(p_r^\gamma(z,y)\ge g_r(z-y)\), which follows from~\eqref{FundamentalSolution} since \(h_r^\gamma\ge0\), and the third equality uses the semigroup property of the free heat kernel $g$. Since \(\mathbf S^{T,\gamma}\) is a
bounded continuous submartingale with respect to the augmented
filtration \(\{\mathcal F_t^{T,x}\}_{t\in[0,T]}\), which satisfies the usual conditions, the Doob--Meyer decomposition
theorem~\cite[Thm.~1.4.10]{Karatzas} applies.

Using the elementary inequality
\(\lvert a+b\rvert^2\le2\lvert a\rvert^2+2\lvert b\rvert^2\) and the It\^o isometry, we obtain from~\eqref{MartingaleFormula}
\begin{align}
\mathbb E_x^{T,\gamma}
\left[
\big|\mathbf M_T^{T,\gamma}\big|^2
\right]
&\le
2\big|S_T^\gamma(x)\big|^2
+
2\mathbb E_x^{T,\gamma}
\left[
\int_0^T
\big|\nabla S_{T-s}^{\gamma}(X_s)\big|^2
\,ds
\right]
\nonumber\\
&=
2\big|S_T^\gamma(x)\big|^2
+
2\int_0^T
\int_{\mathbb R^d}
q_s^\gamma(x,y)
\big|\nabla S_{T-s}^{\gamma}(y)\big|^2
\,dy\,ds
<\infty, \nonumber
\end{align}
where the finiteness follows from Lemma~\ref{LemSquareIntegrability}, identity~\eqref{GradS}, and the bound \(0\le S_{T-s}^\gamma\le1\). Consequently, \(\mathbf M^{T,\gamma}\) is a square-integrable martingale.

By the definition of \(\mathbf A^{T,\gamma}\) and
formula~\eqref{MartingaleFormula}, the process
\(\mathbf A^{T,\gamma}
=
\{\mathbf A_t^{T,\gamma}\}_{t\in[0,T]}\), defined by
\(\mathbf A_t^{T,\gamma}
:=
\mathbf S_t^{T,\gamma}-\mathbf M_t^{T,\gamma}\), can be expressed as
\begin{align}\label{ADEF}
\mathbf A_t^{T,\gamma}
&=
S_{T-t}^{\gamma}(X_t)
-
S_T^{\gamma}(x)
-
\int_0^t
\mathbf 1_{\{X_s\neq0\}}
\big(\nabla S_{T-s}^{\gamma}\big)(X_s)
\cdot dW_s^{T,\gamma}
\nonumber\\
&=
S_{T-t}^{\gamma}(X_t)
-
S_T^{\gamma}(x)
-
\int_0^t
\big(\nabla S_{T-s}^{\gamma}\big)(X_s)
\cdot dW_s^{T,\gamma},
\qquad
\mathbb P_x^{T,\gamma}\text{-a.s.},
\end{align}
where the second equality holds since the set of times \(\{s \in [0,T] : X_s = 0\}\) has Lebesgue measure zero \(\mathbb P_x^{T,\gamma}\)-a.s., after setting \(\nabla S_t^\gamma(0):=0\). To prove that \(\mathbf A^{T,\gamma}\) is constant on the excursions of \(X\) away from the origin, fix an interval \([a,b]\subset[0,T]\) such that \(X_s\neq0\) for all \(s\in[a,b]\). Since
\((t,y)\mapsto S_{T-t}^{\gamma}(y)\) is \(C^{1,2}\) away from the
origin, It\^o's formula gives, for \(t\in[a,b]\),
\begin{align*}
d S_{T-t}^{\gamma}(X_t)
=\,&
-\Big(\frac{\partial}{\partial t}
S_{T-t}^{\gamma}\Big)(X_t)\,dt
+
\big(\nabla S_{T-t}^{\gamma}\big)(X_t)\cdot dX_t
+
\frac12
\big(\Delta S_{T-t}^{\gamma}\big)(X_t)\,dt  \\
=\,&
\Bigg[
-\Big(\frac{\partial}{\partial t}
S_{T-t}^{\gamma}\Big)(X_t)
+
\big(\nabla S_{T-t}^{\gamma}\big)(X_t)\cdot
b^\gamma(X_t)
+
\frac12
\big(\Delta S_{T-t}^{\gamma}\big)(X_t)
\Bigg]dt  
+
\big(\nabla S_{T-t}^{\gamma}\big)(X_t)\cdot dW_t^{T,\gamma} \nonumber \\
=\,&
\big(\nabla S_{T-t}^{\gamma}\big)(X_t)\cdot dW_t^{T,\gamma},
\end{align*}
where the second equality uses that the coordinate process \(X\) satisfies the SDE~\eqref{CanonicalGSDSDE} on \([a,b]\) by Lemma~\ref{lemmaLocalizedSDEAwayOrigin}, and the final equality follows because the term inside the square brackets vanishes identically by PDE~\eqref{PDEforS}. Therefore, writing~\eqref{ADEF} in differential form, we obtain
\[
d\mathbf A_t^{T,\gamma}
=
d S_{T-t}^{\gamma}(X_t)
-
\big(\nabla S_{T-t}^{\gamma}\big)(X_t)
\cdot dW_t^{T,\gamma}
=
0,
\qquad t\in[a,b].
\]
Thus, \(\mathbf A^{T,\gamma}\) is constant on every excursion interval \([a,b]\) away from the origin.

It remains to prove that \(\mathbf A^{T,\gamma}\) is increasing. For
\(\varepsilon>0\), let
\(\{\varrho_n^{\downarrow,\varepsilon}\}_{n\in\mathbb N}\) and
\(\{\varrho_n^{\uparrow,\varepsilon}\}_{n\in\mathbb N_0}\) be the
sequences of stopping times defined by
\(\varrho_0^{\uparrow,\varepsilon}=0\) and, for \(n\in\mathbb N\),
\begin{align}
\varrho_n^{\downarrow,\varepsilon}
:=
\inf\big\{
s\in(\varrho_{n-1}^{\uparrow,\varepsilon},T]:
X_s=0
\big\},
\qquad
\varrho_n^{\uparrow,\varepsilon}
:=
\inf\big\{
s\in(\varrho_n^{\downarrow,\varepsilon},T]:
|X_s|=\varepsilon
\big\}, \nonumber
\end{align}
where we interpret \(\inf\emptyset=\infty\). Next, define the process
\(\chi^\varepsilon=\{\chi_s^\varepsilon\}_{s\in[0,T]}\) by
\(\chi_s^\varepsilon=1\) whenever
\(s\in[\varrho_{n-1}^{\uparrow,\varepsilon},
\varrho_n^{\downarrow,\varepsilon})\) for some \(n\in\mathbb N\), and
\(\chi_s^\varepsilon=0\) otherwise. Define the process, for $ 0\le t\le T$, by
\begin{align}\label{MartingaleApproximationEpsilonFormula2}
\mathbf M_t^{T,\gamma,\varepsilon}
:=
S_T^\gamma(x)
+
\int_0^t
\chi_s^\varepsilon
\nabla S_{T-s}^{\gamma}(X_s)
\cdot dW_s^{T,\gamma}
=
S_T^\gamma(x)
+
\sum_{n=1}^{\infty}
\left(
\mathbf S_{t\wedge\varrho_n^{\downarrow,\varepsilon}}^{T,\gamma}
-
\mathbf S_{t\wedge\varrho_{n-1}^{\uparrow,\varepsilon}}^{T,\gamma}
\right),
\end{align}
where the second equality follows by applying It\^o's formula on each interval \([\varrho_{n-1}^{\uparrow,\varepsilon}, \varrho_n^{\downarrow,\varepsilon}]\) and summing over \(n\).
The process \(\mathbf M^{T,\gamma,\varepsilon}\) is a square-integrable martingale (since \(\mathbf M^{T,\gamma}\) is square-integrable and
\(0\le\chi_s^\varepsilon\le1\)) and, by Doob's inequality and the It\^o isometry,
\begin{align*}
\mathbb E_x^{T,\gamma}
\bigg[
\sup_{t\in[0,T]}
\Big|
\mathbf M_t^{T,\gamma,\varepsilon}
-
\mathbf M_t^{T,\gamma}
\Big|^2
\bigg]
&\le
4\,
\mathbb E_x^{T,\gamma}
\bigg[
\Big|
\mathbf M_T^{T,\gamma,\varepsilon}
-
\mathbf M_T^{T,\gamma}
\Big|^2
\bigg]
\\
&=
4\,
\mathbb E_x^{T,\gamma}
\Bigg[
\int_0^T
(1-\chi_s^\varepsilon)
\big|
\nabla S_{T-s}^{\gamma}(X_s)
\big|^2
\,ds
\Bigg]
\\
&\le
4
\int_0^T
\int_{\mathbb R^d}
\mathbf 1_{\{|y|\le\varepsilon\}}
q_s^\gamma(x,y)
\big|
\nabla S_{T-s}^{\gamma}(y)
\big|^2
\,dy\,ds
\longrightarrow0,
\qquad
\varepsilon\downarrow0,
\end{align*}
where in the first equality we have also used that
\(\chi_s^\varepsilon\in\{0,1\}\), and the final inequality follows from \(1-\chi_s^\varepsilon \le\mathbf 1_{\{|X_s|\le\varepsilon\}}\). The convergence to \(0\) as \(\varepsilon\downarrow0\) follows from Lemma~\ref{LemSquareIntegrability}, identity~\eqref{GradS}, and the dominated convergence theorem.

Define the process
\(\mathbf A^{T,\gamma,\varepsilon}
=
\{\mathbf A_t^{T,\gamma,\varepsilon}\}_{t\in[0,T]}\)
by
\[
\mathbf A_t^{T,\gamma,\varepsilon}
:=
\sum_{n=1}^{\infty}
\left(
\mathbf S_{t\wedge\varrho_n^{\uparrow,\varepsilon}}^{T,\gamma}
-
\mathbf S_{t\wedge\varrho_n^{\downarrow,\varepsilon}}^{T,\gamma}
\right)
=
\sum_{n=1}^{\infty}
\mathbf S_{t\wedge\varrho_n^{\uparrow,\varepsilon}}^{T,\gamma}
\mathbf 1_{\{\varrho_n^{\downarrow,\varepsilon}\le t\}},
\]
where the second equality uses that
\(\mathbf S_{\varrho_n^{\downarrow,\varepsilon}}^{T,\gamma}
=
S_{T-\varrho_n^{\downarrow,\varepsilon}}^\gamma(0)=0\)
on \(\{\varrho_n^{\downarrow,\varepsilon}<T\}\), together with
\(\mathbb P_x^{T,\gamma}[X_T=0]=0\). Consequently,
\(\mathbf S_{t\wedge\varrho_n^{\downarrow,\varepsilon}}^{T,\gamma}=0\)
whenever \(\varrho_n^{\downarrow,\varepsilon}\le t\), while
\(\mathbf 1_{\{\varrho_n^{\downarrow,\varepsilon}\le t\}}=0\) otherwise,
\(\mathbb P_x^{T,\gamma}\)-a.s. Let
\(\mathbf M^{T,\gamma,\varepsilon}\) be the process given
in~\eqref{MartingaleApproximationEpsilonFormula2}. Then we have
\[
\mathbf M_t^{T,\gamma}
+
\mathbf A_t^{T,\gamma}
=
\mathbf S_t^{T,\gamma}
=
\mathbf M_t^{T,\gamma,\varepsilon}
+
\mathbf A_t^{T,\gamma,\varepsilon},
\qquad
0\le t\le T.
\]
Consequently, since
\(\mathbf A_t^{T,\gamma,\varepsilon}-\mathbf A_t^{T,\gamma}
=
\mathbf M_t^{T,\gamma}-\mathbf M_t^{T,\gamma,\varepsilon}\)
for \(0\le t\le T\). Therefore, using the uniform $L^2$-convergence of $\mathbf M^{T,\gamma,\varepsilon} \to \mathbf M^{T,\gamma}$ established above, we have $\sup_{t\in[0,T]} \big| \mathbf A_t^{T,\gamma,\varepsilon} - \mathbf A_t^{T,\gamma} \big|$ converges in $L^2(\mathbf P_x^{T,\gamma})$. In particular, along a sequence \(\varepsilon_k\downarrow0\), we may
assume that
\[
\sup_{t\in[0,T]}
\left|
\mathbf A_t^{T,\gamma,\varepsilon_k}
-
\mathbf A_t^{T,\gamma}
\right|
\longrightarrow0,
\qquad
\mathbb P_x^{T,\gamma}\text{-a.s.}
\]
Fix \(\delta\in(0,T)\) and \(\varepsilon\in(0,1]\). For
\(u\in[0,T-\delta]\), the process
\(\mathbf A^{T,\gamma,\varepsilon}\) can decrease only during one of the
intervals
\(
[\varrho_n^{\downarrow,\varepsilon},
\varrho_n^{\uparrow,\varepsilon})
\),
on which \(|X_u|<\varepsilon\). During each such interval, its only
time-varying term is \(\mathbf S_u^{T,\gamma}\), while all the preceding
completed terms are nonnegative and constant. Since \(T-u \in [\delta, T]\) and \(|X_u| < \varepsilon\), it follows that
\[
\mathbf S_u^{T,\gamma}
=
S_{T-u}^{\gamma}(X_u)
\le
\sup_{\substack{r\in[\delta,T]\\ |y|\le\varepsilon}}
S_r^\gamma(y).
\]
Hence, for \(0\le s\le t\le T-\delta\),
\[
\mathbf A_t^{T,\gamma,\varepsilon}
\ge
\mathbf A_s^{T,\gamma,\varepsilon}
-
\sup_{\substack{r\in[\delta,T]\\ |y|\le\varepsilon}}
S_r^\gamma(y).
\]
Since \((r,y)\mapsto S_r^\gamma(y)\) is continuous on the compact set
\([\delta,T]\times\overline{B(0,1)}\) and
\(S_r^\gamma(0)=0\) for every \(r\in[\delta,T]\), the supremum above
converges to zero as \(\varepsilon\downarrow0\). Therefore, passing to
the sequence \(\varepsilon_k\downarrow0\) in the preceding inequality
and using the uniform convergence
\(\mathbf A^{T,\gamma,\varepsilon_k}\to\mathbf A^{T,\gamma}\), we obtain
\[
\mathbf A_t^{T,\gamma}
\ge
\mathbf A_s^{T,\gamma},
\qquad
0\le s\le t\le T-\delta,
\qquad
\mathbb P_x^{T,\gamma}\text{-a.s.}
\]
Thus, \(\mathbf A^{T,\gamma}\) is increasing on \([0,T-\delta]\). Since \(\delta\in(0,T)\) was arbitrary and \(\mathbf A^{T,\gamma}\) is
continuous, it follows that \(\mathbf A^{T,\gamma}\) is increasing on
\([0,T]\).
\end{proof}

\section{Proof of Corollary~\ref{CorGroundStateVisitation}} \label{SubsectionCorGroundStateVisitation}

\begin{proof}
\noindent Part (i). By Proposition~\ref{PropSubMart}, the increasing component \(\mathbf A^{T,\gamma}\) of the Doob--Meyer decomposition of
\(\mathbf S^{T,\gamma}\) is constant during the excursions of \(X\)
away from the origin. Since \(\mathbf A_0^{T,\gamma}=0\) and
\(\mathbf A^{T,\gamma}\) is continuous, it follows that
\[
\mathbf A_{t\wedge\tau}^{T,\gamma}=0,
\qquad
t\in[0,T],
\quad
\mathbb P_x^{T,\gamma}\text{-a.s.}
\]
Thus, applying the optional stopping theorem to the martingale
\(\mathbf M^{T,\gamma}
=
\mathbf S^{T,\gamma}-\mathbf A^{T,\gamma}\)
at the bounded stopping time \(\tau\wedge T\), we obtain
\begin{align} \label{UsingOST}
S_T^\gamma(x)
&=
\mathbf S_0^{T,\gamma}
=
\mathbb E_x^{T,\gamma}
\left[
\mathbf S_0^{T,\gamma}
-
\mathbf A_0^{T,\gamma}
\right]
=
\mathbb E_x^{T,\gamma}
\left[
\mathbf S_{\tau\wedge T}^{T,\gamma}
-
\mathbf A_{\tau\wedge T}^{T,\gamma}
\right]
=
\mathbb E_x^{T,\gamma}
\left[
\mathbf S_{\tau\wedge T}^{T,\gamma}
\right].
\end{align}
We next identify the terminal random variable in the last expectation.
On the event \(\{\tau<T\}\), we have \(X_\tau=0\) and
\(T-\tau>0\), and hence, by the definition
\(S_t^\gamma(0)=0\) for \(t>0\),
\[
\mathbf S_{\tau\wedge T}^{T,\gamma}
=
S_{T-\tau}^\gamma(X_\tau)
=
S_{T-\tau}^\gamma(0)
=
0.
\]
On the other hand, on the event \(\{\tau>T\}\), we have
\(\tau=\infty\), so that \(\tau\wedge T=T\) and \(X_T\neq0\).
Therefore, using \(S_0^\gamma(y)=1\) for \(y\neq0\), we have $\mathbf S_{\tau\wedge T}^{T,\gamma}
=
S_0^\gamma(X_T)
=
1$. Moreover, since \(X_T\) has transition density
\(q_T^\gamma(x,\cdot)\) with respect to Lebesgue measure, we have $\mathbb P_x^{T,\gamma}[\tau=T]
\le
\mathbb P_x^{T,\gamma}[X_T=0]
=
0$. Consequently,
\begin{align}\label{1OComp}
\mathbf S_{\tau\wedge T}^{T,\gamma}
=
\mathbf 1_{\{\tau>T\}},
\qquad
\mathbb P_x^{T,\gamma}\text{-a.s.}
\end{align}
Substituting~\eqref{1OComp} into~\eqref{UsingOST} gives the first equality below,
\begin{align} \label{PartOneFinal}
\mathbb P_x^{T,\gamma}[\tau>T]
=
S_T^\gamma(x)
=
e^{-\lambda_\gamma T}
\frac{(g_T * \psi_\gamma)(x)}{\psi_\gamma(x)}
=
\frac{
K_{d/2-1}\!\left(
\sqrt{2\lambda_\gamma}\,|x|,
\lambda_\gamma T
\right)
}{
K_{d/2-1}\!\left(
\sqrt{2\lambda_\gamma}\,|x|
\right)
},
\end{align}
where the second equality follows from~\eqref{Sfunction}, and the final equality follows from the expression for \(\psi_\gamma\) in~\eqref{EigenFunctionEigenValueUnified} together with Lemma~\ref{LemmaGaussianGroundStateConvolution}. \vspace{.2cm}

\noindent Part (ii). Fix \(t\in(0,T]\). Since
\(\mathbb P_x^{T,\gamma}\) is the restriction to \([0,T]\) of the
corresponding infinite-horizon ground-state diffusion law, its
restriction to \(\mathcal F_t^{T,x}\) coincides with the finite-horizon law \(\mathbb P_x^{t,\gamma}\), and hence we have the first equality below.
\begin{align}
\mathbb P_x^{T,\gamma}[\tau>t]
&=
\mathbb P_x^{t,\gamma}[\tau>t]
=
S_t^\gamma(x)
=
\frac{
K_{d/2-1}\!\left(
\sqrt{2\lambda_\gamma}\,|x|,
\lambda_\gamma t
\right)
}{
K_{d/2-1}\!\left(
\sqrt{2\lambda_\gamma}\,|x|
\right)
}  \nonumber
\end{align}
The second equality follows from~\eqref{PartOneFinal}. Consequently, we have
\begin{align}
\frac{
\mathbb P_x^{T,\gamma}
\left[
\tau\in dt
\,\middle|\,
\tau\le T
\right]
}{dt}
=
\frac{d}{dt}
\mathbb P_x^{T,\gamma}
\left[
\tau\le t
\,\middle|\,
\tau\le T
\right]
=
\frac{d}{dt}
\left(
\frac{
1-S_t^\gamma(x)
}{
1-S_T^\gamma(x)
}
\right)
=
\frac{-
\partial_t S_t^\gamma(x)
}{
1-S_T^\gamma(x)
}.\nonumber
\end{align}
Also, by differentiating the defining integral~\eqref{DefIncompleteBessel} of
the incomplete modified Bessel function, we obtain
\begin{align}
-\frac{d}{dt}
K_{d/2-1}\!\left(
\sqrt{2\lambda_\gamma}\,|x|,
\lambda_\gamma t
\right)
=
\left(
\frac{|x|}{\sqrt{2\lambda_\gamma}}
\right)^{d/2-1}
\frac{1}{2t^{d/2}}
e^{-\lambda_\gamma t-\frac{|x|^2}{2t}}\,. \nonumber
\end{align}
Combining the above observations, we obtain
\begin{align}
\frac{
\mathbb P_x^{T,\gamma}
\left[
\tau\in dt
\,\middle|\,
\tau\le T
\right]
}{dt}
=
\frac{
\left(
\frac{|x|}{\sqrt{2\lambda_\gamma}}
\right)^{d/2-1}
}{
K_{d/2-1}\!\left(
\sqrt{2\lambda_\gamma}\,|x|
\right)
-
K_{d/2-1}\!\left(
\sqrt{2\lambda_\gamma}\,|x|,
\lambda_\gamma T
\right)
}
\frac{1}{2t^{d/2}}
e^{-\lambda_\gamma t-\frac{|x|^2}{2t}} \, . \nonumber
\end{align}
\end{proof}

\end{appendix}

\end{document}